\documentclass[10pt]{amsart}
\usepackage{amsthm, amsfonts, amssymb, amsmath, amscd}
\usepackage[all]{xy}
\usepackage{lmodern}
\usepackage[margin=3.5cm]{geometry}

\newcommand{\C}{\mathbb{C}}
\newcommand{\R}{\mathbb{R}}
\newcommand{\Q}{\mathbb{Q}}
\newcommand{\Z}{\mathbb{Z}}

\newcommand{\GL}{\mathrm{GL}}

\newcommand{\G}{\Gamma}
\newcommand{\F}{\mathbb{F}}

\newcommand{\Comp}{\mathcal{K}}
\newcommand{\LL}{\mathcal{L}}
\newcommand{\hi}{C(\bd \G)\rtimes_r \G}

\newcommand{\Hilm}{\mathcal{E}}
\newcommand{\Mult}{\mathcal{M}}
\newcommand{\Bound}{\LL}
\newcommand{\Calkin}{\mathcal{Q}}

\newcommand{\Poincare}{\mathcal{P}}
\newcommand{\K}{\mathrm{K}}
\newcommand{\KK}{\mathrm{KK}}
\newcommand{\EG}{\mathcal{E}\G}
\newcommand{\PD}{\textup{PD}}

\newcommand{\RKK}{\textup{RKK}}
\newcommand{\ev}{\textup{ev}}

\newcommand{\cM}{\tilde{M}}

\newcommand{\ez}{o_{z_0}}
\newcommand{\ezx}{\mathrm{i}_{z_0, X}}
\newcommand{\inflate}{\textup{inflate}}
\newcommand{\Dirac}{\mathcal{L}}
\newcommand{\fund}{[\widehat{\G\ltimes X}]}

\newcommand{\dol}{\overline{\partial}}
\newcommand{\Index}{\mathrm{Index}}
\newcommand{\FM}{\mathcal{F}_d}

\newcommand{\T}{\mathbb{T}}
\newcommand{\pnt}{\cdot}

\newcommand{\cZ}{\overline{Z}}
\newcommand{\bZ}{\partial Z}

\newcommand{\D}{\mathbb{D}} 
\newcommand{\Hyp}{\mathbb{H}} 

\newcommand{\Grd}{\mathcal{G}}

\newcommand{\bd}{\partial}

\newcommand{\pr}{\textup{pr}}

\newcommand{\id}{\mathrm{id}}

\numberwithin{equation}{section}

\usepackage{microtype}
\numberwithin{equation}{section}
\theoremstyle{theorem}
\newtheorem{theorem}[equation]{Theorem}
\newtheorem{lemma}[equation]{Lemma}
\newtheorem{proposition}[equation]{Proposition}
\newtheorem{corollary}[equation]{Corollary}
\theoremstyle{definition}
\newtheorem{definition}[equation]{Definition}

\theoremstyle{remark}
\newtheorem{remark}[equation]{Remark}
\newtheorem{example}[equation]{Example}

\begin{document}
\title[The class of a fibre in Noncommutative Geometry ]{The class of a fibre in Noncommutative Geometry}


\author{Heath Emerson}
\email{hemerson@math.uvic.ca}
\address{Department of Mathematics and Statistics\\
  University of Victoria\\
  PO BOX 3045 STN CSC\\
  Victoria, B.C.\\
  Canada V8W 3P4}


\date{\today}

\thanks{This research was supported by an NSERC Discovery grant. }

\begin{abstract}
This note addresses the K-homology of a C*-algebra 
crossed product of a discrete group acting 
smoothly on a manifold, with the goal of better understanding its noncommutative geometry. 
The
 Baum-Connes apparatus is the main tool. 
Examples suggest that the correct notion of the `Dirac class' of such a noncommutative 
space is the image under the equivalence determined by Baum-Connes of the fibre of the 
canonical fibration of the Borel space associated to the action, and a smooth model for the 
classifying space of the group. We give a systematic study of such fibre, or `Dirac classes,'
with applications to the construction of interesting spectral triples, and computation of their 
K-theory functionals, and we prove in particular that both the well-known deformation of the 
Dolbeault operator on the noncommutative torus, and the class of the boundary extension of 
a hyperbolic group, are both Dirac classes in this sense and therefore can be treated 
topologically in the same way. 

\end{abstract}

\maketitle

\tableofcontents

\section{Introduction}

The purpose of this article is to use the Baum-Connes apparatus to shed some 
light on the noncommutative geometry of some examples of C*-algebras that
probably deserve to be thought of as
 `noncommutative manifolds,' since they are 
canonically \(\KK\)-equivalent to classical manifolds.

We do this by fixing a definition of a class in the \(\K\)-homology of a crossed-product 
\(C_0(X)\rtimes \G\) of a smooth action by a discrete group 
 which we call the \emph{Dirac class}
of the action, and which is determined by the \(\K\)-homology class of a fibre 
in the natural fibration \( p\colon E\G\times_\G X \to B\G\) and the Dirac map involved
in the Baum-Connes apparatus.  This makes the Dirac class dependent on not only 
the action, but on aspects of the group itself including, in a certain sense, its
 coarse geometry. Our class differs from the transverse Dirac class studied by 
 A. Connes and others, which is, roughly speaking,
  invariant under the whole diffeomorphism
 group of the manifold, and doesn't really involve the group as such. 

We use this set-up to prove that 
 the boundary extension of a classical hyperbolic group acting on its sphere at 
 infinity, 
 and the deformed Dolbeault spectral triple over the irrational rotation algebra 
 of Connes, 
are, at the level of \(\K\)-homology, instances of the 
same same construction: they are each Dirac classes for the respective 
 actions. 
 
 We also deduce an index theorem for Dirac classes, which computes the 
 \(\K\)-theory functional determined by a Dirac class, in terms of 
 topological data (the intersection index of a Baum-Douglas  
 cycle with a fibre.) When specialized to either 
 the boundary 
 extension of a hyperbolic group, where it computes the boundary map on 
 \(\K\)-theory, or the irrational rotation situation, the resulting index formulas 
seem quite promising.

We now explain of all of this in more detail.

The Baum-Connes conjecture \cite{BCH} seeks to reduce the analytic problem of computing 
the \(\K\)-theory groups of a crossed-product \(C_0(X)\rtimes \G\) to topology (by and 
large we work with the max crossed-product in this paper, as it is functorial, and 
since most of the actions we consider specifically are amenable.) In the work of 
Meyer and Nest, following a tradition initiated by Kasparov, Lusztig, Higson
and others, 
it is shown that to \(\G\) one can associate a proper \(\G\)-C*-algebra 
\(\mathcal{P}\) and a Kasparov morphism \(D \in \KK_*^\G(\mathcal{P}, \C)\) (the 
`Dirac morphism') with the property that the forgetful map 
\( \KK^\G \to \KK^H\) maps \(D\) to an equivalence, for any finite subgroup \(H\) of \(\G\). 
This condition determines \(D\). External product in \(\KK^\G\) gives a map 
\begin{equation}
\label{equation:intro_dirac_map}
 \KK^\G_*(A, B) \to \KK^\G_*(\mathcal{P}\otimes A, B)
 \end{equation}
for any \(\G\)-C*-algebras \(A,B\), and the co-domain of this map is of a purely 
topological nature because it is isomorphic to \(\RKK^\G_*(\EG; A, B)\) by 
an important Poincar\'e duality theorem of Kasparov \cite{Kas}.

The C*-algebra \(\mathcal{P}\) and the morphism \(D\)  are not always easy to 
represent concretely. In this paper, we assume that \(\EG\) can be modelled by a 
smooth, co-compact, equivariantly \(\K\)-oriented manifold \(Z\). This covers the 
case of \(\Z^d\) actions, and most actions of discrete subgroups of semisimple 
Lie groups. The hypothesis implies that we can take \(\mathcal{P}:= C_0(Z)\) and 
\(D := [Z] \in \KK_{-d}^\G(C_0(Z), \C)\) the class of the \(\G\)-equivariant 
Dirac operator on \(Z\). 
The assumptions are satisfied by, for example, 
fundamental groups \(\G = \pi_1(M)\) of compact, oriented, aspherical, 
spin\(^c\)-manifolds, taking \(Z := \tilde{M}\) with its lifted 
\(\G\)-equivariant \(\K\)-orientation. 

The geometric content of the Dirac map starts to appear if one 
puts  \(A =B = \C\) and \(\G = \Z^d\). The domain of the Dirac map is 
\(\KK_*^{\Z^d}(\C, \C) \cong \KK_*(C(\widehat{T^d}), \C) \) and the 
co-domain is \(\KK_*^{\Z^d}(C_0(\R^d), \C) \cong \KK_*(C_0(\R^d/\Z^d), \C) 
= \KK_*(C(T^d), \C)\) where \(T^d := \R^d/\Z^d\) and \(\widehat{T^d}\) is 
by definition \(\widehat{\Z^d}\), the `dual' torus. 
 The Dirac map therefore is a map 
\begin{equation}
 \K_*(\widehat{T^d}) \to \K_{*+d}(T^d),
 \end{equation}
and, it is not that difficult to compute that it is precisely the well-known 
\emph{Fourier-Muai transform}, 
implemented by composing cohomology cycles 
with the smooth correspondence
\[ T^d \leftarrow (T^d\times \widehat{T^d}, \beta) \to \widehat{T^d},\]
from \(T^d\) to the dual torus \(\widehat{T^d}\), 
where \(\beta\) is the Mischenko-Poincar\'e element, and the 
maps are the projection maps. Furthermore, as we show, it has 
an interesting effect on Baum-Douglas \(\K\)-homology, for it 
interchanges (the \(\K\)-homology class of) 
 a \(j\)-dimensional subtorus in \(\widehat{T^d}\) to (the class of a)
certain canonical \(d-j\)-dimensional `dual' torus in \(T^d\). It was 
this observation that first made the author want to study the Dirac 
map more closely, and especially for actions.

If \(A = C_0(X)\), for a \(\G\)-space \(X\), and \(B = \C\),  the Dirac map looks like 

\[[Z]\otimes_\C \cdot :  \KK^\G_*(C_0(X), \C) \to \KK^\G_{*-d} (C_0(Z\times X), \C)\]
and the domain is the \(\K\)-homology of the crossed-product \(C_0(X)\rtimes \G\), 
while the co-domain, if \(\G\) is torsion-free, is naturally isomorphic to 
\(\K_{d- *  }(Z\times_\G X)\) -- the \(\K\)-homology of the `Borel space' 
\(Z\times_\G X\), which fibres over \(\G\backslash Z \cong B\G\), with fibre \(X\). 
The Dirac map shifts degrees by \(-d\). 

We define a \emph{Dirac class} for groups with torsion, but in the torsion-free case, 
the Dirac class is any class in \(\K^{d-n} (C_0(X)\rtimes \G)\) mapped by the Dirac 
map to the Baum-Douglas \(\K\)-homology class of the fibre \(X\): a spin\(^c\)-manifold 
mapping (properly) to \(Z\times_\G X\), by including it as a fibre. 

This definition does not guarantee that a Dirac class exists, nor that it is unique, because 
the localization map is neither onto nor \(1\)-\(1\) in general. However, if \(\G\) has a dual-Dirac 
morphism, then the Dirac map can be split, yielding a existence result about Dirac 
classes (although still not uniqueness). It is this method which, when applied to isometric 
actions of nice discrete groups \(\G\) (like \(\Z^d\)), leads to spectral triple representations of the Dirac class by 
spectral triples over \(C_0(X)\rtimes \G\), whose general format, of the
Schr\"odinger kind, \(D+\delta\), with \(\delta\) an operator on the group \(\G\), 
\(D\) the Dirac operator on \(X\), are somewhat similar to the ones appearing
 \cite{Haw} (for \(\G = \Z^d\)).  When \(\Z\) acts by irrational 
rotation on the circle, the Dirac class is represented by
 the famous spectral triple (the deformed Dolbeault 
operator \(\dol_\theta\) over \(A_\theta\) first defined by 
A. Connes (see \cite{Connes}). 

Actions of discrete (co-compact) 
groups of M\"obius transformations \(\G\subset \mathrm{SL}_2(\R)\) on the 
circle \(\T\) are smooth actions preserving a \(\K\)-orientation; they are 
special cases of broader classes of hyperbolic groups acting on their boundaries. 
These examples cannot be treated like isometric actions as in the previous 
paragraph, by \(\Z^d\): one cannot form an external product 
of the type \(D+\delta\) 
as in the previous paragraph because there is no \(\G\)-invariant Dirac operator 
\(D\) on the circle (because there is no \(\G\)-invariant probability measure) with which 
one can take external product with. 

It turns out that the 
the Dirac class of such an action can be represented by completely orthogonal 
methods: probabilistic ones concerning the action of \(\G\) on probability 
measures on its boundary (see \cite{EN}). These imply that
 the regular representation of \(C(\partial \G)\rtimes \G\) on 
\(L^2\bigl( \G, L^2(\partial \G , \mu)\bigr)\), together with the orthogonal 
projection \(P_{l^2(\G)}\) onto the subspace \(l^2(\G)\) of functions constant 
on the boundary, make up a Fredholm module representing the 
\emph{boundary 
extension class}:  the 
class in \(\KK^1(C(\partial \G) \rtimes \G, \C)\) of the boundary extension 
\begin{equation}
\label{sdfsdfsdfs}
 0 \to C_0(\G)\rtimes \G \to C(\overline{\G})\rtimes \G \to C(\partial \G)\rtimes \G\to 0
 \end{equation}
with \(\overline{\G}\) the compactification of \(\G\) obtained by mapping it in 
as an orbit in the disk \(\D\), and compactifying in the closed disk \(\overline{\D}\). 
The Fredholm module described above has finite summability 
the Hausdorff dimension of \( (\partial \G, \mu)\). This appears a step 
forward in understand the noncommutative geometry of these (Type III) 
examples, but it is only, in a sense, the noncommutative \emph{conformal} geometry that 
is being understood. Whether one can integrate a noncommutative notion of 
\emph{length} (or distance) into these examples remains unknown.

We prove here only that the Dirac class is the boundary extension class. The fact 
of being a Dirac class has, in any case, various topological consequences: it allows 
a completely topological description of its pairing with \(\K\)-theory. 

In general, what the Dirac class detects, topologically, is a certain intersection number. 
If one has a 
Baum-Douglas cycle (or cocyle) for \(Z\times_\G X\), a higher index construction
produces a \(\K\)-theory class for \(C_0(X)\rtimes \G\) which pairs with the 
Dirac class to give a certain analytic index. The index theorem is that this 
analytic index is the topological intersection number of the Baum-Douglas 
cycle (or cocyle) with the fibre \(X\subset Z\times_\G X\). The irrational 
rotation algebra is already an interesting example. In Connes' work it is 
shown that \(\dol_\theta\) can be extended, by constructing a connection, and 
so on,  to act on 
sections of various `noncommutative vector bundles' over \(A_\theta\) -- that is, 
f.g.p. modules \(\mathcal{E}_{p,q}\). These
bundles are parameterized by pairs of relatively prime integers and are higher 
indices of the \(1\)-dimensional Baum-Douglas cocycles for the ordinary 
torus given by loops \(L_{p,q}\); the content of our index theorem here is that 
the index of the operator \(D_\theta\) extended to act on \(L^2(\mathcal{E}_{p,q})\) 
is the topological 
intersection number of the loop with the standard meridian loop of the torus (that is, 
\(q\), in this parameterization.) 

The boundary extension class of a hyperbolic group, 
due to work of the author and Ralf Meyer, is torsion of 
order \(\chi (\G)\) if 
 \(\G\) is torsion-free, and \(\chi (\G) \not= 0\), and it is non-torsion, nonzero, if 
 \(\chi (\G) = 0\). The intersection index computes the boundary map 
 \[ \delta \colon \K_1( C(\partial \G)\rtimes \G) \to \K_0(C_0(\G)\rtimes \G) = \K_0(\C) = \Z,\]
and one of the author's initial interests was in finding \(\K\)-theory classes 
in \(\K_1(C(\partial \G)\rtimes \G)\) 
in the case \(\chi (\G) = 0\)  (\emph{e.g.} for Kleinian groups 
with nonzero pairing with the boundary extension. We 
give in fact a direct geometric construction of such a \(\K\)-theory class, 
based on a non-vanishing vector field on \(Z\backslash \G\), using the 
intersection 
 index formula to compute it's image under \(\delta\). We also use our framework 
 and an argument with \(\Z/k\)-manifolds to explain the torsion of the boundary 
 extension class, in general.

I would like to thank Paul Baum, for all his his boundless enthusiasm for 
the subject of Dirac operators and \(\K\)-homology has taught me.
I would also like to thank Nigel Higson, for 
several very pertinent remarks, and the referee, for the number of 
 suggestions, 
whose adoption has 
greatly improved the layout and content of this article. 

\section{The Dirac-localization map for K-oriented groups }
\label{section:bcloc}

If \(\G\) is a locally compact group acting smoothly on a smooth Riemannian manifold 
\(X\), then a \(\G\)-equivariant \(\K\)-orientation on \(X\) consists of
\begin{itemize}
\item[a)] A \(\G\)-invariant Riemannian metric on \(X\). 
\item[b)] A \(\G\)-equivariant complex vector bundle \(S\to X\) (the spinor bundle), equipped with a 
\(\G\)-invariant Hermitian metric, and, if \(n\) is even, a \(\G\)-invariant
 \(\Z/2\)-grading on \(S\). 
 \item[c)]  A \(\G\)-equivariant fibrewise irreducible representation of the Clifford algebra 
 bundle of \(X\), on \(S\) compatible with the \(\Z/2\)-gradings if relevant.  
 
\end{itemize}

Assuming 
that the action is proper, one can then construct a \(\G\)-invariant connection on \(X\) compatible with the Levi-Civita connection on \(TX\), and corresponding \(\G\)-equivariant 
Dirac operator on \(X\), providing a cycle and corresponding class in 
\( \KK_{-n}^\G (C_0(X), \C)\). The class is canonically associated to the \(\K\)-oriented) 
\(\G\)-manifold \(X\), and we will denote it \([X]\). We will call \([X]\) the 
\emph{transverse Dirac class} of \(X\), the reason for the word `transverse' and the 
exact definition given below. A good source for the construction of analytic 
Dirac cycles is \cite{BHS}. Another good source involving equivariant 
Dirac operators is the seminal paper \cite{BCH}.

If \(X\) is a smooth, but not 
necessarily proper \(\G\)-manifold, it is more difficult to directly construct a 
Dirac class of the above type, since \(\G\) may not even preserve any 
Riemannian metric on \(X\). But there are several ways of arguing that there 
still exists a \emph{class} in \(\KK_{-n}^\G(C_0(X), \C)\) playing the role 
of the Dirac class, even though a representative \emph{cycle} is more 
difficult to describe. 

For example, if \(\G = \Z\) is the integers, then 
one of the foundational results of \(\KK\)-theory is that there is a 
\(\KK\)-equivalence between \(C_0(X)\rtimes \Z\) and the C*-algebra of 
continuous functions on the mapping cylinder 
\(C_0(\R\times_\Z X)\), shifting degrees by \(+1\). With the appropriate 
orientation hypothesis on \(X\), the mapping cylinder is \(\K\)-orientable with 
associated (non-equivariant) Dirac class 
\[ [\R\times_\Z X]\in \KK_{-n-1} (C_0(\R\times_\Z X), \C)\]
and we can uniquely define a class \([X]\in \KK_{-n } (C_0(X)\rtimes \Z, \C)\) 
by the requirement that  
the \(\KK\)-equivalence 
\[\KK_{-n } (C_0(X)\rtimes \Z, \C) \to \KK_{-n-1} (C_0(\R\times_\Z X), \C)\]
alluded to above, 
maps \([X]\) to \([\R\times_\Z X]\).

Recently, these ideas are often understood in terms of the 
set-up of Meyer 
and Nest, which abstracts the earlier work of 
Connes, Baum, Kasparov, Higson and many others, and interprets
 the \(\KK\)-equivalence 
just discussed as a \emph{localization map}.

Localizing a category at a collection of morphisms inverts the 
morphisms. Meyer and Nest consider
 the category \(\KK^\G\) of \(\G\)-C*-algebras, where \(\G\) is a 
locally compact group, and localize it at the \emph{weak equivalences}, where 
\[ f\in \KK^\G(A, B)\]
is a weak equivalence if, for every compact subgroup \(H \subset \G\), the 
restriction map 
\[ \KK^\G \to \KK^H\]
maps \(f\) to an equivalence.  They show that to any \(\G\) can be associated 
a unique \(\G\)-C*-algebra \(\mathcal{P}\) and morphism 
\[ D\in \KK^\G(\mathcal{P}, \C)\]
such that the localization of \(\KK^\G\) at the weak equivalences has 
morphisms between objects \(A\) and \(B\) the elements of 
\( \KK^\G(\mathcal{P}\otimes A, B)\)
and the localization map from \(\KK^\G\) to its localization, identifies with 
the map on morphisms given by Kasparov product 
\[ \KK^\G(A, B) \xrightarrow{D\otimes_\C \, \cdot \,} \KK^\G(\mathcal{P}\otimes A, B).\]
In most applications, \(\mathcal{P}\) is a proper \(\G\)-C*-algebra. 
In this paper, we will be working with instances of \(G\) where 
\(\mathcal{P}\) is represented by a very specific cycle that of the 
Dirac operator. This means that the more important of our definitions 
(like of Dirac class) depend on this structure, and do not apparently 
make much sense for more general groups, although the map we are 
going to discuss is a special case of the more general Meyer-Nest localization map.

We are going to be working 
with discrete groups  
\(\G\) for which \(\mathcal{P}\) can be realized as a smooth, 
proper, \(\G\)-equivariantly \(\K\)-oriented \(\G\)-manifold \(Z\). For such 
\(Z\) we can directly construct an analytic cycle and Dirac class 
\[ [Z]\in \KK_{-d} (C_0(Z), \C)\]
which is equal to \(D\) if \(d\) is even, and a suspension of \(D\) otherwise. 
We will then be able to describe the localization map in very concrete 
geometric terms. 

The exact hypotheses on the discrete group \(\G\) we will be using 
are the following.

\begin{definition}
\label{definition:smooth_oriented_group}
A \emph{\(\K\)-orientation} on the discrete group \(\G\) will refer to
 a smooth, proper \(\G\)-equivariantly
\(\K\)-oriented co-compact 
\(\G\)-compact manifold \(Z\) which is \(H\)-equivariantly 
contractible for every compact subgroup \(H\) of \(\G\). 
We refer to the pair \( (\G, Z)\) as a smooth \(\K\)-oriented group.
We let 
\[ [Z]\in \KK_{-d} (C_0(Z), \C)\]
be the class of the associated \(\G\)-equivariant
 Dirac operator on \(Z\).  We call it the \emph{transverse Dirac class of \(Z\)}. 

\end{definition}

The contractibility assumption means that \(Z\) is a model for the classifying 
space \(\EG\) for proper actions of \(\G\), that \(Z\) can be identified with the 
localizing object \(\mathcal{P}\) of Meyer and Nest, and that \([Z] \) is the 
Dirac morphism. 

\begin{example}
\label{example:examples_of_smooth_oriented_groups}
Every compact group admits a smooth \(\K\)-orientation with \(Z\) a point.

The group \(\Z^d\) admits a smooth \(\K\)-orientation using \(Z:= \R^d\) with 
the smooth action of \(\R^d\) by translation; since \(\Z^d\) is a closed subgroup 
of \(\R^d\), \((\Z^d, \R^d)\) is a smooth \(\K\)-oriented group. 

Suppose that 
 \(M\) is a compact Riemann surface (a compact two-dimensional manifold 
equipped with a complex structure). Then it admits a canonical orientation. 
The universal cover \(Z:= \tilde{M}\) has a free and proper action of \(\G := 
\pi_1(M)\), and can be equipped with a \(\G\)-invariant metric and 
orientation (lifted 
from \(M\), \emph{i.e.} a complex structure) and metric of constant negative 
curvature, making it contractible, and more generally, \(H\)-equivariantly contractible 
for any compact group of isometries of \(Z\). 

The Riemannian manifold \(Z\) can of course be identified with
 the hyperbolic plane \(\Hyp^2\) with an appropriate 
proper, isometric action of \(\G\). 

Thus,  \( (\G, \Hyp^2)\) is a smooth, oriented group.

Similarly, any orientation-preserving co-compact 
discrete group of hyperbolic isometries of 
\(\Hyp^3\) admits the structure of a smooth \(\K\)-oriented group, 
since any compact oriented \(3\)-manifold 
also carries a \(\K\)-orientation. 

More generally, if \(\G\) is a torsion-free 
uniform lattice in a semi-simple 
Lie group with associated symmetric space \(Z = G/K\), then \(\K\)-orientability of 
\(\G\backslash X\) implies \(\G\)-equivariant \(\K\)-orientability of \( Z\), and 
hence \( (\G, Z)\) admits a canonical structure of a smooth oriented group in 
this case, from a \(\K\)-orientation on \(\G \backslash Z\).

This amounts to the well-known
 procedure of `lifting' a \(\K\)-orientation 
 under a covering map.

 \end{example}

\begin{definition}
\label{definition:Dirac_map}
Let \( (\G,Z)\) be a smooth oriented group. The \emph{Dirac-localization map}
is the map  
\[ \Dirac: \KK^\G_*(A, B) \xrightarrow{\otimes_\C [Z]} \KK^\G_*(C_0(Z)\otimes A, B)\]
induced by external product in \(\KK^\G\), 
with the class \([Z]\) of the transverse Dirac class for \(\G\) acting on \(Z\).
\end{definition}
We will generally just use the term `localization map.'

\begin{remark}
 Any model for the Dirac morphism of \cite{Meyer-Nest:BC} 
determines its own 
corresponding`localization map', as we have already discussed above, 
localization in this sense makes sense for general locally compact groups, 
without further assumptions on their classifying spaces. But in 
this article, we are 
interested in doing computations. These computations are most easily 
done when the Dirac morphism has the simple geometric model that it has under 
our assumptions. For example, dropping the equivariant 
\(\K\)-orientation assumption on \(Z\) forces one to use a different 
co-domain for the localization map: \(C_0(Z)\) must be replaced by 
\(C_\tau (X)\), the algebra of sections of the Clifford algebra bundle of \(Z\), 
which is noncommutative, and, furthermore, \(\Z/2\)-graded, or 
by \(C_0(TZ)\), with \(TZ\) the tangent bundle, which is no longer 
\(\G\)-compact, and has various other disadvantages from a computational point of view. 
\end{remark}

\begin{remark}
If \(B\) is a trivial \(\G\)-C*-algebra, and \(A\) is an arbitrary \(\G\)-C*-algebra, 
recalling that \(\G\) is \emph{discrete}, 
there is a completely canonical isomorphism 
\[ \KK^\G_*(A, B)\cong \KK_*(A\rtimes \G, B),\]
due to the standard bijection between *-homomorphisms with domain 
a crossed-product \(A\rtimes G\), and covariant pairs. 

Taking this into account, the Dirac-localization map for \(B=\C\) 
can be considered as a map 
\begin{multline*}
\label{equation:Dirac_map}
 \K^*(A\rtimes \G) := \KK_*(A\rtimes \G, \C) \cong \KK^\G_{*}(A, \C) \rightarrow 
\KK^\G_{*-d}(C_0(Z)\otimes A, \C)\\ \cong \K^{*-n}( C_0(Z, A)\rtimes \G)
\end{multline*}
for any \(\G\)-C*-algebra \(A\). 
If \(\G\) is torsion-free, \(A = C_0(X)\), some \(\G\)-space \(X\), then 
 the target of Dirac-localization is \(\KK_*(C_0(X\times_\G Z), \C) = \K_{-*}(Z\times_\G X)\). 
 Using 
 a standard Morita equivalence one can identify this with
  the \(\K\)-homology 
of the Borel space \(X\times_\G Z\), which 
fibres over \(\G\backslash Z\) under the second projection map, with fibre \(X\).

\end{remark}

For a compact group, the Dirac-localization map is the identity map.

\subsection{Factorization of the localization map}
We close this section with a review of an important factorization of the 
localization map. 

Let \((\G, Z)\) be a smooth oriented group. 
Then a result going back to Kasparov 
shows that there is a Poincar\'e duality isomorphism 
\begin{equation}
\label{equation:PD}
\KK^\G_*(C_0(Z)\otimes A, B) \cong  \RKK^\G_*(Z; A, B)
\end{equation}
shifting degrees by \(d\). See \cite{EM:Dualities}. 
The group  \(\RKK^\G_*(Z; A, B)\), explained in \cite{Kas}, 
 is identical, by the definitions, 
to the groupoid-equivariant group 
\(\KK_{*+d}^{\Grd_\G}(C_0(Z)\otimes A, 
C_0(Z)\otimes B \bigr)\) defined by LeGall; this point of view is 
convenient, because \(\KK^\Grd\) is equivalent to the 
category of \(\Grd\)-equivariant correspondences, by \cite{EM:Geometric_KK}, 
for proper groupoids \(\Grd\).

The way 
the Poincar\'e duality map 
\[\PD:  \RKK^\G_*(Z; A, B) \to \KK^\G_{*-d} (C_0(Z)\otimes A, B)\]
is defined is as follows. It is 
the composition of the map 
\[\RKK^\G_*(Z; A, B) \to \KK^\G_* (C_0(Z)\otimes A , C_0(Z)\otimes B)\]
which forgets the \(\G\ltimes Z\) equivariance on a cycle, remembering only 
\(\G\)-equivariance,
which we denote by \(f\mapsto \overline{f}\), and the map 
\begin{equation}
\KK^\G_* (C_0(Z)\otimes A , C_0(Z)\otimes B)
\\  \xrightarrow{\otimes_{C_0(Z)} [Z] } 
 \KK^\G_{*-d}(C_0(Z)\otimes A, B).
 \end{equation}
of composition with the Dirac class \([Z]\in \KK_{-d}^\G(C_0(Z), \C)\). 

The main ingredient of the Poincar\'e duality 
isomorphism is thus the class 
\([Z]\in \KK^\G_{-d} (C_0(Z), \C)\) of the \(\G\)-equivariant Dirac 
operator on \(Z\) -- the transverse Dirac class, in our terminology, for \(\G\) acting 
on \(Z\).

The map \(\PD^{-1}\) inverse to \(\PD\) described above, 
is defined using the map, which we call the \emph{inflation map}, of 
Kasparov: 
\begin{equation}
\label{equation:inflate}
\inflate\colon \KK^\G_*(C_0(X), \C) \to 
 \RKK^\G_*(Z; C_0(X), \C).
 \end{equation}
In these terms, 
\begin{equation}
\label{equation:pd1}
\PD^{-1}(f)  = \Theta \otimes_{C_0(Z)} \inflate (f),\end{equation}
where
\[ \Theta \in \RKK^\G_{d}(Z; \C, C_0(Z)) \cong \KK^{\Grd_\G}_{+d}(C_0(Z), C_0(Z\times Z)),\]
is the 
class of the \(\Grd_\G\)-equivariant correspondence 
\[ Z \xleftarrow{\id} Z \xrightarrow{\delta} Z\times Z,\]
where the momentum map for the \(\Grd\)-space \(Z\times Z\) is 
in the first variable, and 
\(\delta_Z\colon Z \to Z\times Z\) be the diagonal map.

Correspondences and their associated
 \(\KK\)-morphisms are discussed 
in Section \ref{subsec:topcorr}. 
 A cycle
representing \(\Theta \in 
\RKK^\G_{+d}(Z; \C, C_0(Z) )\) is built by constructing a family 
of Bott cycles \(\Theta_z\) at 
\(z\in Z\). Such a Bott cycle is defined as follows. 
In each small Riemannian ball \(B_z\) around \(z\), use the 
Clifford multiplication and a vector vector on \(B_z\) pointing towards \(z\), to 
 construct a multiplier of the module of spinors over the 
ball and \(\KK_{+d}(\C, C_0(B_z))\) cycle and then a \(\KK_{+d}(\C, C_0(Z))\) 
cycle by the map induced by the open inclusion \(B_z\subset Z\). 

The required cycle for 
\(\RKK^\G_{+d}(Z;  \C, C_0(Z))\) is then 
built from considering \(\Theta\) as an operator on the 
Hilbert module of sections of the field. See \cite{Kas}.

See Section \ref{subsection:inflationstuff} for more on the inflation map 
and Poincar\'e duality.

\begin{proposition}
\label{proposition:factoring_Dirac}
Let \((\G, Z)\) be a smooth \(\K\)-oriented \(d\)-dimensional group. Then the 
localization map factors as 
\begin{equation}
\label{equation:factoring_the_dirac_map}
 \xymatrix{ \KK^\G_*(A,B) \ar[dr]_{\Dirac} \ar[r]^{\inflate} & \RKK^\G_*(Z; A,B)\ar[d]^{\PD}\\ & 
 \KK^\G_{*+d}(C_0(Z)\otimes A, B)}
 \end{equation}
where \(\PD\) is Poincar\'e duality. 
\end{proposition}

For the proof, see Theorem 4.34 of \cite{EM:Dualities}.

Since \(\PD\) is always an isomorphism, the inflation map and the Dirac 
map are equivalent; thus one is an isomorphism if and only if the other is.

If \(\G\) is torsion-free, \( (\G, Z)\) a smooth \(\K\)-oriented group, 
then \(\G\) acts freely on \(Z\), and the commutative 
diagram \eqref{equation:factoring_the_dirac_map} becomes
\begin{equation}
\label{equation:factoring_the_dirac_map_torsion_free_special}
 \xymatrix{ \K^*(C^*\G) \ar[dr]_{\Dirac} \ar[r]^{\inflate} & \K^{*} (\G \backslash Z) \ar[d]^{\PD}\\ & 
 \K_{*-d}(\G \backslash Z)}.
 \end{equation}
with \(\PD\) Poincar\'e duality for the \(\K\)-oriented manifold \(\G\backslash Z\).

The Dirac-localization map for torsion-free \(\K\)-oriented groups is therefore
is between the \(\K\)-homology of the 
C*-algebra \(C^*(\G)\), and the \(\K\)-homology of the classifying space 
\(\G\backslash Z \cong B\G\).

\section{The localization map for free abelian groups}

In this section we describe Dirac-localization for free abelian groups \(\Z^d\) 
acting on points: that is, we describe the localization map 
\[ \Dirac \colon \KK^{\Z^d}_*(\C, \C) \to \KK^{\Z^d}_{*-d}(C_0(\R^d), \C).\]
in geometric terms -- it turns out to be essentially a direct 
\(\K\)-theory version of the Fourier-Mukai 
transform of algebraic geometry.

 \subsection{Topological correspondences}
\label{subsec:topcorr}
A (sometimes called `topological'\footnote{just to distinguish them from C*-correspondences, for example}) correspondence, is a specification of a certain set of 
geometric data which produces a class in \(\KK\). The concept is due 
to \cite{Connes-Skandalis}. A closely related concept plays an important role in 
algebraic geometry, in which context, correspondences 
they are sometimes referred to 
as Fourier-Mukai transforms; they are morphisms in a suitable category between 
projective varieties. 

 For purposes of \(\K\)-theory, if 
  \(X\) and \(Y\) are smooth manifolds, a 
correspondence from \(X\) to \(Y\) 
is a pair of maps and a \(\K\)-theory class, 
usually depicted by a diagram
\begin{equation}
\label{equation:correpsondence}
X \xleftarrow{b} (M, \xi) \xrightarrow{f} Y,\end{equation}
where \(f\) is a smooth \(\K\)-oriented normally non-singular map; \(b\) is an ordinary smooth map
(not necessarily proper), and 
the class \(\xi \) lies in the representable \(\K\)-theory of \(M\) with \(b\)-compact 
support. If \(b\) is proper, this is just the ordinary \(\K\)-theory of \(M\). 
The theory of 
correspondences is due to Connes and Skandalis \cite{Connes-Skandalis}. 
Connes and Skandalis associate to such a correspondence a certain
 \emph{analytic} cycle for \(\KK\)-theory in the following way. We will assume that 
 the \(\K\)-theory class is in dimension zero and is specified by a 
 vector bundle \(E \to M\) over \(M\). The most delicate part of the construction involves 
 map \(f\), to which one wishes to associate an analytically defined morphism 
 \[ f_{\textup{an}}!\in \KK_{\dim Y - \dim M} (C_0(M), C_0(Y)).\]
 
  If \(f\) is a \emph{submersion}, then \(f\) gives 
 rise to a bundle of smooth manifolds over \(Y\) with fibre the fibres of \(f\), 
the 
 \(\K\)-orientation assumption on \(f\) implies a bundle of \(\K\)-orientations on the fibres, 
 and 
 by a well-known procedure one can then construct from this data a bundle of 
 Dirac operators on the fibres of \(f\). This gives
  a cycle and analytically defined 
 class \(D_f\) we we set 
 \[ f_{\textup{an}}! := D_f\in \KK_{\dim Y - \dim M} (C_0(M), C_0(Y))\] 
 If \(E\) is a vector bundle over \(M\), \(D_f\) can be twisted by the vector bundle, 
 giving a twisted version \(D_{f, E}\) of \(D_f\) in the same group. Finally, 
 \(b\) induces a map \(b^*\colon \KK_*(C_0(M) , C_0(Y))\to \KK_*(C_0(X), C_0(Y))\), and 
 now the analytic class in 
 \(\KK\) defined by the correspondence is by definition 
 \[ b^*(D_{f,E})\in \KK_{\dim Y - \dim M} (C_0(M), C_0(Y))\].

  If \(f\) is merely assumed a smooth map, then it can be factored into a submersion and 
  an immersion, and an analytically defined morphism 
  \[ f_{\textup{an}}!\in \KK_{\dim Y - \dim M}  (C_0(M), C_0(Y))  \]
  is defined by composing the two Kasparov morphisms obtained from 
  the factorization: the submersion determines an element of \(\KK\) as just described; 
and a \(\K\)-oriented immersion 
  \( f\colon M \to Y\) defines a \(\KK\)-morphism in the following way. 
  The immersion 
  has a \(\K\)-oriented normal bundle \(\nu\), with a tubular neighbourhood 
  embedding 
  \[ \varphi \colon \nu \to C_0(Y),\]
 onto an open subset of \(Y\). Combining the Thom isomorphism class 
 \[ \xi_\nu \in \KK_{\dim Y -\dim M} (C_0(M), C_0(\nu))\]
  and the open embedding 
  \[ \varphi ! \in \KK_0(C_0(\nu), C_0(Y))\]
  gives a purely topologically-defined morphism in 
  \(\KK_{\dim M - \dim Y} (C_0(M), C_0(Y))\)
  associated to the immersion. 
  
  There is another way of building a natural \(\KK\)-element 
  from the data \eqref{equation:correpsondence}. Part of the 
  recipe above was in fact purely topological: if the map \(f\) was 
  an immersion, then \(f!\) is defined purely in terms of Thom modification 
  in \(\K\)-theory (by the normal bundle of \(f\).) The idea behind the 
  topological index (of an elliptic operator) of Atiyah and Singer was 
  as follows. Since we are only really concerned about submersions, let 
  \[ f\colon M \to Y,\]
  with \(\K\)-oriented fibres. We have defined \(f_{\textup{an}}!\) above 
  using the bundle of Dirac operators along the fibres of \(f\). Instead, 
  let 
 \( \zeta \colon M \to \R^n\) be a smooth embedding, for some \(n\). Then 
 \[ M \to Y\times \R^n, \;\;\; \tilde{f} (x) = \bigl( \zeta (x), f(x)\bigr)\]
 is a smooth embedding. We obtain by the procedure above an element 
 \[ \tilde{f}!\in \KK_{\dim Y -\dim M+n}(C_0(M), C_0(Y\times \R^n)\bigr) \cong  
 \KK_{\dim Y -\dim M}(C_0(M), C_0(Y)\bigr).\]
  where the second equality is by the Bott Periodicity 
  \(\KK\)-equivalence \(C_0(\R^n) \cong \C\), 
  which shifts degrees by \(-n\). If \(f !\) is defined as \(\tilde{f}\) then
   the Atiyah-Singer Index Theorem in \(\KK\)-theory is the 
  statement that 
  \[ f_{\textup{an}}!=f!\]
  for any smooth \(\K\)-oriented map \(f\).

The most important feature of correspondences is that they can be composed
in a purely geometric manner: the composition 
\[ X \xleftarrow{b} M \xrightarrow{f} Y \xleftarrow{b'} M' \xrightarrow{f'} Z\]
if the 
maps \(f\) and \(b'\) are 
transverse, is represented by the correspondence 
\[ X \leftarrow M\times_Y M' \rightarrow Z,\]
with \(M\times_Y M'\) having its canonical 
smooth manifold structure, and where the map \(M\times_Y M'  \to Z \) (the composition of the projection \(M\times_Y M\to M'\) and the map 
\(f'\colon M'\to Z\)) carries a certain \(\K\)-orientation induced by the \(\K\)-orientations on
 \(f\) and \(f'\). The left map \(M\times_Y M'\to X\) is similarly the composition of the 
 first coordinate projection and the map \( b\). It is easy to integrate the \(\K\)-theory 
 data into this recipe.

 This efficient recipe of composing correspondences (\(\KK\)-elements) will be 
 used later in this article.

The \emph{dimension} of the correspondence \eqref{equation:correpsondence} 
is \( \dim Y - \dim M + \mathrm{deg}\, \xi\). If \(\xi\) is the class of the 
trivial line bundle, so that all the information in the correspondence lies in the
maps and the \(\K\)-orientations, 
then the dimension is \(\dim Y - \dim M\). With 
this notion of dimension, correspondences composed by the transversality recipe 
described above, do so 
additively with respect to dimension.

\subsection{Localization and Fourier-Mukai duality}
We now describe a correspondence which encodes the localization map. The 
main ingredient will be the \(\K\)-theory class 
 \(\Poincare_d \in \K^0(T^d\times \widehat{T^d})\) of the 
 Poincar\'e bundle, defined to be the 
  class of the f.g.p. module over \(C(T^d\times \widehat{T^d})\) consisting of all 
 continuous functions 
 \(f \) on \(\R^d\times \widehat{T^d}\) such that 
 \begin{equation}
 \label{equation:sdfsdfsdfgjdkfjg}
  f( x+v, \chi) = \chi (v) f(x),\;\;\; x\in \R^d, \;\;\chi \in \widehat{T^d} := \widehat{\Z^d}.
  \end{equation}
 The bimodule structure over \(C(T^d\times \widehat{T^d})\) is given by 
 \[ (f\cdot h) (x,\chi) = f(x,\chi) h(x, \chi),\]
 where \(T^d\) is understood as \(\R^d/\Z^d\) so \(h\) in this formula is to be interpreted as 
 a continuous 
 function on \(\R^d \times \widehat{T^d}\) which is \(\Z^d\)-periodic in the first 
 variable. 
 
  If \(\chi \colon \Z^d\to \T\) is a 
 character, it induces a complex line bundle \(L_\chi\) over \(\R^d/\Z^d = T^d\), 
 and this is precisely the restriction of \(\beta\) to a the slice 
 \(T^d\times \{\chi\}\cong T^d\) is  the induced vector bundle \(L_\chi\). The 
 total family of these bundles makes up the space of the Poincar\'e bundle. 
 More exactly, the space 
 \( \R\times \widehat{\T} \times \C/\, \sim \) where \( (x-n , \chi, \lambda) \sim 
 (x, \chi, \chi (n)\lambda)\), which projects to \(\T\times \widehat{\T}\) with 
 fibres \(\C\), forms a rank-one complex vector bundle over \(T^d\times  \widehat{T^d}\) 
 whose module of sections is as described above.

The Poincar\'e bundle 
 is the Fourier transform of a 
 finitely generated projective module over \(C(T^d)\otimes C^*(\Z^d)\), which 
 is defined for general discrete groups \(\G\), and in this more general context 
 called the Mischenko bundle, figuring in a common formulation of the 
 Baum-Connes assembly map involving only non-equivariant \(\KK\)-theory 
 and not the equivariant flavour. 
 
 \begin{definition}
 \label{definition:fmcorresnodence}
The \emph{Fourier-Mukai correspondence} is given by the 
topological correspondence of degree \(-d\) 
 \begin{equation}
\label {equation:aodsifasldfjas;ldkf1}
T^d\xleftarrow{\pr_1} ( T^d\times \widehat{T^d}, \Poincare)\xrightarrow{\pr_2} \widehat{T^d},
 \end{equation}
 from \(T^d\) to \(\widehat{T^d}\). 
 The coordinate projection is given its standard \(\K\)-orientations. 
\end{definition}

\begin{remark}
\label{remark:whowantsananaylticcylceme!}
The 
class in \(\KK_{-d}\bigl(C(T^d), C(\widehat{T^d})\bigr) \) of the Fourier-Mukai 
correspondence, given our general remarks earlier on wrong-way maps from 
submersions, is the class of the following analytically defined cycle.

 To each \(\chi \in \widehat{T^d}\), we associate the 
flat Hermitian induced 
vector bundle \(L_\chi:= \R^d\times_{\Z, \chi} \C\) over \(T^d\) and form 
the corresponding twisted Dirac operator \(D_\chi\); then the
 ensemble 
\(\{D_\chi \}_{\chi \in \widehat{T^d}}\) makes up a bundle of elliptic 
operators along the fibres of the coordinate projection 
\(T^d\times \widehat{T^d}\to \widehat{T^d}\), and a cycle for 
for \(\KK_{-d}\bigl(C(T^d), C(\widehat{T^d})\bigr)\).

\end{remark}

\begin{theorem} 
\label{theorem:dirac_is_fourier_mukai}
Under the identification 
\[\KK_{-d}(C_0(\R^d)\rtimes \Z^d, C^*(\Z^d)\bigr) \cong \KK_{-d}(C(T^d)) , C(\widehat{T^d})\bigr)\]
by Morita equivalence in the first variable, Fourier transform in the second, 
the image \(j_{\Z^d} ([\R^d]\) of the Dirac class of \( (\Z^d, \R^d)\) 
under descent, 
is the class of the Fourier-Mukai correspondence 
\eqref {equation:aodsifasldfjas;ldkf1}. 

In particular, the  Dirac map \( \K_*(\widehat{T^d})   \to \K_{*+d}(T^d)\) for \(\Z^d\) acting on a 
point is the Fourier-Mukai transform, given by the map on Baum-Douglas cycles 
of composition 
with the smooth correspondence 
\begin{equation}
\label{equation:aodsifasldfjas;ldkf}
T^d\xleftarrow{\pr_1} ( T^d\times \widehat{T^d}, \Poincare_d)\xrightarrow{\pr_2} \widehat{T^d},
\end{equation}
where \(\Poincare_d\) is the (class of the) Poincar\'e bundle. 
\end{theorem}

\begin{proof}

We start with reviewing the descent map, and work 
in the generality of a general \(\K\)-oriented group \(\G\) as above (though the 
orientation plays no role).

Cycles for \(\KK^\G_*(C_0(Z), C)\) are given by Hilbert spaces \(H\) equipped with a 
unitary action of \(\G\), a \(\G\)-equivariant representation 
\( \pi \colon C_0(Z) \to \Bound (H)\), 
and an operator, which is almost \(\G\)-equivariant in the appropriate sense. 
 Descent, 
applied to this data, setting aside the operator for the moment, 
produces the right \(C^*(\G)\)-module \(j_\G(H)\),
 which is the completion 
of \(C_c(\G, H)\), elements of which we write in group-algebra style 
\[ \sum_{g\in \G} \xi_g [g],\]
under the inner product 
\[ \langle \xi , \xi'\rangle_{C^*(\G)}  = \sum_{g_1, g_2\in \G}  \langle \xi_{g_1}, \xi'_{g_2}\rangle\, [g_1^{-1}g_2],\]
which of course can be re-written as a convolution. 
The right \(C^*(\G)\)-module structure on \(j_\G (H)\) is 
\[ \bigl( \sum_{g\in \G} \xi_g [g] \bigr) \cdot [h] =\sum_{g\in \G} \xi_g [gh].\]
The left action of \(C_0(Z)\rtimes \G\) is given by the covariant pair 
\[ f( \sum \xi_g [g]) = \sum \pi (f) \xi_g[g], \;\;\;\; h\bigl( \sum \xi_g[g]) = \sum h(\xi_g)[hg].\]

The standard Morita equivalence \( C(\G\backslash Z)\)-\(C_0(Z)\rtimes \G\)-bimodule 
Of Rieffel and others is the completion 
\(\mathcal{E}_\G\) of \(C_c(Z)\) under the inner product 
\[ \langle f_1, f_2\rangle_{C_0(X)\rtimes \G} = \sum_{g\in \G} f_1^*g^{-1}(f_2)\]
and the right module structure by the (anti-) covariant pair
\[ f \cdot [h] := h^{-1}(f), \;\;\; f\cdot f':= ff'.\]
In computing descent at the level of cycles, we therefore need to describe the right 
Hilbert \(C^*(\G)\)-module 
\begin{equation}
\label{equation:a_nasty_tensor_product}
 \mathcal{M}(H) := \mathcal{E}_\G\otimes_{C_0(Z)\rtimes \G} j_\G (H) .
 \end{equation}

We will focus on describing this tensor product module when \(H\) is of the form 
\(L^2(Z)\), for a \(\G\)-invariant measure on \(Z\). Exactly the same arguments 
go through to compute \(\mathcal{M}(H)\) when \(H = L^2(S)\) is the space of 
spinors for the Dirac operator on \(Z\). 

Let \(\otimes\) in the following denote the algebraic tensor product. We work with 
simple tensors in the 
tensor product \eqref{equation:a_nasty_tensor_product}, which we can denote  
\( f \otimes_{C_0(X)\rtimes \G} \xi [g]\), where \(f, \xi \in C_c(Z)\), and \(\xi\) thought of 
as an element of the Hilbert space \(L^2(Z)\). 
The null vectors are spanned by the \(C^*(\G)\)-invariant submodule 
spanned over \(\C\) by the 
vectors 
\begin{multline}
\label{equation:nullvectors}
f'f\otimes_{C_0(X)\rtimes \G} \xi[g] - f'\otimes_{C_0(X)\rtimes \G} f\xi[g],\;\;\;
h^{-1}(f)\otimes_{C_0(X)\rtimes \G} \xi[g] - f\otimes_{C_0(X)\rtimes \G} h(\xi)[hg].
\end{multline}
Using the first relation, one sees that 
we are describing the quotient of 
\( H\otimes \C\G\) by the span of the vectors 
\begin{equation}
 \label{equation:nullvectorsdf}
 h(\xi)\otimes [hg] -  \xi\otimes [g],
 \end{equation}
\emph{i.e.}, the tensor product module amounts to forming  the 
quotient 
module, by the diagonal left action of \(\G\) (by right \(C^*(\G)\)-module maps) 
 and then completing under the inner product 
 \begin{equation}
 \label{equation:theyayinnerproduct}
  \langle \xi_1 \otimes [g_1] , \xi_2 \otimes [g_2]\rangle_{C^*(\G)} := \langle \xi_1, \xi_2\rangle \cdot [g_1^{-1}g_2].
  \end{equation}

Now we will use the special structure of \(Z\). Let \(F\subset Z\) be a fundamental 
domain for the \(\G\)-action. 
We define a map, somewhat formally, on elementary tensors by 
\[ \Phi (\xi \otimes [g]) := \sum_{k\in \G} k(\xi|_F) \otimes [kg].\]
Here we start with \(\xi \in C_c(Z)\subset L^2(Z)\), restrict it to \(F\), and 
then move this restricted function on \(F\) periodically over \(Z\), giving the 
functions \(k(\xi|_F)\). 

We can take the co-domain of \(\xi\) to be the space of 
bounded, measurable maps 
\[ \xi \colon Z \to \C\G,\;\;\textup{s.t.} \;\; \xi (gx) = [g] \xi (x) \; \forall x\in Z.\]
With this interpretion, \(\Phi\) sends 
null-vectors \eqref{equation:nullvectorsdf} to zero. We define an 
inner product on such functions by 
\[ \langle \xi_1, \xi_2\rangle_{C^*(\G)}  := \int_F \xi_1^* \xi_2,\]
with \(\xi^* (x):= \xi(x)^* \in \C\G\).
 
Then for \(\xi_i \in L^2(Z)\), \(g_i \in \G\), 
\[ \langle \Phi (\xi_1 \otimes [g_1]) , \Phi (\xi_2\otimes [g_2]\rangle_{C^*(\G)}  = 
\int_F   (\xi_1 [g_1])^* \xi_2[g_2] = \int_F \overline{\xi_1}\xi_2  \;\cdot  [g_1^{-1}g_2]\in \C\G,\]
which matches \eqref{equation:theyayinnerproduct}, so that 
\(\Phi\) is an isometry.

To summarize, we have proved the following. 

\begin{lemma}
Descent and strong Morita equivalence maps the \(\G\ltimes Z\)-Hilbert space 
\(L^2(Z)\) to the right \(C^*(\G)\)-module of maps 
\(\xi \colon Z\to C^*(\G)\) such that \(\xi (gx) = [g] \xi (x)\) for all \(x\in Z\), with 
inner product, module structure 
\[ \langle \xi_1, \xi_2\rangle_{C^*(\G)} := \int_F \xi_1^*\xi_2 \; \;\;\;\;\;\xi [g] \, (x) := \xi (x) [g].\]

\end{lemma}

Very little change needs be made to the above argument when \(H\) is generalized to 
be \(L^2(S)\) for a spinor bundle \(S\) over \(Z\), equipped with a unitary 
\(\G\)-action by bundle maps. Furthermore, if \(D\) is a \(\G\)-equivariant Dirac operator 
on \(L^2(S)\), it descends to a densely defined operator on the right \(C^*(\G)\)-module 
\(j_\G \bigl( L^2(S)\bigr)\) by the obvious formula 
\[ \tilde{D}   (\sum \xi_g[g]) := \sum D\xi_g \, \otimes [g].\] 
The class \( ( j_\G\bigl( L^2(S)\bigr) , \tilde{D})\) 
of this spectral triple represents
 \(j_\G( [Z])\) in \(\KK_{-d} \bigl(C_0(Z)\rtimes \G, C^*(\G)\bigr)\). 
 The Kasparov product 
 \[ [\mathcal{E}]\otimes_{C_0(Z)\rtimes \G} \lambda_\G ([Z])\]
 in the case 
 \(Z = \R^d\), \(\F = \Z^d\) can now be described in the following way. 
 By our work above, the right \(C^*(\Z^d)\)-module 
 \(\mathcal{E}\otimes_{C_0(\R^d)\rtimes \Z^d} j_{\Z^d}\bigl(L^2(Z)\bigr)    \)  
 is isomorphic to the right \(C^*(\Z^d\)-module of maps 
\(\xi \colon \R^d\to C^*(\Z^d)\) such that \(\xi (x+n) =  [n] \xi(x), \;\; \forall x\in \R^d, \; n\in \Z^d,\)
with the \(C^*(\Z^d)\)-valued inner product defined above. 
Under Fourier transform \(C^*(\Z^d) \cong C(\widehat{T^d})\) these correspond to 
maps on \(T^d\times \widehat{T^d}\) satisfying the conditions described in 
 \eqref{equation:sdfsdfsdfgjdkfjg}. That is, 
 the right \(C(\widehat{T})\)-module we are describing is the module of 
 sections of the Miscenko-Poincar\'e module \(L\). A small elaboration of the 
 computations just given integrates a spinor bundle, and the discussion in 
 Remark \ref{remark:analyticclassoffm} shows that this space carries a 
 natural bundle of Dirac operators on it, because we can twist by flat connections. 
 The axioms of the Kasparov 
 product proves the result.

For the second statement, note that for 
any discrete group, the 
standard identification (at the level of cycles) \(\KK^\G_*(A, \C) \cong 
\KK_*(A\rtimes \G, \C)\) factors through the composition of the 
descent functor 
 and the map 
\( \epsilon_*\colon \KK_*\bigl(A\rtimes \G, C^*(\G)\bigr) \to \KK_*(A\rtimes \G, \C)\) 
induced by the trivial representation \(C^*(\G) \to \C\). 
It follows from a quick calculation that if \((\G, Z)\) is \(\K\)-oriented, then 
the Dirac map interpreted as a map  
\[ \KK_*(C^*(\G), \C) \to \KK_{*-d} (C_0(Z)\rtimes \G, \C)\]
is given by Kasparov composition with the image 
\( j_\G ([Z]) \in \KK_{*-d} (C_0(Z)\rtimes \G, C^*(\G)\bigr)\) of the Dirac 
class \([Z]\in \KK^\G_{-d}(C_0(Z), \C)\) under descent.  We have computed 
this class above and showed that it is that of the Fourier-Mukai correspondence. 
This proves the other assertion.

\end{proof}

\begin{remark}
Taking its index of the correspondence \eqref{equation:aodsifasldfjas;ldkf1}, that is, 
composing it with the correspondence 
\[ \pnt \leftarrow T^d\xrightarrow{\id}T^d,\]
gives the correspondence 

 \begin{equation}
\label {equation:aodsifasldfjas;ldkf1}
\pnt \leftarrow ( T^d\times \widehat{T^d}, \Poincare)\xrightarrow{\pr_2} \widehat{T^d},
 \end{equation}
The corresponding analytic cycle for \(\KK_{-d} (\C, C(\widehat{T}^d))\) is 
sometimes described as a kind of Hilbert module index: twisting in a suitable sense 
the 
Dirac operator on \(T^d\) by the Poincar\'e bundle over \(T^d\times \widehat{T}^d\) 
gives a Fredholm operator on a \(C(\widehat{T}^d)\)-module whose \(C(\widehat{T}^d)\)-
index gives a 
class in \(\K_{-d}(C^*(\widehat{T}^d))\). 
After Fourier transform, what t we are describing is the application of the 
Baum-Connes analytic assembly map 
\[ \mu \colon \KK_*(C(T^d), \C) \to \KK_{-d}(\C, C^*(\Z^d))\]
to the class \([T^d]\) of the Dirac operator on \(T^d\). 

The Theorem above implies that application of 
descent 
\[ \KK^{\Z^d}_*(C_0(\R^d), \C) \to \KK_*(C_0(\R^d)\rtimes \Z^d, C^*(\Z^d))\]
and Morita equivalence \(C_0(\R^d)\rtimes \Z^d\sim C(T^d)\) to the 
\(\Z^d\)-equivariant Dirac class \([\R^d]\) gives the class of the Fourier-Mukai 
correspondence, but from this it follows that that after taking the index, as just discussed,
we obtain the class \(\mu ([\R^d]) \in \KK_{-d} (\C, ^*(\Z^d))\).

This equality 
of classes in \(\K_*\bigl( C^*\Z^d)\bigr)\) has a more general version, equating
two apparently two slightly different 
methods of defining the Baum-Connes assembly map 
for torsion-free discrete groups. The paper \cite{Land} of Land 
shows that they are the same. 
There is therefore some connection between our statement and Land's, 
at least in this special case of the group \(\Z^d\), and only after taking the 
index; our statement involving the Fourier-Mukai correspondence 
is slightly more `bivariant' in nature. But some of the arguments in the proof 
above rather similar to some of those of Land.

  \end{remark}

We now return to the localization map for \(G = \Z^d\). We are going to 
describe it geometrically.

\subsection{Computation of the localization map for free abelian groups}

If \(\G = \Z^d\),  \(T^d\) the torus \( T^d:= \R^d/\Z^d\), and \(\widehat{T^d} := \widehat{\Z^d}\) the 
dual torus  
then \(\KK^{\Z^d}_*(\C, \C) \cong \KK_*(C^*(\Z^d), \C)\) and 
by Fourier transform 
\(\KK_*(C^*(\Z^d), \C) \cong \KK_*(C(\widehat{T^d}), \C)\), while Morita invariance implies that 
\(\KK_*^{\Z^d}(C_0(\R^d), \C) \cong \KK_*(C(\R^d/\Z^d) , \C) = \KK_*(C(T^d), \C)\).

The Dirac-localization map therefore identifies with a map
\begin{equation}
\label{equation:dirac_map_abelian}
 \K_{-*}(\widehat{T^d})   \to \K_{*+d}(T^d).
 \end{equation}

Now fixing a basis and dual basis, the 
groups \(\K_*(T^d)\) and \(\K_*(\widehat{T^d})\) 
 can each be identified, by the K\"unneth Theorem,
 with the graded tensor product 
 \[\K_*(\T) \hat{\otimes_\Z} \cdots \hat{\otimes_\Z} \K_*(\T).\]
 This blurs a little the difference between \(\T^d\) and \(\widehat{T^d}\) but 
 we will return to a coordinate free perspective below. 
 
 The group \(\K_*(\T)\) has two generating \(\K\)-homology classes: the \(\K_0(\T)\)-class 
 \([\pnt]\) of a point \(p\) in \(\T\), 
 and the \(\K_1(\T)\)-class \([\T]\) of the Dirac operator on the 
 circle. Now for an \(r\)-tuple 
 \({\bf k} = k_1 < \cdots <k_r\) of integers from \( \{1, \ldots , n\}\) 
 we associate the 
  \(l\)-dimensional correspondence (or Baum-Douglas cycle) 
 \begin{equation}
 \label{equation:bd_cycle_ofa_standard_coordinate_embedding}
  T^d \leftarrow T^l \to \pnt,
  \end{equation}
 with \(T^j\) carrying its standard (product) \(\K\)-orientation, and the left map 
 \(T^l \to T^d\) sending the \(j\)th coordinate 
 circle of \(\T^l\) into the \(k_j\)th coordinate of \(T^d\), and putting the point 
 \(p\) into the other coordinates. 
 We call this a \emph{standard coordinate 
 embedding} of an \(l\)-torus in \(T^d\). Let \([\bf{k}] \in \K_r(T^d)\) be it's Baum-Douglas 
 class.

Obviously, every standard coordinate embedding \(T^l\to T^d\) comes along with 
a `dual' coordinate embedding \(T^{d-l}\to T^d\), mapping circles into the complementary 
coordinates, and putting \(p\)'s in the other coordinate spots. Let \([{\bf k}^\perp] \in 
\K_{*+d}(T^d)\) be 
its class.

\begin{theorem}
\label{theorem:dirac_for_zd}
Let 
\[ \Dirac \colon \K_*(\widehat{T^d})  \cong
\ \KK^{\Z^d}_{-*}(\C ,\C) \to \KK^{\Z^d}_{-*-d}(C_0(\R^d), \C)\cong 
\K_{*+d}(T^d)\]
be the localization map for \(\Z^d\). 

If 
\([\bf{k}]  =  k_1 < \cdots < k_r\)
 is the class of a standard coordinate embedding, then 
\[ \Dirac ({\bf k}]) = (-1)^{dk+\frac{k(k-1)}{2}} \,\cdot [\bf{k}^\perp],\] 
with \([\bf{k}^\perp]\) the class of the dual coordinate embedding.  

\end{theorem}

In particular, the Dirac map, that is, composition with \eqref{equation:aodsifasldfjas;ldkf}, 
 interchanges the \(\K\)-homology classes of a 
point in the dual torus \(\widehat{T}^d\), and the class of the Dirac operator on the 
torus \(T^d\).

\begin{proof}

Given the K\"unneth theorem, and
 taking products, it suffices to verify the assertion for \(\Z\), since the 
 \(\K\)-homology classes we are considering are all external products 
 of \(\K\)-homology classes for \(\T\).  
 
   The Dirac 
map is thus 
\begin{multline}
 \KK_*( C(\widehat{\Z}), \C) \cong \KK_*(C^*(\Z), \C) =  \KK^\Z_*(\C, \C) \xrightarrow{[\R]\otimes} \KK^\Z_{*-1}(C_0(\R), \C) \\ \cong \KK_{*-1}(C(\R/\Z), \C).
 \end{multline}
The \(\K\)-homology class of a point in \(\widehat{\Z}\) maps to the 
class of the trivial representation in \(\KK^\Z_0(\C, \C)\), which acts as a unit, so the 
image under the next map is the \(\Z\)-equivariant Dirac class \( [\R]\in \KK_{-1}^\Z(C_0(\R), \C)\), 
\emph{i.e.} the Dirac class for the circle in \(\KK_{-1}(C(\R/\Z), \C)\).

Hence the Dirac map sends a point class \([\pnt]\) 
 in \(\K_0(\widehat{\T})\) to the Dirac class 
\([\T]\in \K_{-1}(\T)\). It remains to show that it maps \([\widehat{\T}]\) to the class 
\([\pnt]\) of a point in \(\K_0(\T)\). 

Under the identification 
\(\K_0(\T) \cong \KK^\Z_0(C_0(\R), \C)\), the point homology 
class for \(\T\) corresponds to the 
class 
\([\ev] \in \KK^\Z_0(C_0(\R), \C)\) of the \(\Z\)-equivariant representation 
\(C_0(\R) \to C_0(\Z) \subset \Comp(^2\Z)\) due to the inclusion \(\Z \to \R\). 
We need to prove that 
\([\widehat{\T}]\otimes_\C [\R]  = [\ev] \in \KK^\Z_0(C_0(\R), \C)\). 
To do so 
we employ the fact that \([\R]\in \KK_{-1}^\Z(C_0(\R), \C)\) 
is invertible in \(\KK^\Z\).  Let \(\eta \in \KK_1^\Z(\C, C_0(\R))\) be the class of the 
self-adjoint unbounded multiplier \(\delta (x) = x\) of \(C_0(\R)\). Then 
\[ \eta \otimes_{C_0(\R)} [\R] = 1\in \KK^\Z_0(\C, \C), \;\;  [\R]\otimes_\C \eta = 1_{C_0(\R)}\in 
\KK_0(C_0(\R), C_0(\R))\]
is the content of the Dirac-dual-Dirac method for \(\Z\).

Therefore the equation \([\widehat{\T}]\otimes_\C [\R]  = [\ev] \in \KK^\Z_0(C_0(\R), \C)\)
 we want to prove is 
equivalent to the equation 
\[ \eta\otimes_{C_0(\R)} [\ev] = [\widehat{\T}]\in \KK_{-1}^\Z(\C, \C).\]
But the composition on the left is clearly represented by the 
spectral triple with Hilbert space \( l^2(\Z)\), and operator the densely defined 
operator of multiplication by \(n\). It's Fourier transform is therefore the 
Dirac operator on the circle \(\widehat{\T}\). This completes the proof.

\end{proof}

\begin{remark}
Naturally, one can try
 to prove that the Dirac map, say, for \(\Z\), interchanges the point class and 
the Dirac classes for \(\T\) and \(\widehat{\T}\), using the correspondence 
\eqref{equation:aodsifasldfjas;ldkf} rather directly. 

 Composition of the correspondences
\begin{equation}
   \T \leftarrow (\T\times \widehat{\T}, \beta) \rightarrow 
      \widehat{\T} \leftarrow \pnt \rightarrow \pnt 
\end{equation}
gives by transversality 
\begin{equation}
\T \leftarrow (\T , \beta|_{\T}) \rightarrow \pnt
\end{equation}
where we understand \(\T\) as a subset of \(\T\times \widehat{\T}\) by 
\(\T\times \{p\}\), the point we have picked. Taking \(p=0\), the zero 
character of \(\Z\), is convenient. The inclusion \(\T\to \T\times \T\) 
of the slice pulls back the f.g.p. \( C(\T \times \T)\)-module defining 
\(\beta\) to the module of continuous functions on \(\R\times \{0\}\) 
such that \( f(x+n,0) = f(x,0)\), \emph{i.e.} periodic functions, and hence 
the restricted module is precisely \(C(\T)\). Hence 
the restriction of \(\beta\) to \(\T\) is the class of the 
trivial line bundle, proving that 
the Dirac map sends the point \(\K\)-homology class to the correspondence 
\[ \T \xleftarrow{\id} \T \to \pnt,\]
which describes the Dirac class \([\T]\). 

Things do not run so smoothly if one starts with the class \([\T]\). The resulting 
correspondence involves a manifold of dimension \(2\) which then must be 
Thom, or Bott un-modified, to a manifold of dimension zero (a point), and one has
to take care of the bundle as well. It is easier to use the dual-Dirac method
to do the calculation, as we did in the proof of Theorem 
\ref{theorem:dirac_for_zd}, since then the cycle one \emph{begins} the calculation
 with is a point. 
\end{remark}

With some labour, the following result can be deduced from Theorem 
\ref{theorem:dirac_for_zd}, but we will prove it more geometrically in a 
forthcoming note with D. Hudson \cite{EH}.

Suppose \(T\subset T^d\) is a \(j\)-dimensional 
torus subgroup. It lifts to a \(j\)-dimensional linear subspace 
\(L\subset \R^d\).

 Let \(L^\perp\) be all characters of \(\R^d\) which vanish on 
\(L\), \( (\Z^d)^\perp\) all \(\chi \in \widehat{\R^d}\) that vanish on \(\Z^d\) and 
\(T^\perp\) the projection to \(\widehat{\R^d}/ (\Z^d)^\perp \cong \widehat{\Z^d}= \widehat{T^d}\)
 of \(L^\perp\). It is a \(d-j\)-dimensional subtorus of \(\widehat{T^d}\). It can 
 be suitably \(\K\)-oriented in order to make the following result true.

\begin{theorem}
Let \(T\subset T^d\) be a \(j\)-dimensional linear torus, \(T^\perp \subset \widehat{T^d}\)
 its 
\(d-j\)-dimensional dual torus. Then the Dirac-Fourier-Mukai transform 
\[\K_j(T^d) \to \K_{j+d}(\widehat{T^d})\]
maps the cycle \([T]\) to the cycle \([T^\perp]\).

\end{theorem}

Finally, we note that just as application of descent and the Fourier transform 
to the Dirac class of \( (\R^d, \Z^d)\) 
produced the Fourier-Mukai correspondence
\[ \FM \in \KK_d\bigl(C(T^d), C(\widehat{T^d})\bigr),\]
when we do the same with the \emph{dual}-Dirac element 
\[\eta \in \KK_d^{\Z^d}\bigl( \C, C_0(\R^d)\bigr)\] 
we obtain the class 
\(\FM ' \in \KK_d\bigl(C(\widehat{T^d}), C(T^d)\bigr)\) of the 
`reversed' Fourier-Mukai correspondence 
\begin{equation}
\label{equation:aodsifasldfjas;ldkf_reversed}
\widehat{T^d}\xleftarrow{\pr_1} ( \widehat{T^d}\times T^d, \beta)\xrightarrow{\pr_2} T^d,
\end{equation}

\begin{corollary}
Application of the composition of functors 
\[\KK_d^{\Z^d}\bigl( \C, C_0(\R^d)\bigr) \to \KK_d\bigl(C^*(\Z^d), C_0(\R^d/\Z^d)\bigr) 
\cong \KK_d\bigl(C(\widehat{T^d}), C(T^d)\bigr)\]
to the dual-Dirac element \(\eta\) for \(\R^d\) gives the class of 
the dual Fourier-Mukai correspondence \(\FM'\) of 
\eqref{equation:aodsifasldfjas;ldkf_reversed}.

In particular, 
\[ \FM '\circ \FM = \id, \;\;\;\;\FM \circ \FM ' = \id,\]
as morphisms in the category \(\KK\).

\end{corollary}

The fact that the Dirac map 
descends to the Fourier-Mukai transform seems to be one of the `well-known to experts'
kind, but has not apparently been written down anywhere as far as the author has been able 
to determine. The physics literature, however, is extensive on the matter, as Fourier-Mukai 
duality is a form of T-duality (see for example \cite{MR}). The recent paper 
\cite{Block} treats Fourier-Mukai duality in terms of \(\K\)-theory as well, but the context is 
slightly different.

\section{The Dirac class of a discrete group}

 Motivated by the 
calculations with \(\Z^d\) of the previous section,
 we now proceed to a
  notion of the Dirac class on the 
`noncommutative manifold' underlying the C*-algebra  \(C^*(\G)\)
for possibly non-abelian discrete groups \(\G\). 

\subsection{Point classes and Dirac classes}
We will continue to assume that \(\G\) admit a smooth orientation \( (\G, Z)\). 
The group C*-algebra
 in this case may be considered -- at least at the level of 
homology -- as a `noncommutative compact smooth oriented manifold' of dimension 
\(d\) with the `Dirac class' \([\widehat{\G}]\in \KK^\G_{d}(\C, \C) \cong 
\K^{d}\bigl( C^*(\G)\bigr)\) defined below, playing 
the role of the class of the associated Dirac operator.

Choose a point \(x_0 \in Z\). The group \(\G\) acts on \(l^2\G\) by the 
regular representation. 
The composition 
\begin{equation}
\ev \colon C_0(Z) \xrightarrow{\mathrm{restr}} C_0(\G x_0) \rightarrow C_0(\G) \xrightarrow{\mathrm{mult}} \Comp(l^2\G)
\end{equation}
is a \(\G\)-equivariant *-homomorphism and determines a class 
\begin{equation}
\label{equation:point_evaluation_class}
 [\ev]\in \KK^\G_0(C_0(Z), \C).
 \end{equation}
 Since 
\(Z\) is a classifying space, it is path-connected and \([\ev]\) is 
independent of the choice of point.

\begin{definition}
\label{definition:Dirac_class_for_group}
Let \((\G, Z)\) be a smooth oriented group of dimension \(d\). 
A \emph{Dirac class} for \(\G\) is any 
class 
\[ [\widehat{\G}]\in \K^{d} \bigl(C^*(\G)\bigr) = \KK^\G_d(\C, \C)\]
such that 
\[ \Dirac ([\widehat{\G}]) = [\ev]\in \KK^\G_0(C_0(Z), \C),\]
where \(\Dirac \) is the localization map (Definition \ref{definition:Dirac_map}.) 
\end{definition}

The definition is of course largely motivated by Theorem
 \ref{theorem:dirac_for_zd}.

\begin{remark}
One could make a variant of the definition involving the 
\emph{reduced} C*-algebra of the group. The projection 
\(C^*(\G) \to C_r^*(\G)\) induces a pullback 
map \(\K^*\bigl( C^*_r(\G)\bigr) \to \K^*\bigl( C^*(\G)\bigr)\) so 
there is an associated (reduced) localization map 
\begin{equation}
\label{equation:reduced_Dirac_map} 
\K^*\bigl( C^*_r(\G)\bigr) \to \K^*\bigl( C^*(\G)\bigr) \cong \KK^\G_*(\C, \C) 
\xrightarrow{\otimes_\C [Z_\G]}\KK^\G_{*+d}(C_0(Z), \C).
\end{equation}
Occasionally we will call a class in \(\KK_{+d} (C^*_rG, \C)\) 
mapping to \([\ev]\in \KK^\G_0(C_0(Z), \C)\) under the reduced localization map 
a \emph{reduced Dirac class}; note 
that the projection \(C^*(\G) \to C^*_r(\G)\) pulls a reduced 
Dirac class back to a Dirac class, so the existence of a reduced class is 
stronger. 

\end{remark}

\begin{example}
If \(\G\) is finite, so \(Z\) is a point, 
the localization map is the identity map and the Dirac class for 
\(\G\) is represented by the homomorphism 
\[ \lambda\colon C^*(\G) \to \K(l^2\G),\]
induced by the regular representation.
In particular, if \(\G\) is finite abelian, so that \(C^*(G) \cong C(\widehat{G})\), 
the Dirac class of \(G\) is the sum of the point \(\K\)-homology classs of 
\(\widehat{G}\). More generally, the Dirac class of a finite group 
is the sum of the point \(\K\)-homology classes of the points of 
\(\widehat{G}\), each with multiplicity given by its dimension.

\end{example}

\begin{proposition}
\label{proposition:Dirac_of_zd}
 The Dirac class \([\widehat{\Z^d}]\in \KK^{\Z^d}_{d}  (\C, \C) \cong 
 \K^{d}(C^*(\Z^d)\) 
 of the smooth oriented group 
 \( (\Z^d, \R^d)\) is the 
Fourier transform of the transverse Dirac class \([ \widehat{\Z^d}]\) of 
its Pontryagin dual.  

\end{proposition}

\subsection{The Dirac class is non-torsion}
The following proves that \([\widehat{\G}]\) is always nonzero and non-torsion 
in 
\(\K^{d}\bigl( C^*(\G)\bigr)\). The result will be refined later. 


\begin{proposition}
\label{proposition:nonvanishing_of_Dirac_for_groups}
Let \( (\G, Z)\) be a smooth, \(\K\)-oriented group, and \([\widehat{\G}]\in \K^{d}\bigl( 
C^*(\G)\bigr)\) a Dirac class. Let 
\[ \mu \colon \KK^\G_*(C_0(Z), \C) \to \K_*\bigl(C^*(\G)\bigr)\]
be the Baum-Connes assembly map. Then 
\begin{equation}
\label{equation:pairings_1}
 \langle \mu ([Z]), [\widehat{\G}]\rangle 
:=\mu ([Z])\otimes_{C^*(\G)} [\widehat{\G}]) = 1\in \KK(\C, \C) \cong \Z.
\end{equation}
In particular, the Dirac class \([\widehat{\G}]\) of any smooth \(\K\)-oriented group 
induces a surjection \(\K_*\bigl(C^*(\G)\bigr)\to \Z\), and is 
never zero in \(\K\)-homology, nor torsion. 

The analogous statement holds for a reduced Dirac class.  
\end{proposition}

\begin{proof}

Denote by \( f\mapsto \overline{f}\) the descent map 
\(\KK^\G_*(A,B) \to \KK_*(A\rtimes \G, B\rtimes \G)\). 
If \( (\G, Z)\) is a smooth oriented group, then \(\EG = Z\) and 
the Baum-Connes assembly map with coefficients \(B\) is, by definition, 
\(\mu (f) = [P_\varphi]\otimes_{C_0(Z)\rtimes \G} \overline{f}\) for 
any \(f\in \KK^\G_*(C_0(Z), B)\), and any 
 \emph{cut-off} function 
\(\varphi \in C_c(Z)\), \(0\le \varphi \le 1\) and \(\sum_{g\in \G} g(\varphi) = 1\). 
The space of cut-off functions is convex and nonempty, 
and \(P_\varphi := \sum_{g\in \G} g(\varphi)[g]\in 
C_0(Z)\rtimes \G\) is a projection whose homotopy-class does not depend on 
the choice of cut-off function. For the element \([\ev] \in \KK_0^\G(C_0(Z), \C)\) and 
the element \([P_\varphi]\in \K_0(C_0(Z)\rtimes \G)\), observe that 
the map \(C_0(Z) \rtimes \G \to C_0(\G ) \rtimes \G \cong C_0(\G)\rtimes \G = \Comp(l^2\G)\) 
maps \(P_\varphi\)  to a projection homotopic through projections to a 
rank-one projection. Hence 
\[ \langle [\ev] , [P_\varphi]\rangle = 1.\]
On the other hand, descent maps the Dirac class \([Z]\in 
\KK^\G_{-d}(C_0(Z), \C)\)  to a class 
\[[\overline{Z }]\in \KK_{-d}\bigl(C_0(Z)\rtimes \G, C^*(\G)\bigr),\] and 
identifying \(\KK^\G_*(C_0(Z) , \C) \cong \KK^\G_*(C_0(Z)\rtimes \G, \C)\), 
the Dirac map can be identified with the map 
\[ \KK_*(C^*(\G), \C) \to \KK_{*-d}(C_0(Z)\rtimes \G, \C)\]
of composition with \([\overline{Z}]\). Hence 
\begin{multline}
\langle \mu ([Z]) , [\widehat{\G}]\rangle 
=  [P_\varphi]\otimes_{C_0(Z)\rtimes \G} 
[\overline{Z_\G}]\otimes_{C^*(\G)} [\widehat{\G}]
= [P_\varphi]\otimes_{C_0(Z)\rtimes \G} \Dirac ([\widehat{\G}])\\
= [P_\varphi]\otimes_{C_0(Z)\rtimes \G} [\ev] = 1\in \Z \cong \KK_0(\C, \C)
\end{multline}
as claimed. 
\end{proof}

\subsection{Dirac classes and the dual-Dirac method}

If \((\G, Z)\) is a smooth, oriented group, then a \emph{dual-Dirac class for \(\G\)}
is a class \[\eta \in \KK^\G_{d}(\C, C_0(Z))\] such 
that \([Z]\otimes_\C \eta = 1_{C_0(Z)}\in \KK_0^\G(C_0(Z), C_0(Z))\). 
The existence or non-existence question of such \(\eta\) is determined by a certain 
coarse co-assembly map (see \cite{EM:dual_Dirac}), and so is a question about 
the large-scale geometry of \(\G\). The fact that the two coordinate projections 
\(Z\times Z \to Z\) are \(\G\)-equivariantly homotopy implies by an easy exercise 
that \([Z] \otimes_{\C} \eta = 1_{C_0(Z)}\) as well. 

 If \(Z\) admits a 
\(\G\)-invariant Riemannian metric of nonpositive curvature, then 
a dual-Dirac class exists and is described in detail below. More generally, 
a dual-Dirac morphism 
\(\eta\) exists if \(\G\) admits a uniform embedding in Hilbert space
 (this theorem appears as Theorem 9.2 
in \cite{EM:Coarse} but it is 
a consequence of assembling results of many others, like
 \cite{HR2}, \cite{STY}, or \cite{Tu}). This 
 is the case if \(G\) is linear, or hyperbolic, for example.

\begin{proposition}
\label{proposition:dual_Dirac_implies_Dirac}
If \( (\G, Z)\) is a smooth oriented, group with a dual-Dirac morphism 
\( \eta \in \KK^\G_{d}(\C, C_0(Z))\), then 
\[ \eta\otimes_{C_0(Z)} [\ev]\in \KK^\G_{d}(\C, \C) \cong \K^{d}\bigl( C^*(\G)\bigr)\] is a Dirac class for \(\G\). 

\end{proposition}

\begin{proof}
\([Z]\otimes_\C ( \eta \otimes_{C_0(Z)} ) = ([Z]\otimes_\C \eta) \otimes_{C_0(Z)} [\ev] = [\ev]\) 
since \( [Z]\otimes_\C \eta = 1_{C_0(Z)}\). 

\end{proof}


Let \(\G\) be torsion-free, so 
 \(\K_{*}(C_0(Z)\rtimes \G) \cong \K^{*}(\G\backslash Z) \cong 
\K^{*}(B\G)\) and a class \(a\) in this ring yields a \emph{higher signature}, 
associating to any compact manifold \(M\) with fundamental group \(\G\) 
the index \(\index (D_M^{\mathrm{sig}}\cdot \chi^*a)\) of the signature operator on 
\(M\) twisted by the \(\K\)-theory class \(\chi^*a\), where \(\chi \colon M \to B\G\) 
is the classifying map for \(M\).

 By a theorem of Kaminker and Miller \cite{KM}, building on work of many others, 
such classes are homotopy-invariant when they are in the range of the 
inflation map 
\begin{equation} \inflate\colon \KK^\G_*(\C, \C) \to \RKK^\G_*(Z; \C, \C)\\ \cong \K^*(\G \backslash Z)\cong \K^*(B\G).
\end{equation}
see \eqref{equation:inflate}. 

Since the composition of the inflation map and Poincar\'e duality
\[ \KK^\G_*(\C, \C) \xrightarrow{\inflate}\RKK^\G_*(Z; \C, \C) 
\cong \K^*(\G\backslash Z)\xrightarrow{\PD}\KK^\G_{*-d}(C_0(Z), \C)\cong 
\K_{*-d}(\G\backslash Z)=\K_{*-d}(B\G)\]
is exactly the Dirac map, it follows that the higher signature associated to the 
Poincar\'e dual of any class in the range of the 
Dirac map, is homotopy-invariant. 

This applies in particular to the point class 
 \([\pnt]\in \K_0(\G \backslash Z)\) (or \([\ev]\in \KK^\G_0(C_0(Z), \C)\), as we have 
 been denoting it above.)

\begin{corollary}
\label{corollary:higher_sign_for_a_point}
Let \( (\G, Z)\) be a smooth oriented, torsion-free group. Then a Dirac class 
for \(\G\) is a pre-image in \(\KK^\G_{d}(\C, \C)\) of the Poincar\'e dual 
of a point 
\( \in \K^{-d}(\G\backslash Z)\cong \K^{-d}(B\G)\), 
under the inflation map.

 In particular, the higher signature associated to the 
 Poincar\'e dual of a point in \(\G\backslash Z\cong B\G\) is homotopy-invariant as soon 
 as a Dirac class exists for \(\G\). 

\end{corollary}

The Poincar\'e dual of a point, and the question of the homotopy-invariance 
of the higher signature associated to it, is studied in \cite{CGM} by Connes, Gromov 
and Moscovici. In 
their paper, they analyze geometric conditions one can put on a group 
under which one can write down a `Dirac class', thus guaranteeing 
homotopy-invariance (for a single cohomology class). 
 The procedure they use to build such `Dirac cycles' (they build more general 
 ones as well), is based on the non-positive curvature idea we exploit below.

Assume that \((\G,Z)\) is a smooth oriented group such that \(Z\) admits a 
\(\G\)-invariant metric of nonpositive curvature. For example, \(\G\) could be a 
discrete subgroup of a connected and semisimple
 Lie group \(G\); if \(K\subset G\) is the maximal 
compact subgroup, then \(G/K\) is a symmetric space of nonpositive curvature 
admitting a \(G\)-invariant (and hence \(\G\)-invariant) orientation. 

The hypotheses say that there is a \(\G\)-equivariant spinor bundle \(S\) on 
\(Z\), and a \(\G\)-equivariant bundle map 
\[ \mathrm{Cliff}(TZ)\otimes S \to S\]
specifying the Clifford multiplication on \(S\). We denote Clifford multiplication 
by a tangent vector \(\xi\) by \(c(\xi)\), so if \(z\in Z\) and \(\xi \in T_zZ\) is a 
unit tangent vector then 
\(c(\xi)\) is a certain self-adjoint 
endomorphism of \(S_z\) with square \(1\). 

For \(z\in Z\) let 
 \(\exp\colon T_{z}Z \to Z\) be the exponential 
map, a diffeomorphism, and \(\log_z \colon Z \to T_{z}Z\) its inverse. Then 
nonpositive curvature implies that \(\log\) is a Lipschitz map, that is, 
\( | \log_z (x) - \log_z (y) | \le d(x,y)\), where \(d\) is the metric on \(Z\) induced 
by the Riemannian metric on \(Z\). Let 
\[ \delta \in \mathrm{End} (S)\]
be the (unbounded) bundle endomorphism of the spinor bundle given by 
\[ \delta (z, \xi) = c\bigl(\log_z (z_0)\bigr)\]
The group \(\G\) acts by unitary bundle 
maps of \(S\), giving the module \(C_0(Z, S)\) of sections the structure of a 
\(\G\)-Hilbert module. Moreover, the multiplier
 \(\delta\) almost-commutes with \(\G\), because \(\log_z\) is Lipschitz. 

The cycle consisting of the \(C_0(Z)\)-module \(C_0(Z, S)\) of sections of \(S\), 
together with the unbounded endomorphism \(\delta\), represents the 
dual-Dirac morphism 
\[ \eta \in \KK_{+d}^\G(\C, C_0(Z)).\]
Restricting it to an orbit gives a concrete model for the composition 
\[ \delta \otimes_{C_0(Z)} [\ev]\in \KK^\G_{+d} (\C, \C).\]
The Hilbert space is \(l^2(\G, S|_\G)\), the unitary group action is 
induced from the initial action on \(S\), and the operator is the 
restriction of \(\delta\) above to the orbit, a Clifford multiplication operator 
on \(l^2(\G, S|_\G)\). 

 It defines a 
spectral triple \( (l^2(\G, V), \pi, \delta)\) for \(C^*(\G)\) (not finitely summable in general, 
see Remark \ref{remark:Dirac_fred}).

\begin{theorem}
\label{theorem:Dirac_class_for_nonpositively_curved_groups}
If \((\G, Z)\) is a smooth oriented group and \(Z\) is equipped with a \(\G\)-invariant 
metric of nonpositive curvature, then the spectral triple \( (l^2(\G, S|_\G),  \delta)\) 
defined above represents a Dirac class \([\widehat{\G}]\) for \(\G\).  

Furthermore, it 
 is defined over the \emph{reduced} C*-algebra of \(\G\) 
and hence there is a reduced Dirac class for \(\G\) in this case, given by the same cycle 
but regarded as defined over \(C^*_r(\G)\). 
\end{theorem}

\begin{proof}
This follows from Proposition \ref{proposition:dual_Dirac_implies_Dirac} and the 
discussion prior to the statement.

\end{proof}

\begin{remark}
\label{remark:Dirac_fred}
The Fredholm module (defined over \(C_r^*(\G)\), its class is a 
reduced Dirac class) just described is finitely summable if 
\(\G\) has polynomial growth.  
By a result of Connes, the existence of a finitely summable spectral 
triple representing a nonzero class in \(\KK^\G_0(\C, \C))\) 
implies amenability of \(\G\) 
(see \cite{Connes}). 
\end{remark}

\begin{remark} 
A. Connes has described Dirac cycles for discrete groups 
 -- representatives, in our terminology, 
of the Dirac class \([\widehat{\G}]\) -- in his book \cite{Connes}, and again, these 
are the same cycles we describe below. 
\end{remark}

\section{Transverse Dirac classes for smooth actions }
\label{subsection:inflationstuff}
In Section \ref{section:bcloc} we discussed the 
 \emph{transverse Dirac class} \([Z]\in \KK_{-d}(C_0(Z), \C)\) when \((\G, Z)\) is a 
 \(\K\)-oriented group, also pointed out that 
 one can similarly define the transverse 
 Dirac class \([X]\in \KK_{-\dim X} (C_0(X), \C)\) for a smooth manifold admitting a 
 \emph{proper} and \(\K\)-orientation-preserving action of \(\G\), 
 in the form of a direct construction of 
 an analytic cycle for \(\KK\) in the form of the Dirac operator on \(X\) 
 associated to the 
 \(\K\)-orientation. 
 
In this section we are going to extend the definition of such transverse 
Dirac classes to 
non-proper actions, some of which, like the action 
of a surface group \(\G\) on the boundary circle of the hyperbolic plane, leave 
no Riemannian metric invariant, which makes it difficult to construct a 
\(\G\)-equivariant Dirac operator directly. 
Our main goal in this section is to define the problem, in terms of the 
localization map.

For any locally compact group with classifying space \(\EG\) for proper actions, the 
transformation groupoid \(\Grd_\G := \G\ltimes \EG\) is a proper groupoid. A 
proper groupoid with a one-point unit space is a compact group. So proper 
groupoids generalize compact groups. 
The \(\Grd\)-equivariant Kasparov category 
\(\KK^{\Grd}_*\) is defined by Le Gall in \cite{LeGall} for any locally compact groupoid. 
 It is functorial in the groupoid, so the  
homomorphism 
of groupoids \(\Grd_\G\to \G\) by mapping \(\EG\) to a point induces a natural 
map
\begin{equation}
\label{eq:inflation_map}
\inflate:\KK^\G_*(A,B) \to \KK_*^{\Grd_\G}\bigl(C_0(\EG)\otimes A, C_0(\EG)\otimes B\bigr)\end{equation}

This is Kasparov's \emph{inflation map} \eqref{equation:inflate}, after identifying 
\(\KK_*^{\Grd_\G}\bigl(C_0(\EG)\otimes A, C_0(\EG)\otimes B\bigr)\) with Kasparov's 
 \(\RKK^\G_*(\EG; A , B)\) defined in \cite{Kas}.

  \begin{remark}
  \label{remark:inflationisomorphism}
  Isomorphism of the inflation map for 
a given \(\G\) and arbitrary \(A,B\) is implied by the \(\gamma = 1\) version of the 
Baum-Connes conjecture (see \cite{Meyer-Nest:BC}). If 
\(\G\) acts amenably on \(X\) and \(A = C_0(X)\) then
 \eqref{eq:inflation_map} is an isomorphism by 
the Higson-Kasparov-Tu theorem (\cite{Higson-Kasparov}, \cite{Tu}) (these 
authors prove that \(\gamma = 1\)). 
\end{remark}

It is generally easier to direct analytic cycles for \(\RKK^\G\), representing for example 
wrong-way maps \(f!\),  than it is in  
\(\KK^\G\), as the following example illustrates -- it is because inflation results in a 
significant weakening of the equivariance condition on cycles.

\begin{example}
Let \(\Gamma\) be any countable group of
diffeomorphisms of the circle \(\T\). Then the action is
 \(\KK\)-orientably, as we show below, 
and we construct an analytic cycle in 
\(\RKK^\G_{-1}(\EG; C(\T), \C)\)
representing \(p_{\EG, \T}!\).

Assume for 
simplicity that \(\EG\) is \(\G\)-compact and
 can be \(\G\)-equivariant triangulated, fix such a triangulation.

 Identify \(\T = \R/\Z\) with coordinate \(x\). 
  Let \(dx^2\) denote the usual Riemannian metric on the circle. 
 Suppose we change the metric by multiplying by a smooth, 
 positive function \(h \in C^\infty(\T)_+\). The new metric is 
 \(h(x)\cdot dx^2\), it assigns length \(\sqrt{h (x)}\) to the tangent vector 
 \(\frac{\partial}{\partial x}\) at \(x\in \T\). The 
 Dirac operator associated to this
 Riemannian metric and the standard orientation on the circle is then 
 \[D_h :=  h(x)^{-\frac{1}{2}} \cdot D,\;\;\;\; \textup{where}\;\;\; D = -i\frac{\partial}{\partial x}.\]
 \(D_h\) is a self-adjoint elliptic operator on \(L^2(\T, \mu_h)\) where 
 \(\mu_h\) is the measure \(\sqrt{h} \cdot dx\). 
 
 The group \(G\) of diffeomorphisms of \(\T\) acts on the space of Riemannian 
 metrics on \(\T\), and \(g\in G\) pulls the Riemannian metric 
 \( h(x) \cdot dx^2\) to the Riemannian metric 
 \( h( g^{-1}x)g'(g^{-1}x) \cdot dx^2\). 
 Let 
 \[ (U_g\xi )(x) = \xi (g^{-1}x).\]
 Then, for any \(h\), 
  \(U\) extends to a unitary isomorphism 
 \(U\colon L^2(\T, \mu_h) \to L^2(\T , \mu_{g(h)} ),\)
 and 
 \( U_g D_hU_g^* = D_{g(h)},\)
 where
 \[ g(h) (x) = h( g^{-1}x) \cdot g'(g^{-1}x)  ,\]
 corresponding to the pulled-back metric (discussed above), 
 \(\mu\) is Lebesgue measure. 
  
We have described, 
  for every \(h\in C^\infty (\T)_+\), an odd spectral triple 
 \[ (L^2(\T, \mu_h), \pi, D_h),\]
where \(\pi \colon C(\T) \to \Bound \bigl( L^2(\T, \mu_h)\bigr)\) is 
the representation by 
 multiplication operators. 
Furthermore, as observed above,
\(U_g\colon L^2(\T, \mu_h)\to L^2(\T, g_*(\mu_h))\) is a unitary isomorphism
conjugating \(D_h\) to \(D_{g(h)}\).

We now return to the simplicial space \(\EG\). We aim to build a 
\(\G\)-equivariant bundle of spectral triples (for \(C(\T)\) over \(\EG\). 
Start with the vertices \(\EG^{(0)}\). It is a discrete and \(G\)-finite 
discrete space, and divides into \(G\)-orbits. An orbit identifies with 
\(G/H\) where \(H\subset G\) is a finite subgroup of \(G\). Choose any 
\(h \in C^\infty(\T)_+\), determining a Riemannian metric on \(\T\), and 
average it over \(H\) to get an \(H\)-invariant Riemannian metric; in this way
we can assume \(h\in C^\infty(\T)\) gives an \(H\)-invariant metric to begin with. 
This implies that the Dirac operator \(D_h\) is \(H\)-fixed under the translation 
action of \(H\) on \(\T\). We then obtain the family 
\( \{D_{g(h)}\}_{gH \in G/H}\) of spectral triples over \(\T\), so that the 
unitary maps \(U_g\) given by elements of \(g\in G\) 
intertwine them exactly. 

In this way we define our bundle over 
 the vertex set \(\EG^{(0)}\). As we have discussed, the 
 spectral triple at any point is determined completely by the Riemannian 
 metric corresponding to that point, and the conjugation action of \(G\) 
 corresponds to the usual action of \(G\) on Riemannian metrics. The space 
 of metrics (or, equivalently, \(C^\infty(\T)_+\)), is convex. Hence 
 Using barycentric coordinates over a simplex whose vertices are assigned 
 points in a convex space, one can map the simplex into the space. In this 
 way, we obtain a \(G\)-equivariant map \(\EG\to C^\infty (\T)_+\) (where 
 the latter convex space has the \(G\)-action described above,) and hence a 
  family \(\{D_p\}_{p\in \EG}\) of spectral triples over \(C(\T)\), which, as a family, 
  is \(G\)-equivariant.

  This cycle represents 
\[ \pr_{  \T, \EG}! \in \RKK_{-1}^\G(\EG; C(\T), \C)\]
-- see Definition \ref{definition:transverse_Dirac_class} below. 

Note that if a point of \(\EG\) is fixed, corresponding to a given 
\(h \in C^\infty(\T)_+\), the associated spectral triple with operator 
\(D_h\), is homotopic in an obvious sense, to the ordinary Dirac 
operator \(D\) on \(\T\). Along the same lines,
 if the 
 group \(\G\) acts from the beginning by Riemannian isometries of the circle, 
  then one could just form the 
 constant family; this would give a homotopic cycle to the one just discussed.

\end{example}

The papers \cite{EM:Embeddings} and \cite{EM:Geometric_KK}, formalize the 
constructions of wrong-way morphisms \(f!\) in \(\KK^{\Grd}\), where 
\(\Grd\) is any proper groupoid, one of the main motivations being to apply it to 
the current situation. The morphism we have denoted 
\[ \pr_{  \T, \EG}! \in \RKK_{-1}^\G(\EG; C(\T), \C) = \KK^{\Grd_\G}(C_0(\EG\times \T), C_0(\EG))\]
in the above Example is an instance.

The groupoid-equivariant theory of wrong-way maps and
 correspondences 
  is a bundle version, roughly speaking, of the usual theory. 
A \emph{smooth \(\Grd\)-manifold} \(X\) 
is, by definition,
 a bundle of smooth manifolds 
over the base \(\Grd^0\) of \(\G\). Morephisms in \(\Grd\) act by diffeomorphisms 
between fibres. There is a natural notion of smooth \(\Grd\)-equivariant map 
between two \(\Grd\)-manifolds, meaning that it is fibrewise smooth (and equivariant). 
One may similarly define \(\Grd\)-equivariant \(\K\)-orientations, prove the 
\(\Grd_\G\)-equivariant Thom Isomorphism (\cite{LeGall}) and 
\cite{EM:Embeddings},\cite{EM:Geometric_KK} develop a purely topological 
category of \(\Grd\)-equivariant correspondences based on 
these \(\Grd\)-manifolds, and smooth \(\Grd\)-maps between them.

The wrong way elements 
\(f!\in  \KK^\Grd_{\dim Y- \dim (X)} (C_0(X), C_0(Y))\), for smooth 
\(\K\)-oriented \(\Grd\)-maps \(f\) between \(\Grd\)-manifolds \(X\) and \(Y\), 
that we describe 
in this theory, are based on Atiyah's topological index, and use only 
equivariant (fibrewise, equivariant ) 
Thom isomorphisms and (fibrewise, equivariant) open embeddings; the corresponding 
analytic constructions of such morphisms have been studied in the literature for 
a long time, of course, going back to \cite{Connes-Skandalis}. 
 The reference \cite{EM:Dualities} discusses analytic construction of 
 shriek maps in some detail, especially in the groupoid-equivariant setting. 
 The last section of \cite{EM:Geometric_KK} discusses analytic shriek maps 
 and the comparison to topological ones, also in the equivariant setting.

\begin{remark}
In the theory of topological correspondences, in order to define 
wrong-way elements \(f!\in  \KK^\Grd_{\dim Y- \dim (X)} (C_0(X), C_0(Y))\)
where \(\Grd\) is a proper groupoid, we need more in the way of hypotheses. 
What is needed is that \(X\) embeds equivariantly into the 
total space of a \(\Grd\)-equivariant vector bundle over \(\Grd^0\). This is 
analogous to the existence of an embedding of any manifold in a Euclidean space, but 
the equivariance condition in the groupoid-equivariant context is quite strong. 
There are simple examples of proper groupoids which do not embed equivariantly 
in a vector bundle in this sense. (See \cite{EM:Embeddings}.) However, such embeddings 
are guaranteed if \(\Grd\) has a `full vector bundle' over it's base (see 
\cite{EM:Embeddings}) and this is the 
case if \(\Grd = \G\ltimes \EG\) for a discrete group \(\G\) with \emph{co-compact }
\(\EG\), the classifying space for proper actions, 
by a result of L\"uck and Oliver. 
More generally, it is proved in \cite{EM:Geometric_KK} that with the hypothesis, that 
\(Z\) has a full equivariant vector bundle over its base, all three theories: that 
of \(\Grd\)-equivariant correspondences using normally non-singular maps, 
smooth \(\Grd\)-equivariant correspondences (without requiring normal data) ,
and the analytically defined theory \(\KK^\Grd\) of LeGall all give the 
same theory, for smooth \(\Grd\)-manifolds. 

In order to avoid any pathologies, then, we will put a blanket assumption on the 
discrete groups \(\G\) studied in this paper, that \(\G\backslash \EG\) is 
compact. As above, then, the
 transformation groupoid \(\Grd_\G := \G\rtimes \EG\)
is proper and has a full vector bundle on its base, and the above results apply. 
\end{remark}

\begin{definition}
\label{definition:equivariant_orientations}
Let \(X\) a locally compact space with a continuous
 action of the discrete group \(\G\) with \(\G\)-compact 
 \(\EG\). Let 
 \(\Grd_\G:= \EG \rtimes \G\) the corresponding (proper) transformation groupoid.
 a \emph{\(\KK\)-orientation} on the \(\G\)-action on \(X\), 
 is an endowment of 
 the \(\Grd_\G\)-space \(\EG\times \G\) with the structure of a 
 smooth, \(\Grd_\G\)-equivariantly \(\K\)-orientable \(\Grd_\G\)-manifold. 

\end{definition}

 The \(\K\)-orientation assumption posits a \(\Grd_\G\)-equivariant bundle of 
 \(\K\)-orientations on the vertical tangent bundle of the projection \(\EG\times X\to \EG\). Thus, 
a \(\KK\)-orientation in our sense is not a single fixed \(\K\)-orientation, but 
an equivariant bundle of them, over \(\EG\).

The smoothness assumption 
means that we are given a bundle, parameterized 
by the points of \(\EG\), of smooth structures on \(X\), such that a group element 
\(g\in \G\) acts as a diffeomorphism between \(X\) with the smooth structure 
assigned to \(p\in \EG\), to \(X\) with the smooth structure assigned to \(g(p)\).

 In particular, this 
structure endows \(X\) with a manifold structure by including it as a fibre in 
\(\EG\times X\); however, it is not necessarily true that this smooth 
structure is constant as one moves from one fibre to another. This is because the 
 \(\G\)-action on \(X\) can actually fail to be smooth, even if the bundle is smooth 
(boundary actions of hyperbolic groups provide an example, see
 Lemma \ref{lemma:bundlesmooth} and discussion around it).

 If one excludes these examples, 
 there is no reason for purposes of this paper not to assume that the 
 \(\G\)-action on \(X\) was smooth to begin with; then one can of course 
 given \(\EG\times X\) the constant fibrewise smooth structure.

For \emph{compact} groups, \(\EG\) can be taken to be a point, and Definition 
\ref{definition:equivariant_orientations} reduces to the standard 
definition of a smooth, equivariantly (\(\K\)-)oriented manifold \(X\).

\begin{definition}
\label{definition:transverse_Dirac_class}
Let \(\G\) be a discrete group and \(X\) a \(\G\)-equivariantly 
\(\KK\)-orientable manifold
(Definition \ref{definition:equivariant_orientations}.)

A \emph{transverse Dirac class} for \(X, \G\) is a class \[[X] \in 
\KK^\G_{-n}(C_0(X), \C)\]
such that 
\[ \inflate ([X]) = \pr_{X , \EG}!\in \KK^{\Grd_\G}_{-n} \bigl(C_0(X\times \EG), C_0(\EG)\bigr),\]
where 
\(\pr_{X, \EG}!\in \KK^{\Grd_\G}_{-n}\bigl(C_0(X\times \EG), C_0(\EG)\bigr)\) is 
the class of the \(\Grd\)-equivariant fibrewise 
smooth and \(\Grd\)-equivariantly \(\K\)-oriented non-singular 
map \(\pr_{X, \EG}\colon X\times \EG \to \EG\) discussed above, and 
\(\inflate\) is as in \eqref{eq:inflation_map}.

\end{definition}

\begin{example}

If \(\G\) is the trivial group,
 then the class \([X]\in \KK_{-n}(C_0(X), \C) = \K_{n}(X)\) 
of the Dirac operator on \(X\) is a transverse Dirac class.

If \(X\) is a point, \(\G\) an arbitrary locally compact group, then 
the the class \(1  \in \KK^\G_0(\C, \C)\) of the trivial representation of \(\G\) 
is a transverse Dirac class for the action of 
\(\G\) on a point.  

\end{example}

\begin{example}
\label{example:transverse_Dirac_classes_for_propositioner_actions}
If \(\G\) is compact, so \(\EG\) is a point, then one definition of a 
\(\KK\)-orientable action of \(\G\) on \(X\) boils down to the usual 
assumption of a \(\G\)-equviariantly \(\K\)-oriented manifold, and the 
transverse Dirac class \([X]\in \KK^\G_{-\dim X} (C_0(X), \C)\) is 
represented by the usual Dirac cycle discussed in the first paragraph of 
this paper. 

More generally, if \(\G\) acts smoothly and 
\emph{properly} on \(X\), preserving a \(\K\)-orientation in the 
sense of the discussion at the beginning of the paper, then as 
remarked there one can directly construct  a 
\(\G\)-equivariant Dirac operator and class 
\([D_X]\in \KK^\G_{-n}(C_0(X), \C)\). 
Inflating this class gives the cycle for 
\(\RKK^\G_{-n} (\EG; C_0(X), \C)\) consisting of the (constant) bundle 
of Dirac operators along the fibres of \(X\) with respect to the 
(constant) bundle of \(\K\)-orientations. 

 The equality 
\[ \inflate([D_X]) = p_{\EG, X}!\]
and the consequent one that \([D_X]\) is the transverse Dirac class, 
follows from the basic index theorem of Kasparov theory (that the 
analytically and topologically defined shriek maps coincide.)  

Note that in the case of a proper action, the inflation map 
\[ \inflate\colon \KK^\G_*(C_0(X), \C) \to \KK^{\Grd_\G}_*(C_0(X\times \EG), C_0(\EG))
= \RKK^\G_*(\EG; C_0(X), \C)\]
is an isomorphism. Therefore, 
transverse Dirac classes exist and are unique in these cases. 
The same holds if the action is merely amenable. 

\end{example}

\begin{remark}
Hilsum and Skandalis in \cite{HS} define equivariant \(\K\)-orientability of a smooth but 
potentially non-isometric group \(\G\) action on a smooth manifold \(X\), 
using the metaplectic group in place of the 
spin group, and they construct a corresponding analytic morphism 
\[ [X]\in \KK^\G_{-\dim X} (C_0(X), \C)\]
using hypo-elliptic operator theory and a frame bundle construction. It 
would see to us almost certainly true that Hilsum-Skandalis \(\K\)-orientability 
(of a smooth action) implies \(\KK\)-orientability in our sense, but we do not 
address the proof in this article, since it is not particularly 
material for us; we leave it as an open question as to whether transverse Dirac 
classes exist in full generality for smooth actions. 
\end{remark}

If \(\G\) has property \(T\) and a \(\gamma\)-element, then 
the class in \(\KK^\G_0(\C, \C)\) of the \(\gamma\)-element 
maps to \(1\in \RKK^\G_0(\EG; \C, \C)\) under the inflation map, 
and hence the \(\gamma\)-element is a transverse Dirac class for 
\(\G\) acting on a point; however, it is not equal to the class \([\epsilon] \in \KK^\G_0(\C, \C)\) of 
the class of the trivial representation of \(\G\). Thus there is more than one transverse 
Dirac class for a property \(T\) group \(\G\) acting on a point.

\begin{lemma}
\label{lemma:dfahskdfja;sdkfj}
Let \(X\) carry a \(\KK\)-orientable action of \(\Gamma\). Then
the \(\Gamma\)-action on 
\(X\times Z\) acquires an induced \(\KK\)-orientation, 
\(X\times Z\) has a transverse Dirac class \([X\times Z]\), and 
\[ \PD( \pr_{Z, X}! ) = [X\times Z] \in \KK_{-n-d}^\G(C_0(X\times Z), \C),\]
where \(\PD\) is Kasparov's Poincar\'e duality. 
\end{lemma}

This follows from the description of 
\(\PD\) given in the discussion preceding Proposition \ref{proposition:factoring_Dirac}.

\begin{proposition}
\label{proposition:Dirac_and_transverse}
If \( (\G, Z)\) is a smooth, oriented group, and \(\Gamma\) acts 
\(\KK\)-orientably on \(X\), and if  \([X]\in \KK_{-n}^\G(C_0(X), \C)\) 
is a transverse Dirac class for \(X\) and 
\([X\times Z]\) \emph{the} transverse Dirac class for the \(\G\)-manifold 
 \(X\times Z\), 
then 
\[ \Dirac ([X]) = [X\times Z] \in \KK^\G_{-n-d}(C_0(X\times Z), \C).\] 
where \(\Dirac\) is the localization map.  

Moreover, if  \(\G\) is torsion-free, then 
\(\KK^\G_*(C_0(X\times Z), \C)\cong \K_*(X\times_\G Z)\) and with this identification 
\[ \Dirac ([X]) = [X\times_\G Z]\in \K_{n+d}(X\times_\G Z),\]
where \([X\times_\G Z]\) is the 
class of the Dirac operator on \(X\times_\G Z\) (the transverse Dirac 
class of the trivial group acting on \(X\times_\G Z\).) 
\end{proposition}

Since \(\Dirac ([X]) = [X]\otimes_\C [Z]\), the external product in 
\(\KK^\G\), the Proposition is the multiplicativity statement for 
transverse Dirac classes 
\begin{equation}
\label{equation:multiplicitivityofdirac}
 [X]\otimes_\C [Z] = [X\times Z]
 \end{equation}
when \(X\) has a transverse Dirac class.

\begin{proof}

We use the inflation map 
\[ \inflate\colon \KK^\G_*(C_0(X\times Z), \cdot) \to \RKK^\G_*(\EG; C_0(X\times Z), \C) \cong \KK^{\Grd_\Gamma}_*(C_0(\EG\times W), C_0(\EG))\]
in this argument, which is an isomorphism since \(X\times Z\) is proper. 

By the discussion in Example \ref{example:transverse_Dirac_classes_for_propositioner_actions}, 
the Dirac 
class we have always denoted \([Z]\in \KK_{-d}(C_0(Z), \C)\), is the same as the 
transverse Dirac class, thus
\(\inflate ([Z]) = \pr_{\EG, Z}!\),  while 
\(\inflate ([X]) = \pr_{\EG, X}!\) for any transverse Dirac class 
\([X]\). Since inflation is an isomorphism respecting external products, 
 \eqref{equation:multiplicitivityofdirac} is equivalent to the 
identity 
\[ \pr_{\EG, X}!\otimes_\C \pr_{\EG, Z}!= \pr_{\EG, X\times Z}!,\]
which is the basic multiplicativity property of shriek maps. 

\end{proof}

The converse holds as well, which is sometimes useful for identifying when 
a given construction has produced a transverse Dirac class \([X]\), since 
in practical terms, external products like 
\([X]\otimes_\C [Z]\) can be described concretely if the constituents are 
described as unbounded cycles. 

\begin{proposition}
\label{corollary:characterization_of_Dirac_classes}
A class \([X]\in \KK^\G_{-n} (C_0(X), \C)\) is a transverse Dirac class if and only if 
\[ [X]\otimes_\C [Z] = [X\times Z]\in \KK^\G_{-n - d} (C_0(X\times Z), \C).\]
 where \([X\times Z]\) is the transverse Dirac class of the proper space \(X\times Z\) (and the 
 left-hand side is the external product in \(\KK^\G\).) 
\end{proposition}

\begin{proof}
Suppose \([X]\in \KK^\G_{-n}(C_0(X), \C)\) satisfies
\[ [X]\otimes_\C [Z] = [X\times Z],\]
that is, suppose that \(\Dirac (]X]) = [X\times Z]\). I claim that 
\(\inflate ([X]) = \pr_{\EG, X}!\). Applying Poincar\'e duality to both sides of this and using 
the commutative 
diagram \eqref{equation:factoring_the_dirac_map} we see that what a wish to prove is 
equivalent to the statement that 
\[ [X\times Z] = \PD (\pr_{\EG, X}!) \in \KK_{-n-d}(C_0(X\times Z), \C).\]
which is the content of Lemma \ref{lemma:dfahskdfja;sdkfj}.

\end{proof}

\section{Dirac classes for actions}

In this section we simultaneously generalize transverse Dirac classes 
for manifolds and Dirac classes for discrete groups, to define Dirac classes 
for a certain class of smooth discrete actions, by requiring that they 
are constituted in a certain sense by 
splicing together a 
Dirac class for the group, and a transverse direct class for the manifold.

\subsection{Definition of the Dirac class}
The following is the main definition of this paper. 

\begin{definition}
\label{definition:Dirac_class}
Let \((\G, Z)\) be a smooth, \(d\)-dimensional \(\K\)-oriented group, and let 
\(\G\) act \(\KK\)-orientably on 
the smooth compact oriented \(n\)-dimensional 
manifold \(X\). 

Let 
\([X]\in \KK^\G_{-n}(C(X), \C)\) be a transverse Dirac 
class for the action. 

Then a
 \emph{Dirac class} for \(\G\ltimes X\) is any 
class 
\[ \fund\in \KK^\G_{d - n} (C_0(X), \C)\]
such that 
\[ \Dirac (\fund) = [\ev]\otimes_\C [X] \in \KK^\G_{-n}(C_0(Z\times X), \C),\]
where \(\Dirac\) is the localization map, \([\ev]\) as in \eqref{equation:point_evaluation_class}.

A \emph{reduced} Dirac class is a class 
\(\fund  \in \K^{ d-n }(C(X)\rtimes_r \G)\) 
which pulls back to a Dirac class under the map on \(\K\)-homology induced by the projection 
\(C(X)\rtimes \G \to C(X)\rtimes_r \G\). 
\end{definition}

If \(\G\) is torsion-free, 
\(\KK^\G_*(C_0(Z\times X), \C) \cong \K_{-*}(Z\times_\G X)\), where 
\(Z\times_\G X\) denotes the quotient of \(Z\times X\) 
by the diagonal action of \(\G\). The smooth manifold \(Z\times_\G X\) is foliated into 
the images of the slices \(Z\times \{x\}\), for \(x\in X\), and, \(Z\times_\G X\) is also 
a bundle of compact manifolds over \(\G\backslash Z\) 
under the coordinate projection 
\[p\colon Z\times_\G X \to \G\backslash Z.\] 
The submanifold \(p^{-1}(\G z_0)\) is a closed transversal \(X_{z_0}\) 
for the foliation, naturally 
diffeomorphic to \(X\), 
for any \(z_0\in Z\).

Let 
\( \ezx\colon X\to Z\times_\G X\) be the corresponding embedding. 
The class 
\[ [\ev]\otimes_\C [X]\in \KK^\G_{-n}(C_0(Z\times X), \C) 
\cong \K_{n}(Z\times_\G X)\] is, by the definitions, equal to
 the class of the 
Baum-Douglas 
cycle 
\begin{equation}
\label{equation:bd_cycle_for_a_point}
 Z\times_\G X \xleftarrow{ \ezx}X\rightarrow \pnt\end{equation}
  from 
\(Z\times_\G X\) to a point, obtained by mapping the \(\K\)-oriented compact manifold 
\(X\) into \(Z\times_\G X\) as a transversal to the foliation described above. 

\begin{proposition}
\label{proposition:transversal_inthemappingcylinder}
Suppose that 
\(\G\) is torsion-free, and \(X\) is smooth and \(\G\)-equivariantly \(\K\)-orientable, 
 Let \(\ezx^*\in \KK_0(C(Z\times_\G X), C(X))\) be the 
class of the *-homomorphism Gelfand dual to the inclusion of  
 \(X\) into \(Z\times_\G X\) as a 
fibre of \(Z\times_\G X \to \G\backslash Z\). 
Let 
\([X]\in \KK^\G_{-n}(X)\) be a 
transverse Dirac class of \(X\) and \([X]'\in \KK_{-n}(C_0(X), \C)\) the 
ordinary Dirac class of \(X\), obtained by forgetting the \(G\)-action. 

Then
 a Dirac class for \(\G\ltimes X\) is any class mapping to 
 \[   \ezx^* \otimes_{C(X)} [X]'\in \KK_{-n}(C_0(Z\times_\G X), \C) = \K_{n} (Z\times_\G X),\] 
 under the composition of the localization map and the 
natural isomorphism 
\( \KK_{-n}(C_0(Z\times_\G X), \C) \cong \K_{n} (Z\times_\G X)\).
\end{proposition}

\begin{proof}
Since \(\G\) itself is a zero-dimensional, equivariantly \(\K\)-oriented \(\G\)-manifold, 
it has a transverse Dirac class \([\G]\in \KK_0^\G(C_0(\G), \C)\), and if 
\(o_{z_0}\colon \G \to Z\) is the orbit map at \(z_0\), then 
\[ [\ev]  = o_{z_0}^*([\G]) \in \KK_0^\G(C_0(Z), \C).\]
is clear. By an obvious case of multiplicativity of transverse Dirac classes, 
\[ [\ev]\otimes_\C [X] = (o_{z_0}\times \id_X)^*([\G\times X])\in \KK^\G_{-n} (C_0(Z\times X), \C),\]
where \([\G\times X]\) is the transverse Dirac classes of the (proper, equivariantly 
\(\K\)-oriented) manifold \(\G\times X\). Applying the descent functor gives the result, 
since \(o_{z_0}\colon \G \to Z\) descends to the inclusion of the point 
\(\G z_0\) in \(\G\backslash Z\). 

\end{proof}


The analogue of proposition \ref{proposition:dual_Dirac_implies_Dirac} holds for 
groupoids as well. Namely, if 
If \( (\G, Z)\) is a smooth oriented group, \(X\) an equivariantly 
oriented \(\G\)-manifold, then a Dirac class \([\widehat{\G\ltimes X}]\) exists
as soon as the groupoid 
\(\G\ltimes X\) has a dual-Dirac morphism 
\( \eta \in \KK_d^{\G \ltimes X}(C_0(X), C_0(Z\times X))\); this is the 
case if \(\G\) itself has one. If the groupoid 
\(\G\otimes X\) has a dual-Dirac morphism, then the 
 Dirac class for \(\G\ltimes X\) 
is unique if  
\(\gamma_{\G\ltimes X} =1_{\G\ltimes X} \in
 \KK^{\G\ltimes X}_0(C_0(X), C_0(X))\), where \(\gamma_{\G\ltimes X}\) is the 
 corresponding \(\gamma\)-element for the groupoid \(\G\ltimes X\). This occurs if 
  \(X\) is a proper, or more generally, an amenable 
\(\G\)-space.

In the latter case, if \(X\) is also compact, then \(\G\) itself has a 
\(\gamma\)-element (and not just the groupoid \(\G\ltimes X\)). 
More generally, if 
\((\G, Z)\) is a smooth, oriented group, \(X\) a smooth oriented 
\(\G\)-manifold, and if  \(\eta \in \KK_{d}^\G(\C, C_0(Z))\) is a dual-Dirac morphism for 
\(\G\), then the groupoid \(\G\ltimes X\) also has a dual-Dirac morphism and 
corresponding Dirac class, given by the formula 
\begin{equation}
\label{equation:Dirac_equation}
 [\widehat{\Gamma\ltimes X}] = (\eta \otimes_\C 1_{C(X)})\otimes_{C_0(Z\times X)} ([\ev]\otimes_\C [X]).
 \end{equation}

More generally, existence of a Dirac class for \(\G\) implies one for 
any action. 

\begin{proposition}
\label{proposition:Dirac_classes_from_external_products}
If \((\G, Z)\) is a smooth oriented group, \([\widehat{\G}] 
\in \KK^\G_{-d}(\C, \C)\) is
 a Dirac class for \(\G\), and if \([X] \in \KK^\G_{-n}(C_0(X), \C)\) is a 
transverse Dirac class then 
\[ [\widehat{\G}] \otimes_\C  [X]\in \KK^\G_{d-n}(C_0(X), \C)\]
is a Dirac class for \(\G\ltimes X\), where the Kasparov product is the external 
product in \(\KK^\G\). 
\end{proposition}


\begin{proof}
The proof is a trivial consequence of associativity of the Kasparov product. 

\end{proof}

\subsection{Groups of nonpositive curvature acting by Riemannian isometries}

 Suppose 
\((\G, Z)\) is a smooth \(\K\)-oriented group with \(Z\) carrying a \(\G\)-invariant 
metric of nonpositive curvature, and suppose that \(\G\) acts by Riemannian 
isometries of \(X\) preserving a \(\K\)-orientation. The Dirac class of \(\G\) is 
represented by the spectral triple 
{\( (l^2(\G, V), \pi, \delta)\) 
of Theorem \ref{theorem:Dirac_class_for_nonpositively_curved_groups}. It is finitely 
summable if \(\G\) has polynomial group (not otherwise). 

Since the \(\G\)-action is assumed isometric on \(X\), the transverse Dirac class 
is also represented by a spectral triple, of the type \( ( L^2(S), \pi, D_X)\), where 
\(D_X\) is the \(\G\)-equivariant Dirac operator associated to the equivariant 
metric and \(\K\)-orientation.

\begin{corollary}
\label{corollary:Dirac_class_for_non_positively_curved_groups_acting_isometrically}
If \( (\G, Z)\) is a smooth \(\K\)-oriented group with \(Z\) carrying a metric of 
nonpositive curvature, and if \(\G\) acts \emph{isometrically} and preserving a 
\(\G\)-orientation on a Riemannian manifold \(X\), then, in the notation above,  
the Dirac class for \(\G\ltimes X\) is represented by the spectral triple 
\( (L^2(S)\hat{\otimes} l^2(\G, V), D_X\hat{\otimes} 1 + 1\hat{\otimes} \delta)\). 

In particular, if \(\G\) has polynomial growth \(\sim n^d\) 
then the (reduced) Dirac class \(\fund\) is represented by a 
\(\dim (X) +d\)-summable spectral triple over \(C_0(X)\rtimes_r \G\).

\end{corollary}

\begin{proof}
The proof is an immediate consequence of Proposition 
\ref{proposition:Dirac_classes_from_external_products} and the standard recipe for 
taking external products of unbounded cycles in \(\KK^\G\). [** include reference ]

\end{proof}

\begin{example}
\label{example:dirac_of_finite_group_action}
The Dirac class of a \(\K\)-orientation preserving action of a \emph{finite} 
group \(\G\) on \(X\) is represented by the spectral triple 
 \( (L^2(S)\otimes l^2(\G), D_X\hat{\otimes} 1)\), with 
\(D_X\) the \(\G\)-equivariant Dirac operator on \(X\), since \(Z\) is a 
point in this case. 

\end{example}

\subsection{Irrational rotation}

The Dirac class of the irrational rotation algebra \(A_\theta := C(\T)\times_\theta \Z\) is 
represented by a spectral triple first described by A. Connes: let 
\(\tau \colon A_\theta \to \C\) be the standard trace, \(L^2(A_\theta)\) the 
GNS Hilbert space associated to \(\tau\). The algebra \(A_\theta\) is represented
on \(L^2(A_\theta)\) by left multiplication, and the two densely defined derivations 
\[ \delta_1 (\sum_{n\in \Z}  f_n [n]) 
:= \sum_{n\in \Z}  f_n'[n],\;\;\; 
\delta_2 (\sum_{n\in \Z}  f_n[n])  :=  \sum_{n\in \Z} n f_n[n],\]
 using group-algebra notation.  They assemble to a densely defined self-adjoint 
 operator 
 \[ \dol_\theta:= \begin{bmatrix} 0 & \delta_1 -i\delta_2\\\delta_1+i\delta_2 & 0 \end{bmatrix},\]
on \(L^2(A_\theta)\oplus L^2(A_\theta)\), 
 a deformation of the Dolbeault operator \(\dol_{\T^2}\) on \(\T^2\).

Note that, \emph{as a (unbounded) operator on a Hilbert space}, 
\(\dol_\theta\) is absolutely identical to the 
ordinary Dolbeault operator \(\dol_{\T^2}\) operating with its usual initial domain of 
smooth functions in
 the graded Hilbert space \(L^2(\T^2)\oplus L^2(\T^2)\). 
Thus as far as the \emph{operator} is concerned, we are dealing with a classical operator. 
However, the dynamics is encoded by the representation of \(A_\theta\), so that the 
noncommutative aspect of this spectral triple lies entirely in the representation of 
\(A_\theta\) involved in it by the GNS procedure. To be explicit, identifying 
\(L^2(A_\theta)\) with \(L^2(\T^2) \cong l^2(\Z^2)\) then 
the representation of 
\(\pi \colon A_\theta = C(\T)\rtimes \Z \to \Bound\bigl( L^2(\T^2)\oplus L^2(\T^2)\bigr), \)
is the one determined by the covariant pair 
\[ \pi (f) \xi (x,y) = f(x)\xi(x,y), \;\;\; ( n\cdot \xi) (x,y) =  e^{2\pi i ny}  \, \xi\bigl( R_\theta^{-n}(x), y).\]

\begin{remark} (\emph{\(\Z\)-actions}.)
More generally, suppose that \(\Z\) acts on a complete Riemannian manifold \(X\), 
isometrically, and preserving a \(\K\)-orientation in the sense that there is a 
Hermitian \(\Z\)-equivariant spinor bundle \(S_X\) over \(X\), graded or ungraded 
or \(p\)-multigraded (if one is working over the reals), and a \(\Z\)-equivariant 
connection on \(S_X\) and associated \(\Z\)-equivariant Dirac operator 
\[ D_X\colon L^2(S_X) \to L^2(S_X)\]
determining the transverse Dirac class \([X]\in \KK^\Z_{-n}(C_0(X), \C)\).

Following the recipe of Corollary \ref{corollary:Dirac_class_for_non_positively_curved_groups_acting_isometrically}
we form an external product and get the following explicit representative of 
\([\widehat{\Z\ltimes X}] \in \KK_{-n+1} (C_0(X)\rtimes \Z, \C)\). We assume 
for simplicity that \(\dim X\) is odd, and that \(S_X\) is ungraded, \(D_X\) self-adjoint. 
Then the Hilbert space of our spectral triple consists of two copies of
 \(L^2(S_X)\otimes l^2(\Z) \) with standard even grading, and the operator 
 \[ \begin{bmatrix} 0 & D_X\otimes 1 + i (1\otimes \delta) \\
  D_X\otimes 1 - i (1\otimes \delta) & 0 \end{bmatrix}\]
  where \(\delta\) is the number operator on \(l^2(\Z)\). The representation of 
  \(C(X)\rtimes \Z\) on this Hilbert space is by 
  \[ f(\xi \otimes e_k) := f\xi \otimes e_k, \;\;\; m(\xi\otimes e_k) := m(\xi)\otimes e_{k+m}.\]

This of course directly generalizes the irrational rotation example.

\end{remark}

We next consider an action of  \(\Z^2\)  preserving a \(\K\)-orientation on the 
compact manifold \(X\) of dimension \(n\).

If \(n\) is odd, the spinor bundle 
\(S_X\) for \(X\) is ungraded; let \(D_X\) be the corresponding Dirac operator. 
Let \(\dol\) be the Dolbeault operator on the \(\Z/2\)-graded Hilbert space 
\(L^2(\T^2) \oplus L^2(\T^2)\). Then the Dirac class of the action is represented by 
the following odd-dimensional spectral triple over \(C(X)\rtimes \Z^2\). The 
Hilbert space is  
 \(L^2(\T^2)\otimes L^2(S_X) \oplus L^2(\T^2)\otimes L^2(S_X)\)
with \emph{no} grading, and the operator with respect to this decomposition is 
\begin{equation}
\label{equation:Dirac_Zd_on_tori_n_odd}
 \begin{bmatrix} 0 & \dol_{\T^2}\otimes 1 + i \, (1\otimes D_X) \\
 \dol_{\T^2}\otimes 1 - i \, (1\otimes D_X) & 0 \end{bmatrix}.
 \end{equation}
 The action of \(C(X)\) is by letting \(f\in C(X)\) act by a multiplication operator 
 in the \(L^2(S_X)\) factor. The group \(\Z^2\) acts diagonally, with the 
 implicitly assumed unitary action on sections of the spinor bundle \(S_X\), and 
 the action of \(\Z^2\) on \(L^2(\T^2)\) given by the Fourier transform 
 \[ (g\cdot \xi ) (\chi) := \chi (g) \xi (\chi), \;\;\; \chi \in \widehat{\Z^2} \cong \T^2.\]

This spectral triple is clearly \(2+n\)-summable. 

If \(n\) is even, \(S_X\) graded into \(S_X^0\oplus S_X^1\), then the Dirac 
class is represented by the odd operator 
\begin{equation}
\label{equation:Dirac_Zd_on_tori_n_even}
 \begin{bmatrix} 0 & \dol_{\T^2}\otimes 1 + 1\otimes D_X^1\\
\dol_{\T^2}\otimes 1 + 1\otimes D_X^0 & 0 \end{bmatrix}
\end{equation}
on the \(\Z/2\)-graded Hilbert space 
with even part 
\[ L^2(\T^2)\otimes L^2(S_X^0)\, \oplus \, L^2(\T^2)\otimes L^2(S_X^1)\]
and odd part 
\[L^2(\T^2)\otimes L^2(S_X^1)\, \oplus \, L^2(\T^2)\otimes L^2(S_X^0).\]
Here 
\(D_X^0:= D_X|_{L^2(S_X^0)}\colon L^2(S_X^0) \to L^2(S_X^1)\), and 
\(D_X^1:= D_X|_{L^2(S_X^1)}\colon L^2(S_X^1) \to L^2(S_X^0)\).

The action of \(C(X)\rtimes \Z^2\) is as before. 

This Dirac spectral triple is \(n+2\)-summable, of course.

\subsection{Dirac classes and inflation}

We aim to compute 
\[\inflate (\fund) \in \RKK^\G_{-d-n} (Z; C_0(X), \C),\]
equivalently, of 
 \(\PD^{-1}([\ev]\otimes_\C [X]) \)
where, as usual, \((\G, Z)\) is a \(\K\)-oriented group, acting 
\(\KK\)-orientably on \(X\). Assume \([X]\) is a transverse Dirac class for the action, 
and \(\fund\) a Dirac class. 
Recall that 
\[ \RKK^\G_{-d-n} (Z; C_0(X), \C)\cong \KK^{\Grd_\G}_*(C_0(Z\times X), C_0(Z)),\]
so that we might expect a description of 
\(\inflate (\fund)\) in the form of a \(\Grd_\G\)-equivariant correspondence from 
\(Z\times X \) to \(Z\) -- that is, a bundle of Baum-Douglas cycles for \(X\), 
parameterized by the points of \(Z\). As it turns out, \(\inflate (\fund)\) is represented 
by a piece of geometric data which generalizes slightly the data involved in a 
\(\Grd_\G\)-equivariant correspondence in the sense of \cite{EM:Geometric_KK}. We 
will call it a `(smooth, \(\Grd_\G\)-equivariant...) correspondence with 
singular support.' 

Suppose that \(i_M \colon M \to Z\) is a closed, smooth, 
\(\G\)-invariant, co-dimension \(d\) submanifold of \(Z\) with \(\G\)-equivariantly 
\(\K\)-oriented normal bundle \(\pi \colon \nu  \to Z\). Let \(\varphi \colon \nu \to Z\) 
the associated tubular neighbourhood embedding onto an open \(\G\)-invariant 
neighbourhood \(U\) of \(M\). The \(\G\)-equivariant \(\K\)-orientation on 
\(\nu\) determines a Thom class 
\[\xi_\nu \in \KK^{\Grd_\G}_{+d}(C_0(Z), C_0(Z))\cong \RKK^\G_{+d}(Z; \C, \C).\]
The corresponding cycle has Hilbert \(C_0(Z)\)-module sections \(C_0(U, S')\), where 
\(S'\) is the bundle over \(U\) defined by pulling back the 
the spinor bundle 
\(S\to M\) for \(\nu\), to a bundle over \(\nu \cong U\), and operator 
a Clifford multiplier 
\(F\) such that \(f (F^2-1)\) is a \(C_0\)-bundle endomorphism of \(S'\). In terms of 
correspondences, this is the class of the smooth, \(\Grd_\G\)-equivariant 
correspondence 
\begin{equation}
\label{equation:e8r9tu934hernkjfnvoieur}
 Z \leftarrow  (U, \varphi!( \xi_\nu)) \rightarrow  Z,\end{equation}
from \(Z\) to \(Z\), where the arrows designate the open inclusion. 

Since, however, the construction only depends on the \(\G\)-equivariant 
\(\K\)-oriented embedding \(i_M \colon M \to Z\), we will \emph{define} 
\begin{equation}
\label{equation:stupiddefinition}
 i^* \otimes_M i! \in \KK^{\Grd_\G}_{+d}(C_0(Z), C_0(Z)) = \RKK^\G_{+d}(Z; \C, \C)
 \end{equation}
to be the class of the correspondence \eqref{equation:e8r9tu934hernkjfnvoieur}, and 
refer to it as a correspondence \emph{with singular support}. We will 
also therefore also refer to a 
 diagram 
\[ Z\xleftarrow{i_M} M \xrightarrow{i_M} Z\]
a `correspondence with singular support'; the corresponding class is 
given by \eqref{equation:stupiddefinition}.

The entire discussion goes through if \(i_M \colon M \to Z\) is merely assumed a 
\(\G\)-equivariant, equivariantly \(\K\)-oriented immersion, as the reader may easily 
confirm.

The main reason we want to discuss correspondences with singular support is that 
the class we aim to describe has this form. 

Choose \(z_0 \in Z\), let \(\ez\colon \G \to Z\) be the orbit map at \(z_0\). 
It is a smooth, \(\G\)-equivariantly \(\K\)-oriented immersion. The normal bundle 
is \(\cong \ez^*(TZ)\), using the usual tubular neighbourhood 
embedding of the form  
\[ \varphi (g, \xi) := \exp_{gz_0}(\xi'),\]
where \(\xi \mapsto \xi'\) is an appropriate re-scaling of tangent vectors 
(\emph{e.g.} by \(\xi':= \frac{\epsilon \xi}{(1+|\xi |^2)^{\frac{1}{2}}}\)) into 
an open disk sub bundle of the tangent bundle, on which the Riemannian 
exponential map is a diffeomorphism. 

The \(\G\)-equivariant map \(\ez\colon \G \to Z\) gives \(\G\) the structure of a 
\(\Grd_\G\)-space, where \(\Grd_\G := \G \ltimes Z\) as before, and 
\begin{equation}
\label{equation:even_more_pissing_around_with_correspondences_and_duality_1}
Z \xleftarrow{\ez}  \G \xrightarrow{\ez} Z.
\end{equation}
is a \(\Grd\)-equivariant correspondence with singular support in the sense discussed above,  
from \(Z\) to \(Z\), 
with class
\[\ez^*\otimes_{C_0(\G)}\ez!\in 
\KK^{\Grd_\G}_{d}(C_0(Z), C_0(Z))  = \RKK^\G_{d} (Z; \C, \C) \]
where \(\ez^*\in \KK^{\Grd_\G}_0(C_0(Z), C_0(\G)\) is the 
class of the induced *-homomorphism. 

Similarly, if \(X\) is a smooth, compact manifold carrying a 
\(\KK\)-orientation preserving action of \(\G\), then 
\begin{equation}
\label{equation:transverseKclass}
 Z\times X \xleftarrow{\ez\times \id_X} \G\times X \xrightarrow{\ez \circ \pr_\G } Z
 \end{equation}
is a \(\Grd_\G\)-equivariant correspondence with singular support from \(Z\times X\) to \(Z\),
 with class
\[ (\ez\times \id_X)^*\otimes_{C_0(Z)} ( \ez\circ \pr_\G) ! \in \KK^{\Grd_\G}_{ d - n}(C_0(Z\times X), C_0(Z))
= \RKK^\G_{d - n} (Z; C_0(X), \C).\]


\begin{proposition}
\label{proposition:description_of_inflated_class}
If \( (\G, Z)\) is a smooth oriented group, \(X\) an equivariantly oriented \(\G\)-manifold, 
\([\widehat{\G\ltimes X}])\) a Dirac class for the action, then 
\[ \inflate([\widehat{\G\ltimes X}]) 
=  (\ez\times \id_X)^*\otimes_{C_0(Z)} ( \ez\circ \pr_\G) !
\in \RKK^\G_{d -n}(Z; C_0(X), \C).\]
\end{proposition}

\begin{proof}
We need to show that 
 \(\PD^{-1}([\ev]\otimes_\C [X]) 
  (\ez\times \id_X)^*\otimes_{C_0(Z)} ( \ez\circ \pr_\G) \), where 
\[\PD^{-1}\colon \KK^\G_*(C_0(Z\times X), \C) \xrightarrow{\cong} \KK^{\Grd_\G}_{*+d}(C_0(Z\times X), C_0(Z))  \cong \RKK^\G_{*+d} ( Z; C_0(X), \C)\] is Poincar\'e duality. 

The map \(\PD^{-1}\) 
is the composition of inflation
\[ \inflate\colon \KK^\G_*(C_0(Z)\otimes A, B)
\to \RKK^\G_*(Z; C_0(Z)\otimes A,B),\]
with Kasparov product in \(\RKK^\G(Z)\) with a class 
\( \Theta \in \RKK^\G_{d}(Z; \C, C_0(Z))\). The class \(\Theta\) is the 
class of the \(\Grd_\G\)-equivariant correspondence 
\begin{equation}
\label{equation:delta}
 Z \xleftarrow{\id} Z \xrightarrow{\delta} Z\times Z,
 \end{equation}
where the momentum map for the \(\Grd\)-space \(Z\times Z\) is 
in the first variable, and 
\(\delta_Z\colon Z \to Z\times Z\) be the diagonal map.

By definition, \(\inflate([X]) \in \RKK^\G_{-n}(Z; C(X), \C)\) is 
represented by the \(\Grd_\G\)-equivariant correspondence 
\[ Z\times X\xleftarrow{\id} Z\times X  \xrightarrow{\pr_Z}Z\]
and since \(\inflate\) is compatible with external products, 
\[\inflate([\ev]\otimes_\C [X]) = \inflate([\ev])\otimes_{Z, \C} \inflate([X])
\in \RKK^\G_{-n}(Z; C(X), \C)\] is 
represented by the \(\G\ltimes Z\)-equivariant correspondence (with singular support) 
\begin{equation}
\label{equation:pissing_around_with_correspondences_and_duality_3}
 Z\times Z\times X\xleftarrow{\id_Z\times \ez\times \id_X}Z\times \G\times X \xrightarrow{\pr_Z} Z.
 \end{equation}
 where the momentum map for \(Z\times Z\times X\) is projection to the first coordinate, 
 or, more precisely, 
 the class of the ordinary \(\Grd_\G\)-equivariant correspondence 
\begin{equation}
\label{equation:pissing_around_with_correspondences_and_duality_3}
 Z\times Z\times X\xleftarrow{\id_Z\times \varphi \times \id_X}   
 (Z\times \nu \times X, \xi_\nu)  \xrightarrow{\pr_Z} Z
 \end{equation}
with \(\nu \) the normal bundle to the orbit, that is \(\nu =  \ez^*(TZ)\), and 
\(\varphi \colon \nu \to Z\) the tubular neighbourhood embedding. Since 
\(\varphi \times \id_X\) is a submersion, one can compose 
\eqref{equation:pissing_around_with_correspondences_and_duality_3} with 
\eqref{equation:delta} by transversality. The result is bordant to 
 \[ Z\xleftarrow{\ez\circ \pi} ( \ez^*(TZ), \ez^*(\xi_{TZ})) \xrightarrow{\ez\circ \pi} Z\]
 where \(\pi \colon \ez^*(TZ) \to \G\) is the vector bundle projection, as claimed.

 \end{proof}

The case where \(\G\) is torsion free and hence acts freely on \(Z\) can be 
expressed and proved more simply. 

Since \(Z\times X\) is a \(\Grd\)-equivariantly \(\K\)-oriented bundle of smooth 
manifolds, by hypothesis,
 \(\Grd_\G\)-equivariant Poincar\'e duality holds and gives an 
isomorphism 
\begin{equation}
 \RKK^\G_* (Z;  C_0(X), \C) \cong \RKK^\G_{*+n}\bigl(Z; \C, C_0(X)\bigr)
\end{equation}
and composing this with the generalized Green-Julg isomorphism and a 
standard Morita equivalence identifies 
\(\RKK^\G_{*+n}\bigl(Z; \C, C_0(X)\bigr)\) with \(\K^{-*-n}(Z\times_\G X)\), the 
\(\K\)-theory of the mapping cylinder. 

 The composition
\begin{multline}
\K_{-*}(Z\times_\G X) 
=  \KK^\G_*(C_0(Z\times X), \C) \xrightarrow{\PD^{-1}} \RKK^\G_*(Z; C_0(X), \C) \\
\cong \RKK^\G_{*+n} (Z; \C, C_0(X)) \cong \K^{-*-n}(Z\times_\G X)
\end{multline}
is ordinary Poincar\'e duality for \(Z\times_\G X\). By 
Proposition \ref{proposition:description_of_inflated_class} we obtain the 
following.

\begin{corollary}
If \(\G\) is torsion-free, then \(\G\ltimes Z\)-equivariant Poincar\'e duality for 
\(X\)  
\[\RKK^\G_*(Z; C_0(X), \C) \to  \K^{-*-n}(Z\times_\G X)\] 
maps \(\inflate  ([\widehat{\G \ltimes X}] \) to the class in \(\K^{-d} (Z\times_\G X)\) of the  
correspondence 
\[ \pnt \leftarrow X  \xrightarrow{ \ezx}  Z\times_\G X,\]
where \(\ezx \colon X \to Z\times_\G X\) is the inclusion of the fibre \(X\) at 
\(\G z_0 \in \G \backslash Z\). 

\end{corollary}

\subsection{Dirac classes for proper actions}

Suppose now that 
\(X\) is a smooth, \emph{ proper} \(\G\)-manifold, \(\chi \to Z \cong \EG\) a smooth 
classifying 
map for the proper 
action of \(\G\) on \(X\). 
By Sard's theorem, \(\chi\) has a regular value \(z_0\), so 
\(F := \chi^{-1}(z_0)\) is a smooth submanifold of \(X\) of dimension 
\(n - d\), carrying a smooth action of the finite group 
\(\textup{Stab}_\G (z_0)\). If \(\G \backslash X\) is compact, then 
\(F\) is compact.

The fibres \(F_g := \chi^{-1}(gz_0)\) as \(g\) ranges over 
\(\G\) are isomorphic copies, and come with actions of the 
corresponding conjugate isotropy groups. We set 
\begin{equation}
\label{equation:f}
 F_\G := \{ (x, g) \in X\times \G \; | \;\chi (x) = gz_0\}.
 \end{equation}
which is a bundle over \(\G z_0\subset X\) with fibre 
 \( F_g \times 
\textup{Stab}_\G (gz_0)\) over \(gz_0\) .
A \(\G\)-equivariant \(\K\)-orientation on \(X\) induces a canonical 
\(\G\)-equivariant \(\K\)-orientation on \(F_\G\). 

Set \([F_\G] \in \KK^\G_{d-n} (C_0(F_\G), \C)\) 
the transverse Dirac class for \(\G\) acting on \(F_\G\). 
Let \(i \colon F_\G \to X\) be the projection to the first factor.

\begin{theorem}
\label{theorem:ev_periodicity}
Let \( (\G, Z)\) be a smooth oriented group, and \(X\) is a \(\G\)-equivariantly 
\(\K\)-oriented \emph{proper} \(\G\)-manifold, 
\(\chi \colon X \to Z\) be a 
smooth \(\G\)-map,  \(z_0\in Z\) a regular orbit, and \(F_\G\), \(i \colon F_\G \to X\) 
 \emph{etc} as in the 
discussion above. Then 
the Dirac class for \(\G\ltimes X\) is given by the class of the fibre of \(\chi\): 
\[ [\widehat{\G \ltimes X}] = i_*( [F_\G]) \in \KK_{d-n}^\G (C_0(X), \C).\]
In particular, if \(n < d\) then the Dirac class vanishes, and otherwise, 
the Dirac class is represented by a \(n-d\)-dimensional 
spectral triple over \(C_0(X)\rtimes \G\).

\end{theorem}

\begin{proof}
We show that 
\begin{equation}
\label{equation:what_we_want_to_show_in_this_silly_proof}
 \inflate \bigl( i_*([F_\G]) \bigr) = \PD^{-1}([\ev]\otimes_\C [X])
\in \RKK^\G_{d-n} (Z; C_0(Z\times X), \C).\end{equation} The result will follow from 
Proposition \ref{proposition:factoring_Dirac}. 
As in the proof of Proposition \ref{proposition:description_of_inflated_class}, 
\(\PD^{-1}\) involves composition with the class \(\Theta 
\in \RKK^\G_{d} (Z; \C, C_0(Z)) \), where \(\Theta= \delta !\) where 
\(\delta\colon Z \to Z\times Z\) is the diagonal map, canonically 
\(\K\)-oriented, and \(\Grd_\G\)-equivariant. 

Let \(\chi  \colon X \to Z\) be the smooth classifying map with regular value \(z_0\) 
discussed above. Set \[G_\chi \colon X\to Z\times X\] the graph of \(\chi\): 
\(G_\chi (z, x) := \bigl(z, \chi (x)\bigr)\).

\begin{lemma}
The equality 
\[( \delta \otimes_\C \id_X )! = (\id_Z \otimes_\C G_\chi) ! 
\] 
holds in 
\[ \KK^{\Grd_\G}_{d} \bigl(Z; 
C_0(Z\times X), C_0(Z\times Z \times X)\bigr)
 = \RKK^\G_{d} (Z; C_0(X), C_0(Z\times X) \bigr).\]

\end{lemma}

\begin{proof}
By the universal property of \(\EG \cong Z\),
 the coordinate projections \(Z\times Z \to Z\) are \(\G\)-equivariantly 
 homotopic. Fix a \(\G\)-equivariant smooth homotopy 
 \[ F\colon Z \times Z\times [0,1] \to Z\]
 between the two coordinate maps
  and pull it back in one coordinate using the map 
 \(\chi \) to get 
 \[ \tilde{F}\colon Z \times X \times [0,1]\to Z,\;\; \tilde{F} (z, x, t) := 
 F(z,\chi (x), t).\]
Then, as is easily checked, \(\id_Z\times \tilde{F}\) gives a smooth 
\(\Grd_\G\)-equivariant homotopy between 
the smooth \(\K\)-oriented 
 \(\Grd_\G\)-equivariant maps 
 \(\delta \times \id_X\) and \(\id_Z\times G_\chi\), as claimed.

\end{proof}

To complete the proof, we need to evaluate the composition of 
\(\Grd_\G\)-equivariant correspondences 
\begin{equation}
\label{equation:hellsugsousdfsd}
Z\times X \leftarrow Z\times X \xrightarrow{\id_Z\times G_\chi} Z\times Z\times X
\xleftarrow{\id_Z\times e_{z_0} \times \id_X}\ Z\times \G\times X
\rightarrow Z
\end{equation}
The maps \(G_\chi\colon X \to Z\times X\) and \(e_{z_0}\times \id_X \colon \G \times X \to Z\times X\) are transverse since \(z_0\) is a regular value (and hence so is \(gz_0\) for all \(g\in \G\)) 
and the associated coincidence manifold 
\[ \{ ( x, g, y)\in X\times \G\times X \; | \; G_\chi (x) = (gz_0, y) \}\]
is the smooth \(\K\)-oriented manifold \(F_\G\) described above
in \eqref{equation:f}. Taking the product of 
everything with \(\id_Z\) gives the identity 
\eqref{equation:what_we_want_to_show_in_this_silly_proof} as required.

\end{proof}

\begin{corollary}
\label{corollary:dirac_of_z_is_trivial}
\([\widehat{\G\ltimes Z}] = [\ev]\in \KK^\G_0(C_0(Z), \C)\). 

\end{corollary}

This is the case of Theorem \ref{theorem:ev_periodicity} where \(Z = X\), \(\chi\) the 
identity. map.

\begin{remark}
Corollary \ref{corollary:dirac_of_z_is_trivial} admits another proof, much easier, for it 
boils down to the simple statement that 
\[ [\ev] \otimes_\C [Z] = [Z]\otimes [\ev] \in \KK^\G_{-d} (C_0(Z\times Z), \C), \]
and this follows from the fact that the two coordinate 
projections \(Z\times Z\to Z\), are equivariantly proper homotopic, because 
\(Z\) is universal. 
\end{remark}

Finally, we note that when \(\G\) is finite, acting on \(X\), arbitrary, 
our description \(i_*([F_\G])\) of the Dirac class for \(C_0(X)\rtimes \G\) 
just given 
matches that given in Example \ref{example:dirac_of_finite_group_action}, 
since then \(Z\) becomes a point, and \([\ev]\) the class of the regular 
representation of the finite group.

Theorem \ref{theorem:ev_periodicity} is quite satisfying, from a 
certain point of view, as it gives a case where 
the homological subtraction involved in forming a Dirac 
class, which lies in dimension \(d-n\), matches precisely 
the geometric dimension: there is a spectral triple representative 
of summability dimension \(n-d\).

\section{Dirac classes for boundary actions of negative curved groups}

Along with the Dirac class for the irrational rotation algebra, and its spectral triple 
representative, an important example in noncommutative geometry is the 
action of a co-compact discrete subgroup of \(\mathrm{SL}_2(\R)\) acting on the 
circle by M\"obius transformations. This is a special case of a 
Gromov hyperbolic group \(\G\) acting 
on its boundary \(\partial \G\).

If \(\G\) is a Gromov hyperbolic group (see \cite{GdH} for an exposition) 
 hyperbolicity leads to a compact, metrizable, 
\(\G\)-space \(\overline{\G}\) containing \(\G\) as a dense open \(\G\)-invariant 
open subset, and complement \(\partial \G\). This produces an exact 
sequence 
\begin{equation}
\label{equation:boundary_extension}
0  \to C_0(\G)\ltimes \G \to C(\overline{\G})\rtimes \G \to \hi \to 0
\end{equation}
and, using the canonical isomorphism \(C_0(\G)\ltimes \G \cong \Comp(l^2\G)\), and 
amenabity of the action, a result 
due to Adams in \cite{Ada}, we obtain 
a \(\KK\)-class \([\partial_\G] \in\KK^\G_1(C(\partial G), \C) = 
\K^{-1}(\hi)\). Alternatively, in \cite{EN} a \(\G\)-equivariant completely positive 
splitting of \eqref{equation:boundary_extension} is provided, which implies 
the extension determines a \(\KK_1\)-class. 

The main result of this section is that when \(\G\) also fits into the previous 
framework, the boundary extension class is the Dirac class of the action, with 
its boundary \(\K\)-orientation. 

Some fine points about the statement are discussed following the Theorem.

\begin{theorem}
\label{theorem:the_only_non_trivial_theorem_of_this_paper}
Let \( (\G, Z)\) be a smooth \(\K\)-oriented group with \(Z\) negatively curved. 
Then the \(\G\)-action on \(\partial Z\) admits a \(\KK\)-orientation in the 
sense 
of Definition \ref{definition:equivariant_orientations}, and the Dirac class of the 
action is given by 
\[[\widehat{\G\ltimes \partial Z}] = [\partial_\G]  \in \K^1(C(\partial Z)\rtimes \G) \cong \K^1(C(\bd \G \rtimes \G ).\]
where \([\partial_\G]\) is the boundary extension class. 

\end{theorem}

The boundary extension class is represented by a cycle which doesn't involve 
any smooth structures, but only, in a sense, 
 on the asymptotic geometric of the group \(\G\), and its action on it.

\subsection{KK-orientability of boundary actions}

For simplicity, we are going to assume that \(\G\) is torsion-free, or, equivalently, that 
\(\G = \pi_1 (M)\) for a negatively curved, compact, \(d\)-dimensional, 
manifold which we assume, in addition, to be \(\K\)-oriented and 
on which, therefore, we can assemble a 
Dirac operator \(D\) acting on the spinor bundle.
 
The metric, \(\K\)-orientation, \emph{etc} on \(M\) lifts to \(\G\)-equivariant 
data: spinor bundles on \(Z:= \tilde{M}\), connection, and so on, and 
one assembles the 
`lifted,' \(\G\)-equivariant Dirac operator representing \([Z] \in 
\KK_{-d}^\G(C_0(Z), \C) \).

The boundary sphere \(\partial Z\) of the negatively curved space \(Z\) is the 
boundary of the usual geodesic compactification of \(Z\), and it agrees by a 
\(\G\)-equivariant homeomorphism with
 the Gromov 
boundary \(\partial \G\) of the group \(\G\), or of \(Z\) itself. 

The isometric 
group action of \(\G\) on \(Z\) extends to an action of 
\(\G\) on \(\partial Z\)
by homeomorphisms, due to general properties of hyperbolic groups, and 
the action in this case can be shown to be by 
\(C^{1+\epsilon}\)-diffeomorphisms, but is not smooth in general, even 
for surface groups. \footnote{The action is smooth when 
the curvature of \(M\) is constant.} We show below that, however, the 
\(\G\)-action on the boundary sphere is \(\KK\)-orientable in our sense. 

We first describe the localization map in this situation. 
 We use the same notation as above, 
with \(M = \G\backslash Z\), a negatively curved manifold. Let \(SM\) be the 
sphere bundle of its tangent bundle. Then 
\[ Z\times_\G \partial Z \cong SM,\]
in the following canonical way. Given \(z\in Z\), and \((z,\xi)\) a unit tangent 
vector at \(z\), let 
\[\textup{EXP}_z(z, \xi) :=  \lim_{t\to \infty} \exp_z (t\xi),\;\;\ \xi \in S_xZ.\]
It follows a geodesic ray beginning at 
\(z\) and ending at a boundary point in \(\partial Z\). This construction is 
clearly equivariant and determines a homeomorphism between the 
fibre \(S_zZ\) of the sphere bundle \(SZ \to Z\), and the boundary sphere 
\(\partial Z\). Any such homeomorphism of course then gives 
\(\partial Z\) a smooth structure, identifying it with an ordinary sphere.

In any case, we see that 
the localization map can be identified with a map 
\[ \K^*(C(\partial Z)\rtimes \G)  \to \K_{-*+d}( SM), \]
and a Dirac class is one which maps to the class of a single fibre of 
the bundle projection 
\(\pi \colon SM \to M\), \(\K\)-oriented as a sphere, and regarded as a 
Baum-Douglas cycle for \(SM\).

\begin{lemma}
\label{lemma:bundlesmooth}
In the above notation, the \(\G\)-action on \(\partial Z\) is 
\(\KK\)-orientable. 
\end{lemma}

\begin{proof}
 We show that the \emph{bundle} 
 \(\EG\times \partial Z = Z \times \partial Z\) of smooth manifolds (with 
 fibre \(\partial Z\)) admits a canonical fibrewise smooth structure, left 
 invariant by \(\G\), by noting that the exponential map discussed above 
 gives a \(\G\)-equivariant homeomorphism  
 \[ \textup{EXP}\colon SZ \to Z\times \partial Z,\]
 with \(SZ\) the sphere bundle of \(Z\), such that the diagram 
 \[ \xymatrix{ SZ\ar[d]_{\pi} \ar[r]^{\textup{EXP}} &
  Z\times \partial Z \ar[ld]^{\pr_Z} \\ Z & },\]
 commutes. \(SZ\) is a smooth manifold, and \(\G\) acts smoothly on it since it acts smoothly on 
 \(Z\). The 
 bundle projection \(\pi\) is a smooth, \(\G\)-equivariant submersion. Hence 
 it gives \(Z\times \partial Z \cong 
 SZ\) the structure of a \(\G\)-equivariant bundle of smooth manifolds 
 over \(Z\), that is, a \(Z\ltimes \G\)-manifold.  
 
 The above reasoning, replacing the sphere bundle \(SM\) by the 
 (closed) disk bundle \(\overline{D}M\), generates a bundle 
 of manifolds-with-boundary (the manifolds are closed disks), carrying by assumption 
 a bundle of \(\K\)-orientations, equivariant under \(\G\), and this bundle 
 generates a bundle of \(\K\)-orientations on the bundle of boundaries by 
 the two-out-of-three Lemma.

 \end{proof}

 \begin{remark}
From the foliation point of view, \(S M\), in the above notation,
 admits an Anasov foliation into 
asymptotic equivalence classes of geodesic rays; this foliation has generic leaf 
 \(\cong Z\), or more generally \(Z/\G'\) where \(\G'\subset \G\) is the isotropy of a
 boundary point (always a cyclic group, either infinite or trivial). The leaves of this foliation are all smooth, but the foliation is not infinitely differentiable in the transverse direction. 

 The holonomy groupoid accordingly 
 acts by \(C^{1+\epsilon}\)-diffeomorphisms on one of the transversals
 \(S_xM\), and this can be naturally identified with 
 the boundary action of the group \(\G\) on \(\partial Z\), as in the above 
 argument. 
 
 The lack of smoothness of the \(\G\) action on its boundary thus corresponds to 
 failure of transverse smoothness of the Anosov foliation. See \cite{KH}. 
 
\end{remark}

\subsection{Preparatory remarks on extension theory }

Proceeding to the proof of Theorem \ref{theorem:the_only_non_trivial_theorem_of_this_paper}, 
we begin with a discussion of (known) constructions relating to 
extensions and \(\KK\)-theory. We will need at some stage 
equivariant versions of some of 
the constructions below, with respect to a group, or groupoid, so will deal with the 
general theory at that level of generality. However, as the referee has pointed out, 
there are issues with making extension theory equivariant, for example, a continuous 
group action on a C*-algebra may not extend to a (continuous) action on its multiplier 
algebra. We will be working strictly with \'etale groupoids, however, so we will 
fix \(\Grd\) to be an \'etale groupoid in the following discussion -- in fact, in our applications, 
it will be either trivial, proper, and of the form
 \(\Grd_\G = \G\ltimes \EG\), or \(\G\) itself (with \(\G\) a hyperbolic group, as per the 
 discussion above.)

Suppose that

\begin{equation}
\label{equation:exact_sequence}
 0 \rightarrow J \xrightarrow{\alpha} B \xrightarrow{\beta} A \rightarrow 0
 \end{equation}
is a \(\Grd\)-equivariant exact sequence 
of C*-algebras. 
The \emph{Busby} invariant of the equivariant extension \eqref{equation:exact_sequence}
is the  \(\Grd\)-equivariant *-homomorphism 
\(\tau (a) := \pi (\tilde{a}) \in \Calkin(J) := \Mult (J)/J\), where \(\pi \colon \Mult(B) \to \Mult(B)/B\) is the 
quotient map, and where \(\tilde{a}\) denotes a lift of \(a\) under \(\beta\)., regarded as a 
multiplier of \(B\). 

The Busby invariant is uniquely associated to the strong isomorphism 
class of the extension: if \(\tau \colon A\to \Calkin (J)\) is any 
( \(\Grd\)-equivariant) 
*-homomorphism, then \(B:= \{ (a, m) \in A \oplus \Mult(J) \; | \; \pi(m) = \tau (a)\}\) 
determines an ( \(\Grd\)-equivariant) extension with of \(A\) by \(J\) with the given 
Busby invariant, and if \(\tau\) comes from \eqref{equation:exact_sequence} then 
this procedure determines a strongly isomorphic extension. This procedure puts 
strong isomorphism classes of  \(\Grd\)-equivariant extensions in
 in 1-1 correspondence 
with  \(\Grd\)-equivariant *-homomorphisms \(A \to \Calkin (J)\). 

Similarly, if 
\(\Mult^s(B)\) denotes multipliers of \(B\otimes \Comp\) and \(\Calkin^s(J):= 
\Mult^s(J)/J\otimes \Comp\) then we obtain a bijective 
correspondence between strong isomorphism classes of  \(\Grd\)-equivariant 
extensions of 
\(A\) by \(J\otimes \Comp\) and  \(\Grd\)-equivariant *-homomorphisms 
\(A\to \Calkin^s(J)\). 

We generally work with Busby invariants rather than extensions themselves.

\begin{example}
\label{example:ext_from_kk}
Suppose that \(\rho\colon A \to \Bound(J\otimes l^2)\) is  
a  \(\Grd\)-equivariant representation, and that \(P\in \Bound(J\otimes l^2)\) such that 
\[ [P,\rho (a)], \;\; 
\rho (a) \cdot (P^2-P)\;\; g(P)-P \; \in \Comp(J\otimes l^2)\]
for all \(a\in A\), \(g\in \Grd\), then In the usual way this defines a \(\Grd\)-equivariant 
Kasparov triple 
\( (J\otimes l^2, \rho,  F:= 2P-1)\), that is, cycle for \(\KK_1^\Grd(A,  J)\). 
The map  
\[ \tau \colon A \to \Calkin^s(J), \; \tau (a) := P\rho (a)P\; \textup{mod} \, J\otimes \Comp\]
is a  \(\Grd\)-equivariant *-homomorphism, and hence is the
 Busby invariant of some strong 
isomorphism class of \(\Grd\)-equivariant 
extension of \(A\) by \(J\otimes \Comp\). 
It is the  \(\Grd\)-equivariant extension associated to the Kasparov cycle 
\( ( J\otimes l^2, \rho , P)\) for \(\KK^\Grd_1(A,J)\). 

In typical examples (\emph{e.g.} where \(A\) is unital), the operator 
\(P\) is an \emph{essential projection}: that is, \(P^2-P\) is compact. 

\end{example}

Bott Periodicity can be described conveniently by this set-up.

\begin{example}
\label{example:bottclass}
The Kasparov triple \( ( C_0(\R), 1, \chi)\) represents the 
\emph{Bott class} in \(\KK_1(\C, C_0(\R))\), if \(\chi\colon \R \to [-1,1]\) is a 
normalizing function (odd, and having limits \(\pm 1\) at the endpoints.) Equivalently, 
it represents the \emph{Thom class} of the trivial \(1\)-dimensional vector bundle \(\R\) over 
a point. 

Of course \(P:= \frac{\chi+1}{2}\) is then an essential projection. 

\end{example}

Say that a  \(\Grd\)-equivariant Busby invariant \(\tau \colon A \to \Calkin^s(J)\) is \emph{equivariantly 
dilatable} if there is a completely positive, contractive and \(\Grd\)-equivariant
 map 
\(s\colon A \to \Mult^s(J)\) such that \(\pi \circ s = \tau.\) We say that \(s\) is a splitting 
of \(\tau \). It determines an equivariant completely positive splitting of the corresponding extension. 
The \emph{Stinespring construction} 
realizes any dilatable Busby invariant as one of the form 
\(\tau (a) = P\rho (a)P \; \textup{mod} \, J\otimes \Comp\), \emph{i.e.} 
as the Kasparov cycle as in 
 Example \ref{example:ext_from_kk}, where \(P\) is a projection. Note that 
 this procedure produces not just an essential projection (\(P^2-P\) is compact),
  but a projection (\(P^2-P=0\).) 
\begin{remark}
 In Kasparov theory, essential projections can be 
replaced by actual projections at the expense of adding a degenerate to the 
cycle: suppose that \(\Hilm = J\otimes l^2\) is the standard Hilbert \(J\)-module and 
  \( ( \Hilm, \pi, P)\) is an odd cycle for \(\KK^\Grd_1(A,J)\), 
   where \(P\) is a self-adjoint essential-projection in the sense of Definition 
   \ref{example:ext_from_kk} such that \(||P||\le 1\), then the triple \( ( \Hilm, 0, 1-P)\) is 
   evidently a triple as well, where \(0\) denotes the zero representation of 
   the algebra on \(\Hilm\), but it is degenerate. Hence the cycles 
   \( ( \Hilm, \pi, P)\) and \( ( \Hilm \oplus \Hilm ,  \pi \oplus 0, P\oplus 1-P)\) are 
   equivalent in \(\KK\). Let \(\hat{P}:= \left[ \begin{matrix} P & (P-P^2)^{\frac{1}{2}} \\ (P-P^2)^{\frac{1}{2}} & 1-P\end{matrix}\right]\), then \(\hat{P}^2=\hat{P}\), that is, 
   \(\hat{P}\) is an actual projection. The operator homotopy 
   with \[ \hat{P}_t:= \left[ \begin{matrix} P & (tP-t^2P^2)^{\frac{1}{2}} \\ (tP-t^2P^2)^{\frac{1}{2}} & 1-P\end{matrix}\right]\] 
   gives a homotopy between the Kasparov cycles \( ( \Hilm \oplus \Hilm, \pi \oplus 0, P\oplus 1-P)\) 
   and \(( \Hilm \oplus \Hilm, \rho \oplus 0 , \hat{P})\), so \((\Hilm, \rho, P)\) and 
   \( (\Hilm \oplus \Hilm, \rho \oplus 0 , \hat{P})\) determine the same
   class: \(P\) has been replaced by an actual projection. This leads to a completely 
   positive map 
   \[ \hat{s}(a) := \hat{P}\hat{\rho} (a)\hat{P},\]
   for the Busby invariant 
   \[ \hat{\tau}(a) := \hat{P}\hat{\rho} (a)\hat{P} \; \textup{mod} \;  J\otimes \Comp.\]
  Computing, this equals 
  \[\left[ \begin{matrix} P\rho (a)P & P\rho (a) \sqrt{P-P^2} \\ 
   P\rho (a) \sqrt{P-P^2}& \sqrt{P-P^2}\rho (a) \sqrt{P-P^2}\end{matrix}\right] \; \textup{mod}\, M_2(J)
   \]
   and as all entries except for the top left corner are zero mod \(M_2(J)\), this equals 
   \(P\rho (a)P = \tau (a)\), whence \(\hat{\tau} = \tau\). 
   \end{remark}


We are primarily interested in the following example of a geometric source. 

Let \(\overline{M}\)
 be a compact manifold-with-boundary 
\(\bd M\) and interior \(M\). Then 
we have an extension 
\begin{equation}
\label{equation:manifold_with_boundary_extension}
 0 \rightarrow C_0(M) \to C(\overline{M}) \to C(\bd M) \to 0.
 \end{equation}
With the identification \(\Mult\bigl(C_0(M)\bigr) \cong C_b(M)\) we can describe the 
Busby invariant as the *-homomorphism 
\[ \tau\colon C(\bd M) \to C_b(M) /C_0(M), \;\;\tau (f) = \tilde{f} \; \textup{mod} \; C_0(M)\]
where \(\tilde{f}\) is any extension of \(f\) to a continuous 
function on \(\overline{M}\).

It will be convenient to work with another extension rather than \eqref{equation:manifold_with_boundary_extension}.

Let  \(\overline{\nu} \subset \overline{M}\) be a collar of the boundary (collars, even 
\(\Grd\)-equivariant collarings, with \(\Grd\) a groupoid, are discussed extensively 
in \cite{EM:Dualities}.) Thus 
\(\overline{\nu} \cong \bd M \times [0,1)\) with the boundary identifying as 
\(M\times \{0\}\). 
Let \(\chi \in C_b(\overline{\nu})\) with \(0 \le \chi \le 1\), \(\chi = 1\) on 
\(\bd M\) and \(\chi\) has compact support in \(\overline{\nu}\). We may take 
\(\chi (x,t) = 1-t\) for example, defined on \(U\cong \bd M \times [0,1)\) and 
extended to zero outside \(U\), where we have used the identification 
\(\overline{\nu} \cong \bd M \times [0,1)\) Let 
\[ \nu := \overline{\nu}\cap M\,\]
an open subset of \(M\) homemorphic (via the collaring) to \(\bd M \times (0,1)\). 
Let \(r \colon \overline{\nu} \to \bd M\) be the 
projection, then dualizing \(r\) gives a *-homomorphism 
\[ \hat{r}\colon C(\bd M) \to C_b(\overline{\nu}) \subset C_b(\nu),\]
and since \(\chi^2-1 \in C_0(\nu)\),
 we get a Kasparov cycle 
\( ( C_0(\nu), \hat{r}, \chi)\) for \(\KK_1(C(\bd M), C_0(\nu))\). 
On the other hand, the extension
\begin{equation}
\label{equation:germ_ext}
  0 \to C_0(\nu) \to C_0(\overline{\nu}) \to C(\bd M) \to 0
  \end{equation}
  admits the following completely positive splitting: 
  \[ s\colon C(\bd M) \to C_0(\overline{\nu})\; \; s (f) := \chi  \cdot \hat{r}(f) \cdot \chi,\]
  since \(\chi \cdot\hat{r}(f) \cdot \chi = f\) on \(\bd M.\) 
  
  Now, carrying out the general procedure discussed above,
   with \(A:= C(\bd M), \; B:= C(\overline{\nu}), \; 
  J:= C_0(\nu)\) and \(P:= \chi\), and we 
 realize the Busby invariant of the 
  extension \eqref{equation:germ_ext} as the upper left corner of the 
  representation 
  \[\hat{\rho} (f) := \left[ \begin{matrix} 
  \chi \cdot \hat{r} (f)\cdot \chi &  \chi \cdot \hat{r} (f) \cdot \sqrt{\chi -\chi^2} \\ 
   \chi \cdot  \hat{r} (f)\cdot  \sqrt{\chi-\chi^2}& \sqrt{\chi-\chi^2}\cdot \chi  \cdot \hat{r} (f) \cdot \sqrt{\chi-\chi^2}\end{matrix}\right] \; \textup{mod}\, M_2\bigl( C_0(\nu)\bigr),
   \]
  This proves that the Kasparov cycle associated to the extension 
  \eqref{equation:germ_ext} is the triple 
  \( ( C_0(\nu), \hat{r}, \chi)\) for \(\KK_1(C(\bd M), C_0(\nu))\). (Compare to 
  Example \ref{example:bottclass}). 
  
  Let \(\iota\colon \bd M \to M\) the smooth embedding \(\iota (x) = c(x,\frac{1}{2})\), where 
\(c\colon \bd M \times [0,1)\to M\) is the collaring (\(c\) is a diffeomorphism onto 
the open collar \(\overline{\nu}\).) Then 
 \(\iota\) has an evidently trivial normal bundle with total space \(\nu\) the same as 
 above, which therefore 
 carries a canonical \(\K\)-orientation and Thom class \(t_\nu \in 
 \KK_1(C(\bd M) , C_0(\nu))\). 
  
  \begin{lemma}
  1) The class \([\partial_{\overline{\nu}}]\in \KK_1(C(\bd M), C_0(\nu))\) of the extension 
  \eqref{equation:germ_ext} is equal to the Thom class 
  \(t_\nu\) of the normal bundle to the embedding \(\iota\). 
2) Let \([\bd_{\cM}] \in \KK_1(C(\bd M), C_0(M))\) be the class of the extension  
\eqref{equation:manifold_with_boundary_extension}. Then
\[(\varphi^!)_* ([\partial_{\overline{\nu}}]) = [\partial_{\cM}]\in \KK_1(C(\bd M), C_0(M))\]
holds, where \(\varphi^!\colon C_0(\nu) \to C_0(M)\) is the *-homomorphism induced by 
the open embedding \(\varphi\colon \nu\to M\).

\end{lemma}

\begin{remark}  
In the notation of correspondences, this says that 
\[ \iota! = [\partial_{\cM}]\in \KK_1(C(\bd M), C_0(M).\]
\end{remark}

The proofs of the Lemma is straightforward
(see Example \ref{example:bottclass}). See 
 \cite{Connes-Skandalis} for some rather similar argumentation (which is 
 where the author learned it.)

 There is no difficulty 
 whatsoever in making the construction \(\Grd\)-equivariant, if 
 \(\Grd\) is a proper groupoid, acting on a bundle of smooth manifolds.  
 We state the Lemma, and leave details of the proof to the reader; 
 very similar such Lemmas using similar techniques appear in 
\cite{EM:Dualities}. Furthermore, we construct equivariant 
collars in the main example of interest, below.

\begin{lemma}
\label{lemma:main_inflation_lemma}
Let \(\Grd\) be a proper groupoid acting smoothly on a bundle 
of smooth manifolds-with-boundary \(\overline{M}\) with boundary 
\(\bd M\) and interior \(M\). Assume there is a \(\Grd\)-equivariant 
collaring 
\[ c\colon \bd M \times [0,1) \to \overline{M},\]
 \emph{i.e.} \(c\) is a \(\Grd\)-equivariant 
 fibrewise diffeomorphism onto an open neighbourhood 
 of \(\bd M \) in \(\overline{M}\)  whose restriction to \(\bd M \times \{0\}\) is the 
inclusion \(\bd M \to \overline{M}\). Then the \(\Grd\)-equivariant extension 
\[ 0 \to C_0(M) \to C(\overline{M}) \to C(\bd M) \to 0\]
admits a \(\Grd\)-equivariant contractive and completely positive splitting, the 
smooth embedding \(\iota \colon \bd M \to M, \; \iota (x) = c(x,\frac{1}{2})\) admits a 
canonical \(\Grd\)-equivariant orientation, and the equation 
\[ \iota! = [\partial_M]\in \KK_1^\Grd(C(\bd M), C_0(M))\]
holds.

\end{lemma}

  \subsection{Conclusion of the proof}
  
  We now return to hyperbolic groups. 
  
  Let \( (Z, \G)\) be a smooth \(\K\)-oriented 
  group with \(Z\) negatively curved, so that it is Gromov hyperbolic with 
  Gromov compactification \(\cZ\) and boundary
  \(\partial Z \cong \partial \G\) (noting that 
  \(\G\) acts co-compactly and isometrically 
  on \(Z\).)
  
   The boundary \(\bZ\) can be identified with \(S^{n-1}\), if \(\dim (Z) = n\), and 
  the group acts by \(C^{1+\epsilon}\)-diffeomorphisms of this sphere. 
  We have a \(\G\)-equivariant 
  exact sequence 
  \begin{equation}
  \label{equation:boundary_extension_filled_in}
  0 \to C_0(Z ) \to C(\cZ) \to C(\bZ) \to 0.
  \end{equation}
Clearly the boundary does not admit a \(\G\)-equivariant collaring. For if it did, 
  the image \(\iota (\bZ) \subset Z\) would be \(\G\)-invariant, which is impossible, 
  since \(\bZ\) is compact and \(\G\) acts properly. 
  
  However, after inflating this whole situation over \(Z \cong \EG\) we obtain a 
  \(\Grd\)-equivariant bundle of manifolds-with-boundary and corresponding 
  extensions, and now it is possible to find an 
  equivariant collaring 
  \[ c\colon Z\times \bZ\times [0,1)\to Z\times \cZ\]
 quite explicitly, for example as follows. 
For a point \( (x,a) \in Z\times \bZ\), the geodesic ray emanating from 
\(x\) and pointing towards the boundary point \(a\) determines a \emph{canonical}
(geometrically defined) map \(r_{x,a}\colon [0,\infty) \to Z\) (and endpoint
\(a\) in \(\cZ\)), so that 
\[ r_{gx,ga}(t) = gr_{x,a}(t)\]
for any isometry of \(Z\).  Fix a re-scaling 
\([0,1)\stackrel{\cong}{\longrightarrow} [0,\infty)\) mapping \(1\) to \(\infty\). 
 Then combining the re-scaling 
with the map \(r\) just defined gives the required collaring 
\[ c\colon Z \times \bZ \times [0,1) \cong Z\times \bZ\times [0,\infty)\stackrel{\bar{r}}\to Z\times \cZ, \]
the explicit formula may be taken to be
 \[c(x,a,t) = \bigl( x, r_{x,a}( \frac{\sqrt{1-t}}{t})\bigr),\; \textup{if} \; t\not= 0,
  \; \textup{and} \; c(x,a,0)  := a.\]

Lemma \ref{lemma:main_inflation_lemma} now applies. Let 
\[ \iota \colon Z\times \bZ \to Z\times Z,\; \iota (x,a) := c(x,a, \frac{1}{2}).\]
This is a \(\Grd\)-equivariant oriented fibre wise embedding, \emph{i.e.} a 
bundle of smooth embeddings \(\bd Z \to Z\) parameterized by the 
points of \(Z\), and equivariant, as a bundle of maps, under the \(\G\)-action. 
It yields a class 
\begin{equation}
\label{equation:owieurowieurwe}
 \iota! \in \KK^\Grd_1(C_0(Z\times \bZ) , C_0(Z\times Z)) \cong 
\RKK^\G_1(Z; C(\bZ), C_0(Z))
\end{equation}
which, by Lemma \ref{lemma:main_inflation_lemma} 
is the same as the corresponding extension class, that is, the 
class obtained via the \(\Grd_\G\)-equivariant Stinespring construction from the 
\(\Grd_\G\)-equivariant bundle of extensions
\begin{equation}
\label{equation:bundle_of_sdfsdfsh}
0 \to C_0(Z\times Z)\to C_0(Z\times \cZ) \to C_0(Z\times \bZ)\to 0,
\end{equation} 
over \(Z\). 

For convenience in the argument we are using, we point out that 
the \(\G\)-equivariant extension \eqref{equation:boundary_extension_filled_in}, 
though it doesn't admit a \(\G\)-equivariant collaring, it 
\emph{does} admit a \(\G\)-equivariant completely positive splitting of 
another kind. If \(x\in Z\), the exponential map determines a map 
\( \exp_z\colon S_x(Z) \to \partial Z, \) and pushing forward the volume 
element on \(S_xZ\) determined by the metric, we obtain a measure 
\(\mu_x\in \mathrm{Prob}(\partial Z)\). The formula 
\[ (Pf)(z) := \int_{\partial Z} f(\xi)d\mu_x(\xi)\]
provides a Poisson-transform in this situation: \(Pf\) extends continuously 
to \(\overline{Z}\) and restricts to \(f\) on the boundary. The construction 
is clearly equivariant. 

Therefore the extension 
\[ 0 \to C_0(Z) \to C(\overline{Z})\to C(\partial Z) \to 0\]
determines a class \([\partial_{\cZ}]\) in \(\KK^\G_1\bigl(C(\partial Z), C_0(Z)\bigr)\). 

See \cite{Coo} for a treatment of measures on the boundaries of 
hyperbolic spaces.

It is clear from computing with the cycles that 
\[ [\partial_{\cZ}]\otimes_{C_0(Z)} [\ev] = 
[\partial_\G]\in \KK^\G_1(C(\partial Z), \C),\]
where \([\partial_\G]\) is the boundary extension class (the Dirac class, as 
we aim to prove).

\begin{lemma}
The inflation of the boundary extension class
 \([\partial_\G]\in \K^1(\hi)\) factors as 
\begin{equation}
\inflate ([\partial_\G]) = \iota ! \otimes_{C_0(Z\times Z)} \inflate([\ev]) 
\in \RKK^\G_1(Z; C(\bZ), \C).
\end{equation}
with \(\iota!\) as in \eqref{equation:owieurowieurwe}. 

\end{lemma}

\begin{proof}
The 
class \(\inflate([\partial_{\cZ}])\) is that of an extension 
\eqref{equation:bundle_of_sdfsdfsh} with 
two natural \(\Grd_\G\)-equivariant completely positive 
splittings: one using the product of the identity map on \(Z\) and the 
Poisson splitting of \eqref{equation:boundary_extension_filled_in} explained above, 
and the other using the \(\Grd_\G\)-equivariant collaring. The space of 
\(\Grd_\G\)-equivariant completely positive maps is convex, and so the two 
corresponding cycles are homotopic. We obtain 
\[ \iota! =\inflate([\partial_{\cZ}]) \in \RKK^\G_1\bigl(Z; C(\partial Z), C_0(Z)\bigr),\]
by 
Lemma \ref{lemma:main_inflation_lemma}. 
We get 
\begin{multline}
 \iota ! \otimes_{C_0(Z\times Z)} \inflate ([\ev])
\\= \inflate ([\partial_{\cZ}] ) \otimes_{C_0(Z\times Z)} \inflate([\ev])
= \inflate ( [\partial_{\cZ}]\otimes_{C_0(Z)} [\ev]) = \inflate ([\partial_\G])
\end{multline}

\end{proof}

We now complete the proof of Theorem \ref{theorem:the_only_non_trivial_theorem_of_this_paper}. 

We have shown that 
\begin{equation}
\label{equation:proof_of_main_1}
\inflate([\partial_\G]) = \iota! \otimes_{C_0(Z\times Z)} \otimes_{C_0(Z\times Z)} (1_{C_0(Z)}\otimes_\C [\ev])
\end{equation}
where 
\( \iota\colon Z\times \bZ\to Z\times Z\) is 
\[\iota (x, a) = \bigl(x, c(x,a, \frac{1}{2})\bigr).\]
This equality holds in 
\(\RKK^\G_*(Z; C(\bZ), \C).\)

To complete the proof, we show that 
\begin{equation}
\label{32493849238492834}
\PD \bigl( \iota! \otimes_{C_0(Z\times Z)} \otimes_{C_0(Z\times Z)} (1_{C_0(Z)}\otimes_\C [\ev])\bigr) = [\ev]\otimes_\C [\bZ].
\end{equation}

Using the definition of \(\PD\), the left hand side of this equation is 
\begin{equation}
\label{equation:proof_of_main_2}
 \overline{\iota!} \otimes_{C_0(Z\times Z)} (1_{C_0(Z)}\otimes [\ev]) 
 \otimes_{C_0(Z)} [Z] = \overline{\iota !} \otimes_{C_0(Z\times Z)} ( [\ev]\otimes_\C [Z])
\end{equation}
By commutativity of the external product \(\otimes_\C\) and the fact that the two 
coordinate projections \(Z\times Z \to Z\) are \(\G\)-equivariantly homotopic, 
this is the same as 
\begin{equation}
\label{equation:proof_of_main_3}
  \overline{\iota !} \otimes_{C_0(Z\times Z)} ( [Z]\otimes_\C [\ev])
  =\overline{\iota!} \otimes_{C_0(Z\times Z)}([Z]  \otimes_\C 1_{C_0(Z)})\otimes_{C_0(Z)} [\ev]
\end{equation}

Now we claim that 
\begin{equation}
\label{equation:proof_of_main_4}
\overline{\iota!} \otimes_{C_0(Z\times Z)}([Z]  \otimes_\C 1_{C_0(Z)})
= 1_{C_0(Z)}\otimes_\C [\bZ]
\end{equation} 
with \([\partial Z]\) the transverse Dirac class on \(\bZ\). 

The class 
\(\overline{\iota !} \in \KK^\G_*(C_0(Z\times \bZ), C_0(Z\times Z))\) 
is represented by the wrong-wap map 
\[ r\colon Z\times \bZ \to Z, r(x,a):= \bigl( x, c(x,a, \frac{1}{2})\bigr),\]
or, more precisely, by the smooth correspondence 
\[ Z\times \bZ \xleftarrow{\id} Z\times \bZ \xrightarrow{r}Z\times Z.\]
And 
\([Z]\otimes_\C 1_{C_0(Z)} \) is of course represented by the correspondence 
\[Z\times Z\xleftarrow{\id} Z\times Z\xrightarrow{\pr_2}  Z.\]
These correspondences are transverse (the left map of the second is a submersion) so 
can be composed using transversality and coincidence spaces and the outcome is 
easily computed  to be the correspondence 
\[ Z\times \bZ \xleftarrow{\id} Z\times \bZ \xrightarrow{r'} Z,\]
with \(r'(z, a) := c(z, a, \frac{1}{2})\). 

But moving \(c(z,a, \frac{1}{2})\) along the ray \([z,a)\) from \(c(z,a, \frac{1}{2}) \) to 
\( z = \lim_{t\to 1}c(z,a,t) \)  gives a homotopy between \(r'\) and 
the projection \(\pr_{Z}\colon Z\times \bZ \to Z\).

Now given the claim \eqref{equation:proof_of_main_4}, 
plug it into \eqref{equation:proof_of_main_3}, to give 
that the left hand side of \eqref{32493849238492834} equals 
\[ \bigl( 1_{C_0(Z)}\otimes_\C [\bZ]\bigr)  \otimes_{C_0(Z)} [\ev] = [\bZ]\otimes_\C [\ev].\]
as required.

\section{The intersection index formula }

The procedure followed by Connes, Gromov 
and Moscovici say in their paper \cite{CGM}, amounts to,  
as they put it, a sort of 
reverse index theorem. It builds an analytic 
object (a \(\KK\)-class) 
from to a given topological one (group cohomomology class). In their 
case, the purpose was to prove homotopy-invariance of 
the topological object (the higher signature determined by the 
cohomology class).  

We are doing something similar. The class of a fibre in 
\(Z\times_\G X\to \G\backslash X\) may or may not admit an 
\emph{analytic} representative (as a cycle representing the
 Dirac class). But any time this is done, an automatic topological 
 formula for its induced \(\K\)-theory pairing is determined.

Let \( (Z, \G)\) be a smooth, \(d\)-dimensional, \(\K\)-oriented group, and let 
\(X\) be a smooth equivariantly \(\K\)-oriented \(\G\)-manifold of dimension \(n\). 
Assume that a Dirac class \([\widehat{\G \ltimes X}]\in \KK_{d-n} (C_0(X)\rtimes \G, \C)\) 
exists, then it determines a pairing and corresponding map 
\[ \K_{n-d } (C_0(X)\rtimes \G) \to \Z.\]
We are going to compute this map geometrically on the range of the Baum-Connes 
assembly map 
\[\mu\colon  \KK^\G_{n-d }  \bigl(C_0(Z), C_0(X)\bigr) \to \K_{n-d} (C_0(X)\rtimes \G).\]
This is possible due to the topological definition of Dirac class.

To simplify matters, we will 
assume that \(\G\) is torsion-free in the rest of this section. 
As before we let 
\(\ezx \colon X \to Z\times_\G X\) the inclusion of \(X\) as
 a fibre in \(p\colon Z\times_\G X \to \G\backslash Z\). 

Using inverse of the Poincar\'e duality from 
Proposition \ref{proposition:factoring_Dirac}
\[\PD^{-1}\colon  \KK^\G_*\bigl(C_0(Z), C_0(X)\bigr) \cong \RKK^\G_{*+d}\bigl(Z; \C, C_0(X)\bigr)\]
and the generalized Green-Julg Theorem, 
\[ \RKK^\G_*\bigl(Z; \C, C_0(X)\bigr) \cong \KK_{*} 
(\C, C_0(Z\times X) \rtimes \G) \cong 
\K^{-*}(Z\times_\G X),\]
we may re-cast the Baum-Connes assembly map as a map 
\begin{equation}
\label{equation:dualized_bc}
 \hat{\mu} \colon \K^{-*} (Z\times_\G X) \to \K_{*-d}(C_0(X)\rtimes \G)
 \end{equation}
shifting degrees by \(-d\). The relevant dimension for purposes of pairing 
with the Dirac class, is then \(* = n\). 
The domain of \(\hat{\mu}\) may be described in terms of 
geometric equivalence classes of smooth correspondences 
\begin{equation}
\label{equation:correspondence_generators_for_ktop}
 \pnt \leftarrow (M, \xi) \xrightarrow{f}  Z\times_\G X.
 \end{equation}
with \(\xi\) a compactly supported \(\K\)-theory class on \(M\), \(f\) \(\K\)-oriented. Since 
\(Z\times_\G X\) is already endowed with a fixed \(\K\)-orientation, by the \(2\)-out-of-\(3\) 
result for \(\K\)-oriented vector bundles, \(\K\)-orientations on \(f\), that is, on the 
vector bundle \(f^*\bigl( T(Z\times_\G X)\bigr) \oplus TM\), are in \(1\)-\(1\) 
correspondence with \(\K\)-orientations on \(M\), so that we may regard the data of a 
geometric cycle for the domain of \(\hat{\mu}\) as being a smooth, \(\K\)-oriented 
manifold \(M\), a smooth map \(f\colon M \to Z\times_\G X\), and a compactly supported 
\(\K\)-theory class \(\xi \in \K^{-i}(M)\) on \(M\). 

The dimension of the corresponding 
class in \(\K^*(Z\times_\G X)\) is \( -i-n-d + \dim M\), and after Poincar\'e dualizing it 
we obtain thus a class in 
\(\KK^\G_{i+n-\dim M} (C_0(Z), C_0(X))\), so that in order to get
 a \(n-d\)-dimensional \(\K\)-theory class, 
  the right dimension to pair 
with the Dirac class.  
for \(C_0(X)\rtimes \G\), we need \( i = \dim M -d\) (mod \(2\), of course). 

In particular, if \(M\) is compact, 
and \(\xi = \mathbf{1}\) is the class of the trivial line bundle, the case of most immediate 
geometric 
interest, and the one we will focus on, then this shows that we should be interested 
in examples when \(\dim M = d\), the dimension of \(Z\). 

\begin{definition}
A \emph{\(d\)-dimensional 
geometric cocycle} for the \(\G\) action on \(X\) is a pair consisting of a compact, 
\(d\)-dimensional 
\(\K\)-oriented manifold \(M\) and a smooth map \(f\colon M \to Z\times_\G X\). 

Its class in \(\K^{-d}(Z\times_\G X)\) is denoted \(\mathrm{Index} (f!)\). 

\end{definition}

\begin{example}
\label{example:integer_actions_dirac_pairing_1}
If the integers \(\Z\) acts on a compact manifold \(X\) of dimension \(n\), then 
\(d = 1\) and geometric \(1\)-cocycles correspond roughly 
to (homotopy classes of) 
loops in the mapping cylinder \(\R\times_\Z X\). 

\end{example}

\begin{example}
\label{example:z2_action_and_some_geometric_cocycles}
Let \(\Z^2\) act on a torus \(\T^n\) by a pair of group translations, so that 
\(\R^2\times_{\Z^2} \T^n\cong \T^{2+n}\). 
 Let \( \tilde{L}\) be a plane in \(\R^{2+n}\) specified by a
set of \(n\) equations
\[ a_i x +b_i y + u_{i1} t_1 + \cdots + u_{in} t_n = 0, \;\; i = 1, \ldots , n\]
with integer coefficients with the \(n\)-by-\(n\) matrix \( U:= (u_{ij})\) invertible 
over \(\Q\). 

Then \(\tilde{L}/\Z^{n+2}\) is a \(2\)-dimensional torus which maps canonically to  
\(\T^{2+n}\).  We get a \(2\)-dimensional geometric cocycle 
\[ p\colon \T^2 \to \R^2\times_{\Z^2} \T^n.\]

\end{example}

\begin{example}
\label{example:geometric_intersection_index_vector_field}
Let \(M\) be a negatively curved compact \(d\)-dimensional
 \(\K\)-oriented manifold and 
\(\G = \pi_1 (M)\) acting on the universal cover \(Z := \tilde{M}\). Then 
\[ Z\times_\G \partial Z \cong SM,\]
wher e\(SM\) is the sphere bundle of the tangent bundle of \(M\). If the Euler characteristic 
\(\chi (M) \) is zero, then there is a non-vanishing vector field on \(M\) and hence a 
smooth map 
\[ \xi \colon M \to SM,\]
and, \(\K\)-orienting \(\xi\) by the \(\K\)-orientations on its domain and range, 
 we get a \(d\)-dimensional geometric cocycle \(\mathrm{Index}(\xi!)\) 
 for \(\G\) acting on \(\partial Z\). 

\end{example}

Let \( (M, f)\) be a (slightly inapted named) \(d\)-dimensional 
geometric cycle, as it determines a smooth \(n\)-dimensional 
correspondence 
as in \eqref{equation:correspondence_generators_for_ktop}
with class \([M,f]\in \KK_{n}\bigl( \C, C(Z\times_\G X)\bigr)\). The dualized version of the 
assembly map of \eqref{equation:dualized_bc} shifts degrees by \(-d\) and 
so \[\hat{\mu} ([M,f])\in \KK_{n-d} (\C , C_0(X)\rtimes \G) = \K_{n-d}(C_0(X)\rtimes \G)\]
can be paired with the Dirac class
\(  [\widehat{\G\ltimes X}]\in \KK_{d-n} (C_0(X)\rtimes \G, \C)\) to give an 
integer. 
We call this integer the \emph{Dirac index} of the cocycle. It is an analytic invariant.

We now define a topological invariant of a cocycle \(f\colon M \to Z\times_\G X\).
Consider the inclusion
 \(\ezx \colon X \to Z\times_\G X\) of the fibre \(X_{z_0}\). By perturbing 
\(f\) through a homotopy if necessary, we may assume that \(f\) and \(\ezx\) are transverse. Therefore 
we can compose the correspondences  
\[ \pnt \leftarrow M \rightarrow Z\times_\G X  \xleftarrow{\ezx} X\rightarrow \pnt\]
by transversality, yielding the \(\K\)-oriented smooth, \(0\)-dimensional manifold 
\( f^{-1}( X_{z_0})\), where \(X_{z_0}\) is the fibre. 
This inverse image is 
is a finite set of points, suitably \(\K\)-oriented.

We call the algebraic sum of these \(\K\)-oriented (\emph{e.g.} signed) points the 
\emph{intersection index} of the cocycle.

\begin{example}
\label{example:integer_actions_dirac_pairing_2}
In the case of integer actions as in Exampe \ref{example:integer_actions_dirac_pairing_1}
the intersection index of a loop in \(\R\times_\Z X\) is the algebraic number of times 
the loop crosses the hypersurface \(X\cong F \subset \R\times_\Z X\).

\end{example} 

\begin{example}
\label{example:z2_actions_on_tori_geometric_cycles_2}
The intersection index of the \(2\)-dimensional geometric cocycles for a \(\Z^2\)-action 
on \(\T^n\) as in Example \ref{example:z2_action_and_some_geometric_cocycles} is 
given by the cardinality of the finite group \(U(\Z^n) /\Z^n\) with \(U\in M_n(\Z)\cap \GL_n(\Q)\) the integer 
matrix used to define the plane.

\end{example}

\begin{remark}
Given that \(\G\) is torsion-free, the
 long exact sequence of the fibration \( p\colon Z\times_\G X \to \G\backslash Z\) 
(with fibre \(X\)) gives that the inclusion of the fibre \(X\) in \(Z\times_\G X\) 
induces an isomorphism \( \pi_{d} (X) \to \pi_d (Z\times_\G X)\) on homotopy groups
as long as \(d \ge 2\), since \(\G \backslash Z\) is aspherical. 

This shows that 
for \(d>1\), one cannot achieve geometric cocycles with nonzero intersection 
indices by mapping spheres \(S^d\to Z\times_\G X\), since they factor through the 
fibre inclusion, and up to obvious homotopy one can alter one of any two fibre 
inclusions to make them have disjoint range. 

\end{remark}

The main point of Dirac classes for actions is the following result, which can be 
considered a kind of `black box' index theorem. It applies automatically every 
time one constructs a representative of the Dirac class. The result below is 
a special case of a more general one, Theorem \ref{theorem:blackbox_index}.

\begin{theorem}
\label{theorem:my_index_theorem}
If \( f\colon M \to Z\times_\G X\) is a \(d\)-dimensional geometric cocycle for 
\(\G\) acting on \(X\), then it's Dirac (analytic) 
index equals its (topological)  intersection index.

\end{theorem}

This result is a formal consequence of functorial properties of the 
Dirac method and is discussed following the examples below.

\subsection{Integer actions}
In the case of integer actions, say of \(\Z\) acting on \(X\) smoothly 
by a diffeomorphism \(\varphi \colon X  \to X\), the Dirac class 
\([\widehat{\Z\ltimes X}] \in \KK_{1-n} (C(X)\rtimes \Z, \C)\) 
always is non-vanishing and non-torsion in \(\K\)-homology. 
This follows from Theorem
 \ref{theorem:my_index_theorem} and the 
following construction of a \(1\)-cocycle. 

Choose a point \(x_0\) of \(X\) and 
let \(\gamma \colon [0,1]\to X\) be a smooth path from 
 \(x_0\) to \(\varphi (x_0)\) with \(\gamma'(0)\not= 0\).
  Let \(f\colon [0,1]\to \R\times_\Z X\) be the 
loop \(f(t) := [ \bigl(t, \gamma (t)\bigr)]\). It's intersection index is clearly \(+1\).

If the action is isometric, then we may represent the Dirac class by the  
spectral triple(s) described in Example \ref{example:zd_actions_on_tori}.

Suppose \(X = \T\) with \(\Z\) acting by irrational rotation. The Dirac class is 
represented by the deformed Dolbeault operator \(D_\theta\) acting on 
\(L^2(A_\theta)\), while \(\R\times_\Z \T \cong \T^2\) can be identified with 
the \(2\)-torus. If \(L_{p,q} \subset \R^2\) is a line through the origin of 
rational slope \(\frac{p}{q}\), then it projects to a loop 
\[ f_{p,q} \colon \T \to \T^2,\]
that is, a geometric \(1\)-cocycle, 
whose intersection index is \(= + q\). 

The analytic counterpart of the index theorem is as follows.

 Let 
\(L_\theta\) be a line in \(\T^2\) of slope \(\theta\). The real line acts by 
Kronecker flow on \(\T^2\) and the loops
\(f_{p,q}\) are transverse to the flow. Restricting the groupoid 
\(\R\times \R^2\) to the transversals \(f_{p,q}\) and identifying the 
transversals with \(\T\) yields, as one computes, the groupoid 
\( \Z\ltimes_{\theta'} \T\) of irrational rotation by \(\theta' := 
\frac{p\theta + q}{r\theta + s}\). Forgetting the left action of 
\(A_{\theta '}\) on these modules gives (since \(A_{\theta'}\) is 
unital and acts by compact operators) a family of 
\(\mathcal{E}_{p,q}\) of finitely generated projective 
modules over \(A_\theta\) (studied by Rieffel.) 

The Dirac index computes the 
Kasparov pairing 
\[ \langle [\mathcal{E}_{p,q}], [D_\theta]\rangle.\]
Intuitively, this is the index of the Dirac operator on \(\T^2_\theta\) `twisted by'
the `bundle' \(\mathcal{E}_{p,q}\), and by choosing a suitable connection one 
can represent it quite concretely as a deformed Doleault operator, acting 
on sections of the relevant bundle. The Dirac index (the Fredholm index of this 
operator) is thus \(+q\). 

\subsection{Vanishing of the Fredholm index}

The intersection index formula allows us to dispose rapidly of 
the problem of computing the ordinary Fredolm index of 
a Dirac class, for an action of \(\G\) on \(X\) \emph{compact}, that is, the 
pairing 
\[ \langle [1], [\widehat{\G\ltimes X}]\rangle \in \KK_{d-n}(\C, \C) = \Z,\]
which, of course, is only potentially nonzero when \(d-n = 0\) mod \(2\).

\begin{theorem}
\label{theorem:vanishing_index}
The Fredholm index of the Dirac class \([\widehat{\G\ltimes X}]\) of any 
action is zero. 

\end{theorem}

For example, the ordinary Fredolm index of the deformed Dolbeault 
operator \(D_\theta\) on \(\T^2_\theta\) is zero.

The proof is based on the following simple Lemma, whose 
proof already follows from the discussion in the proof of 
Proposition \ref{proposition:nonvanishing_of_Dirac_for_groups}.

\begin{lemma}
\label{lemma:theorem_pairing_nonzero_4}
\(\mu ([\ev]) = [1_{C^*\G}]\in \K_0(C^*\G)\) where 
\(1_{C^*\G}\) is the unit in \(C^*\G\). 
\end{lemma}

\begin{proof} (of Theorem \ref{theorem:vanishing_index}). 
We lift the class \([1]\in \K_0\bigl( C_0(X)\rtimes \G\bigr)\) under \(\hat{\mu}\) to 
a geometric coycle with zero intersection index.

Let \(u \colon \C \to C(X)\) be the \(\G\)-equivariant
 inclusion, \(u\rtimes \G \colon C^*(\G) \to C(X)\rtimes \G\) the induced map.
 The diagram 
 \[ \xymatrix{  \KK_0^\G(C_0(Z), \C)\ar[d]^{u_*} \ar[r]^{\mu}  & \K_0(C^*\G\bigr)\ar[d]^{(u\rtimes \G)_* }\\ 
 \KK^\G_0\bigl( C_0(Z), C(X)\bigr) \ar[r]^{\mu_X} &\K_0( C(X)\rtimes \G) }\]
commutes and thus \(\mu_X \bigl( u_*([\ev]) \in \KK^\G_0\bigl( C_0(Z), C(X)\bigr) = [1_{C(X)\rtimes \G}]\) by 
the Lemma. 
To show this has zero pairing with the Dirac class, it suffices 
to compute the Poincar\'e dual of \(u_*([\ev]) \in \RKK^\G_{+d}(Z; \C, C(X)\bigr) \cong 
\K^{-d}(Z\times_\G X)\), and show that it's intersection index is zero. As the proof is 
similar to that of Proposition \ref{proposition:description_of_inflated_class}, we merely 
sketch it. The class \(u_*([\ev]\) is represented (analytically) by the \(\G\)-equivariant 
correspondence obtained by composing 
\[ Z \xleftarrow{\ez} \G \rightarrow \pnt \leftarrow X \xrightarrow{\id} X,\]
where \(\ez\) is the orbit map at \(z_0\). 

The composition gives 
\[ Z\xleftarrow{\ez\circ \pr_\G} \G\times X \xrightarrow{\pr_X} X\]
from \(Z\) to \(X\). Poincar\'e dualizing as in the proof of Proposition 
\ref{proposition:description_of_inflated_class} and taking \(\G\)-invariants 
gives the smooth correspondence 
\[ \pnt \leftarrow X  \xrightarrow{\ezx} Z\times_\G X.\]
Finally, replacing the point \(\G z_0\) to any different point 
\(\G z_1\), we obtain the equivalent (because the two points can be connected by a path) 
correspondence 
\[ \pnt \leftarrow X \xrightarrow{ \mathrm{e}_{z_1, X}} Z\times_\G X\]
and the map \(e_{z_1, X}\) now has disjoint image from the image of 
\(\ezx\). Hence the intersection index is zero as claimed.

\end{proof}

\subsection{Boundary actions of hyperbolic groups}

\begin{corollary}
Let \(\G\) be the fundamental group of a negatively curved odd-dimensional 
compact 
\(d\)-dimensional 
manifold \(M\) with universal cover \(Z\) and Gromov boundary \(\partial Z\cong \partial \G\). 
Then the boundary extension class
\[ [\partial_\G]\in \K^1(C(\partial \G)\rtimes \G)\] is a non-torsion, nonzero class in 
\(\K\)-homology, and the Dirac index of the \(d\)-dimensional geometric 
cocycle of Example \ref{example:geometric_intersection_index_vector_field} 
determined by a non-vanishing vector field on \(M\), 
is \(+1\).

\end{corollary}

\begin{remark}
We have implicitly \(\K\)-oriented \(\xi \colon M \to SM\) by the separate 
\(\K\)-orientations on \(M\) and on \(SM\) given to us. One could also 
\(\K\)-orient it differently, by switching the \(\K\)-orientation on \(M\). 
This would result in an intersection index of \(-1\).

In the case of isometry groups of classical (say, real), hyperbolic space, where
\(M = \G\backslash Z\) where \(Z\) is classical hyperbolic space, the sphere bundle 
 \(SM\) identifies 
with \(G/\G\), where \(G\) is the full group of orientation-preserving isometries. If now 
 \(\G'\subset \G\) is a 
lattice of finite index, and \(\xi\colon M'\to SM' \cong G/\G'\) is a non-vanishing 
vector field, then 
\[ \pnt \leftarrow M'\xrightarrow{ \xi} G/\G'\to G/\G\]
also gives a geometric \(d\)-cocycle with intersection index \([\G;\G']\), as the 
reader can easily check. 
\end{remark}

Of course the existence of a non-vanishing vector field \(\xi\) in order to make a 
positive index is equivalent to vanishing of the Euler characteristic \(\chi(M)\). 
For surface groups \(\G = \pi_1(M^g)\), this is not the case, and in fact the 
boundary extension class, that is, the Dirac class, is torsion of order \(\chi(M^g) = 2-2g\). 

There is a very natural, geometric way of studying torsion in geometric 
\(\K\)-homology, however, and this gives some information on, for example, the 
torsion degree of, for instance, the boundary extension class. We give a 
brief description of this now.

Let \(k\) be any positive integer. 
The Bockstein sequence for \(\K\)-homology with \(\Z/k\)-coefficients is 
\begin{equation}
\label{equation:bockstein}
\cdots \K_*(SM) \xrightarrow{k\cdot}\K_*(SM) \to \K_*(SM)_{\Z/k} \xrightarrow{\delta} \K_{*-1} (SM) \to \cdots 
\end{equation}
where \(K_*(SM)_{\Z/k}\) is the \(\K\)-homology of \(SM\) with \(\Z/k\)-coefficients. 
The Bockstein sequence describes \(k\)-torsion in \(\K\)-homology.

R. Deeley in \cite{Dee1}\cite{ Dee2} describes this group using a modification of the 
usual Baum-Douglas 
style, using (\(\K\)-oriented) \(\Z/k\)-manifolds in the style of Sullivan, which map to 
\(SM\), and a vector bundle datum. With this description, the Bockstein connecting 
map in \eqref{equation:bockstein} 
has a simple geometric description in terms of
 Baum-Douglas cycles: if 
\(f\colon W \to SM\) is a map from a \(\K\)-oriented 
\(\Z/k\)-manifold \(W\) 
to \(SM\), then \(f\) restricts to a map \(\partial W \to SM\), giving an ordinary Baum-Douglas
cycle for \(SM\) and corresponding \(k\)-torsion class (since \(k\)-copies of it is, 
manifestly, a boundary). 
One deals with a bundle datum by restriction as well. 

Let now \(i \colon S^{d-1}\to SM\) be the inclusion of a fibre in \(\pi \colon SM \to M\), 
with class \(i_*([S^{d-1}])\in \K_{1-d}(SM)\). 
Choose a smooth vector field \(\xi \colon M \to TM\) transverse to the zero section.
The vector field has zeros \(p_1, p_2, \ldots\) and an index \(\pm 1\) at each of these 
zeros, and the sum of the indices is equal to \(k := \chi (M)\).

The Riemannian exponential map identifies the fibres \(S_{p_j}(M)\) 
with the boundaries \(\partial B_j\) of Riemannian balls around \(p_i\). Remove these
open balls from \(M\). This results in a manifold-with-boundary \(W'\). Using a 
parallel transport argument, for example, one can produce diffeomorphisms  
\[f_{ij} \colon \partial B_i \to \partial B_j\]
which are orientation-reversing if \(\xi\) has the same index at \(p_i\) and 
\(p_j\), and are orientation-preserving otherwise. 

Now identify any two \(\partial B_i\)'s, using the relevant diffeomorphism, if the 
vector field has the same index at each of their centres.  This results in a new manifold-with-boundary \(W\), whose 
remaining boundary components come with a collection of 
orientation-preserving diffeomorphisms between them. We thus have a 
\(\Z/k\)-manifold, where \(k\) is the Euler characteristic.
 Finally, the vector field \(\xi\) is non-vanishing 
on the complement of \(\cup_i B_i \), can be arranged compatible with the diffeomorphisms 
\(f_{ij}\), and hence determines a map \(\xi'\colon W\to SM\). 

Thus we obtain a cycle \( (W, \xi')\) for 
\(\K_0(M)_{\Z/k}\).  
\begin{theorem}
The Bockstein map \(\delta \colon \K_0(SM)_{\Z/k} \to \K_1(SM)\) maps 
the class of \(\xi'\colon W \to SM\) to the class \(i_*([S^{d-1}])\in \K_1(SM)\). 

In particular, the fibre class \(i_*([S^{d-1}]\), and hence the 
boundary extension class \([\partial_\G] \in \K^1(C(\partial \G)\rtimes \G)\), is \(\chi (M)\)-torsion in 
\(\K^1(C(\partial \G)\rtimes \G)\). 

\end{theorem}

\bigskip
\bigskip

We conclude with the proof of the intersection index formula, which is  
essentially a formal consequence of functoriality results in equivariant 
\(\KK\)-theory.

\begin{theorem}
\label{theorem:blackbox_index}
If \( \pnt \leftarrow  (M, \xi) \xrightarrow{f} Z\times_\G X\) is a smooth correspondence 
from a point to \(Z\times_\G X\), \(\Index (f!, \xi) \in \KK_*\bigl(\C, C_0(Z\times_\G X)\bigr)\) its
class, \([\ev] \otimes_\C [X]\in \KK_{-d} (Z\times_\G X, \C)\) the \(\K\)-homology 
class of a fibre of \(Z\times_\G X\to \G\backslash X\)  
then 
\[ \langle \hat{\mu} \bigl( \Index (f!, \xi)\bigr), [\widehat{\G\ltimes X}]\rangle
= \langle \Index (f!, \xi) , [\ev]\otimes_\C [X] \rangle,
\]
with in both cases the pairings between between Kasparov \(\K\)-theory and \(\K\)-homology. 

\end{theorem}

\begin{proof}
The work of Meyer and Nest, part of which was extended in 
\cite{EM:Dualities},  on formalizing and abstracting the Dirac method, 
show that in a rather more general context,
 the Baum-Connes assembly map 
 \begin{equation}
 \label{equation:lesson_in_bc_1}
 \KK^\G_*( C_0(\EG), B) \xrightarrow{\mu} \KK_*(\C, B\rtimes \G)
 \end{equation}
agrees with the following map, supposing that one has a suitable 
dual (see \cite{EM:Dualities} Theorem 6.9 and environmental discussion). The 
dual involves various data, including a proper \(\G\)-C*-algebra \(\mathcal{P}\) which 
in the present case of interest is \(C_0(Z)\). Duality identifies the domain of 
\ref{equation:lesson_in_bc_1} with the group 
\[ \mathcal{R}\KK^\G_{*+d} \bigl(\EG; C_0(\EG), \mathcal{P}\otimes C_0(X)\bigr).\]
The generalized Green-Julg Theorem identifies this group with 
\[ \KK_{*+d}(\C, \mathcal{P}\otimes C_0(X) \; \rtimes \G).\]
Hence assembly is equivalent to a map 
\begin{equation}
\label{equation:lesson_in_bc_3}
 \KK(\C, \mathcal{P}\otimes C_0(X)\; \rtimes \G) \to \KK_{*-d}(\C, C_0(X)\rtimes \G).
 \end{equation}
 The map in question is induced by Kasparov product with with the Dirac 
 morphism \(D\in \KK_{-d}^\G(\mathcal{P}, \C)\). 
 
Translating this into the present context, where \(\mathcal{P} = C_0(Z)\), 
\(D  = [Z]\in \KK_{-d}^\G(C_0(Z), \C)\), gives the following Lemma, from which the 
Theorem follows immediately from putting \(\psi := [\widehat{\G\ltimes X}]\) to be
the Dirac class in equivariant theory. 

\begin{lemma}
Let \(\varphi \in \KK^\G_*\bigl(C_0(Z), C_0(X)\bigr)\) and 
\( \check{\varphi} \in \K^{-*-d}(Z\times_\G X) \) it's Poincar\'e dual. 
Let \(\psi \in \KK^\G_{-*}(C_0(X), \C)\) and \(\Dirac (\psi) \in 
\KK_{-*-d}( C_0(Z\times_\G X), \C) = \K_{*+d} (Z\times_\G X)\) it's image under the 
localization map. Then 
\begin{equation}
 \label{equation:lesson_in_bc_4}
 \langle \mu (\varphi), \psi\rangle  = \langle \check{\varphi} , \Dirac (\psi)\rangle \in \Z,
\end{equation}
where the pairing is that between \(\K\)-theory and \(\K\)-homology of \(Z\times_\G X\). 

\end{lemma}

This concludes the proof of the intersection index formula.

\end{proof}

Finally, we note that the more general statement of the Intersection 
Index Formula specializes to a topological formula for computing the boundary 
map of the boundary extension of a hyperbolic group. We record it here, by 
way of conclusion.

\begin{corollary}
The boundary map 
\[ \delta \colon \K_1(C(\partial Z)\rtimes \G) \to \Z\]
associated to the boundary extension admits the following topological 
description. Let
\[ \hat{\mu}\colon \K^{-*} (SM) \to \K_{*-d} (C(\partial Z)\rtimes \G)\]
be the Baum-Connes assembly map. Then if \( \pnt  \leftarrow  (W, \xi) \xrightarrow{f}  SM\) 
is a Baum-Douglas cocycle for \(SM\), with class \(\Index (f!, \xi)\), and 
\(f\colon W \to SM\) 
transverse to the fibre \(S_{z_0}M\), then 
\[(\delta \circ  \hat{\mu} )\bigl( \Index (f!, \xi)\bigr) =
 \Index ( [f^{-1}(S_{z_0} M)]\cdot \xi)\] -- the index of the Dirac operator on 
 the spin\(^c\)-manifold \(f^{-1}(S_{z_0}M)\), twisted by \(\xi\). 

\end{corollary}



\begin{thebibliography}{00}
\bibitem{Ada} S. Adams: \emph{Boundary amenability for word hyperbolic groups and an application to smooth dynamics of simple groups}, Topology 33 (1994), no. 4, 765--783

\bibitem{CE76} M.-D. Choi, E.G. Effros: \emph{The completely positive lifting problem for $C\sp*$-algebras}, Ann. of Math. (2) 104 (1976), no. 3, 585--609


\bibitem{BCH} P. Baum, A. Connes, N. Higson: \emph{Classifying space for proper 
actions and \(K\)\nobreakdash-theory of group \(C^*\)\nobreakdash-algebras}, 
from \emph{\(C^*\)\nobreakdash-Algebras: 1943--1993}, San Antonio, TX, Contemp. Math.167
(1994), Amer. Math. Soc., Providence, RI, 240--291. 

\bibitem{Block} O. Ben-Bassat, J. Block, T. Pantev: \emph{Noncommutative tori and Fourier-Mukai duality}, Compositio Math. 143 (2007) 423?475
doi:10.1112/S0010437X06002636

\bibitem{Connes:fundamental} A. Connes: \emph{Cyclic cohomology and the transverse fundamental class of a foliation}, Geometric methods in operator algebras (Kyoto, 1983), Pitman Res. Notes Math. Ser. 123 (1986), pp. 52--144 

\bibitem{Con89} A. Connes: \emph{Compact metric spaces, Fredholm modules, and hyperfiniteness}, Ergodic Theory Dynam. Systems 9 (1989),  no. 2, 207--220

\bibitem{Connes} A. Connes: \emph{Noncommutative Geometry}, Academic Press 1994


\bibitem{Coo} M. Coornaert: \emph{Mesures de Patterson-Sullivan sur le bord d'un espace hyperbolique au sens de Gromov}, Pacific J. Math. 159 (1993),  no. 2, 241--270


\bibitem{CGM} A. Connes, M. Gromov, H. Moscovici: \emph{Group cohomology with 
Lipschitz control and higher signatures}. Geom. Funct. Anal. 3 (1993) no. 1, 1--78.


\bibitem{Connes-Skandalis} A. Connes, G. Skandalis: \emph{The longitudinal index theorem for foliations}, Publ. Res. Inst. Math. Sci. 20 (1984), no. 6, 1139--1183



\bibitem{CDP} M. Coornaert, T. Delzant, A. Papadopoulos: \emph{G\'eom\'etrie et th\'eorie des groupes. Les groupes hyperboliques de Gromov}, Lecture Notes in Mathematics 1441, Springer 1990

\bibitem{Dee1} R. Deeley: \emph{Geometric K-homology with coefficients I} , J. K-theory 9  (2012) no. 3, 537--564.


\bibitem{Dee2} R. Deeley: \emph{Geometric K-homology with coefficients II}, J. K-theory 12 (2013), no. 2, 235--256.

\bibitem{EH} H. Emerson, D. Hudson: \emph{K-theory and the Fourier-Mukai transform and K-theory}. Preprint. 


\bibitem{Emerson} H. Emerson: \emph{Noncommutative Poincar\'e duality for boundary actions of hyperbolic groups}, J. Reine Angew. Math. 564 (2003), 1--33 


\bibitem{EN} H. Emerson, B. Nica: \emph{K-homological finiteness and hyperbolic groups}, 
J. Reine Angew. Math., appeared online J. 2016-04-07 | DOI: https://doi.org/10.1515/crelle-2015-0115.

\bibitem{EM:Coarse} H. Emerson, R. Meyer: \emph{Dualizing the coarse assembly map}, J. Inst.Math. Jussieu 5 (2006), no. 2, 161-186.

\bibitem{EM:Euler} H. Emerson, R. Meyer: \emph{Euler characteristics and Gysin sequences for group actions on boundaries}, Math. Ann. 334 (2006), no. 4, 853--904

\bibitem{EM:K} H. Emerson, R. Meyer: \emph{Equivariant representable K-theory}, J. Topology.  (2009), no. 2 123--156. 

\bibitem{EM:Embeddings}H. Emerson, R. Meyer: \emph{Equivariant embedding theorems and topological index maps}, Adv. Math. 225 (2010), 2840--2882

\bibitem{EM:Geometric_KK} H. Emerson, R. Meyer: \emph{Bivariant K-theory via correspondences}, Adv. Math. 225 (2010), 2883--2919

\bibitem{EM:Dualities} H. Emerson, R. Meyer: \emph{Dualities in equivariant Kasparov theory}, New York J. Math.16 (2010), 245--313

\bibitem{EM:dual_Dirac} H. Emerson, R. Meyer: \emph{A descent principle for the dual-Dirac method}, Topology 46, no. 2 (2007), 185-209. 

\bibitem{GdH} \'E. Ghys, P. de la Harpe (eds.), \emph{Sur les groupes hyperboliques d'apr\`es Mikhael Gromov}, Progress in Mathematics 83, Birkh\"auser (1990)

\bibitem{Gro} M. Gromov: \emph{Hyperbolic groups}, in \emph{Essays in group theory} (Publ. MSRI 8, Springer 1987), 75--263

\bibitem{Haw} A. Hawkins, A. Skalski, S. White and J. Zacharias, \emph{On spectral triples on crossed products arising from equicontinuous actions}, Math. Scand. 113 (2013), 262-291.  


\bibitem{BHS}  P. Baum,, N. Higson, T. Schick: \emph{A geometric description of equivariant K-homology for proper actions}, Quanta of maths, 1?22, Clay Math. Proc., 11, Amer. Math. Soc., Providence, RI (2010), 1-22. 


\bibitem{Higson-Kasparov}: N. Higson, G. G. Kasparov: \emph{$E$\nobreakdash-theory and $KK$-theory for groups which act properly and isometrically on Hilbert space}, Invent. Math. 144 (2001), no. 1, 23--74

\bibitem{HR} N. Higson, J. Roe: \emph{Analytic K-homology}, Oxford Science Publications (2000). 

\bibitem{HR2} N. Higson, J. Roe: \emph{Amenable group actions and the Novikov Conjecture}, J. reine. angew. Math. 519 (2000), 143-153. 



\bibitem{HS} M. Hilsum, G. Skandalis: \emph{Morphisms K-orient\'es d'espaces de feuilles
et fonctorialit\'e en th'eorie de Kasparov}, Ann. scient. \'Ec. Norm. Sup., 4\(^{e}\) s\'erie, t. 20 (1987), 325--390. 

\bibitem{KH} S. Hurder, A. Katok: \emph{Differentiability, rigidity and Godbillon-Vey classes for Anasov Flows}, Publ. IHES 72 (1990). 

\bibitem{Kas} G.G.  Kasparov: \emph{Equivariant KK-theory and the Novikov Conjecture}, Invent. Math. 91 (1988), no. 1, 147-201. 


\bibitem{KM} J. Kaminker, J. Miller: \emph{ A comment on the Novikov conjecture}, Proc. Amer. Math. Soc. 83 (1981), 656--658. 


\bibitem{KB} I. Kapovich, N. Benakli: \emph{Boundaries of hyperbolic groups}, in \emph{Combinatorial and geometric group theory (New York, 2000/Hoboken, NJ, 2001)}, 39--93, Contemp. Math. 296, Amer. Math. Soc. 2002


\bibitem{Land} M. Land: \emph{The analytic assembly map and index theory}. Preprint. 

\bibitem{LeGall} P-Y. Le Gall: \emph{Th\'eorie de Kasparov \'equivariante et groupoides I}, K-theory 16 (1997), no 1, 3--21.

\bibitem{MR} V. Mathai, J. Rosenberg: \emph{T-duality for torus bundles with H-fluxes via noncommutative topology}, 
Commun. Math. Phys. 253 (2005) no. 3, 705?721.

\bibitem{Meyer-Nest:BC} Meyer, Ralf; Nest, Ryszard: \emph{The Baum-Connes conjecture via localisation of categories}, Topology 45 (2006), no. 2, 209--259. 


\bibitem{Mes} B.  Mesland: \emph{Unbounded bivariant K-theory and correspondences in noncommutative geometry}, J. Reine Angew. Math. 691 (2014), 101?172.


\bibitem{STY} G. Skandalis, J.L. Tu, G. Yu: \emph{Coarse Baum-Connes conjecture and Groupoids},  Topology 41 (2002) 807-834.



\bibitem{Tu} J.-L. Tu: \emph{La conjecture de Baum-Connes pour les feuilletages moyennables}, K-Theory 17 (1999), no. 3, 215--264




\bibitem{Sul} D. Sullivan: \emph{The density at infinity of a discrete group of hyperbolic motions}, Publ. IHES 50 (1979), 171--202.

\end{thebibliography}
\end{document}